\documentclass[
preprint, 3p 
]{elsarticle}
\pdfoutput=1
\journal{Journal of Computational Physics}
\usepackage[utf8]{luainputenc}
\usepackage[english]{babel}

\usepackage[plainpages=false,pdfpagelabels]{hyperref}
\hypersetup{pdfauthor={Hendrik Ranocha, Philipp Öffner, Thomas Sonar}}
\hypersetup{pdftitle={Summation-by-parts operators for correction procedure via reconstruction}}
\hypersetup{pdfsubject={Summation-by-parts operators for correction procedure via reconstruction}}
\hypersetup{pdfkeywords={hyperbolic conservation laws, high order methods,
summation-by-parts, SBP, flux reconstruction, FR, lifting collocation penalty,
LCP, correction procedure via reconstruction, CPR}}

\usepackage{siunitx}
\usepackage{amsmath}
\usepackage{amssymb}
\usepackage{commath}
\usepackage{mathtools}

\usepackage{amsthm}
\theoremstyle{plain}
  \newtheorem{thm}{Theorem}
  \newtheorem{lem}[thm]{Lemma}

\usepackage[small]{caption}
\usepackage{subcaption}
\usepackage{color}

\usepackage{booktabs}
\usepackage{rotating}

\renewcommand{\vec}[1]{\underline{#1}}
\newcommand{\mat}[1]{\underline{\underline{#1}}\,}
\DeclarePairedDelimiter{\diagfences}{(}{)}
\newcommand{\diag}{\operatorname{diag}\diagfences}

\begin{document}
 
\renewcommand{\today}{January 30, 2016}

\begin{frontmatter}
\title{Summation-by-parts operators for correction procedure via reconstruction}

\author[myu]{Hendrik Ranocha\corref{cor1}}
\ead{h.ranocha@tu-bs.de}

\author[myu]{Philipp \"Offner}
\author[myu]{Thomas Sonar}

\cortext[cor1]{Corresponding author}

\address[myu]{Technische Universität Braunschweig, Institut Computational Mathematics, Pockelsstra\ss e 14, D-38106 Braunschweig, Germany }

\begin{abstract}
  The correction procedure via reconstruction (CPR, formerly known as flux
  reconstruction) is a framework of high order methods for conservation laws,
  unifying some discontinuous Galerkin, spectral difference and spectral volume
  methods. Linearly stable schemes were presented by Vincent et al. (2011, 2015),
  but proofs of non-linear (entropy) stability in this framework have not been
  published yet (to the knowledge of the authors).
  We reformulate CPR methods using summation-by-parts (SBP) operators with
  simultaneous approximation terms (SATs), a framework popular for finite
  difference methods, extending the results obtained by Gassner (2013) for
  a special discontinuous Galerkin spectral element method. This reformulation
  leads to proofs of conservation and stability in discrete norms associated
  with the method, recovering the linearly stable CPR schemes of Vincent et al.
  (2011, 2015). Additionally, extending the skew-symmetric formulation of
  conservation laws by additional correction terms, entropy stability for
  Burgers' equation is proved for general SBP CPR methods not including
  boundary nodes.
\end{abstract}


\end{frontmatter}

\section{Introduction}

In the field of computational fluid dynamics (CFD), low-order methods are generally
robust and reliable and therefore employed in practical calculations. 
The main advantage of high-order methods towards low-order ones is the possibility
of considerably more accurate solutions with the same computing cost, but
unfortunately they are less robust and more complicated. In recent years many
researchers focus on this topic. There has been a surge of research activities
to improve and refine high-order methods as well as to develop new ones with more
favourable properties.

We consider in this paper the \emph{correction procedure via reconstruction}
(CPR) method using \emph{summation-by-parts }(SBP) operators. The CPR combines  
the \emph{flux reconstruction} (FR) approach  developed by
Huynh \cite{huynh2007flux} and the \emph{lifting collocation penalty} (LCP) by
Wang and Gao \cite{wang2009unifying}.

Huynh \cite{huynh2007flux} introduced the FR approach to high-order spectral
methods for conservation laws in one space dimension and its extension to
multiple dimensions using tensor products in 2007. For the case of one spatial
dimension, the ansatz amounts to evaluating the derivative of a discontinuous
piecewise polynomial function by using its straightforward derivative estimate
together with a correction term. Wang and Gao \cite{wang2009unifying}
generalised the FR approach in 2009 to lifting collocation penalty  methods on
triangular grids. Later, the authors involved in the construction of these
methods combined the names in the unifying framework of
\emph{correction procedure via reconstruction} (CPR) methods, see
\cite{huynh2014high}. The CPR creates a framework unifying several high-order
methods such as \emph{discontinuous Galerkin} (DG), \emph{spectral difference}
(SD) and \emph{spectral volume} (SV) methods. These connections were already
pointed out and investigated in more detail in
\cite{allaneau2011connections, degrazia2014connections, yu2013connection}.

A linear (von Neumann) stability analysis of FR schemes was carried
out already by Huynh \cite{huynh2007flux} and in extended form by Vincent,
Castonguay and Jameson \cite{vincent2011insights}. A one-parameter family of
linearly stable schemes in one dimension was discovered by the same authors
\cite{vincent2011newclass} using an energy method and extended in 2015 to
multiple-parameter families \cite{vincent2015extended}. Extensions of the
one-parameter family to advection-diffusion problems and triangular grids were
published by the same groups \cite{castonguay2012newclass, castonguay2013energy, 
williams2013energy}.

The analysis of nonlinear stability for CPR methods is far more complex and
not as advanced as in the linear case. First results  are available in
\cite{jameson2012nonlinear, witherden2014analysis, witherden2015identification}. 

The application of \emph{summation-by-parts} (SBP)  operators in the CPR framework
supplies a new perspective here.

In the context of \emph{finite difference} (FD) methods, summation-by-parts
operators with \emph{simultaneous approximation terms} (SATs) provide
a suitable way to derive stable schemes in a multi-block fashion enforcing
boundary conditions in a weak way. They enable the imitation of manipulations
of the continuous problem for the discrete method and are thus able to
translate results like well-posedness to its discrete counterpart stability.
Review articles from Svärd and Nordström \cite{svard2014review} and Del Rey
Fernández, Hicken, and Zingg \cite{fernandez2014review} provide an insight
in the development over the last decades and recent results.
There is a strong connection of SBP operators with both skew-symmetric
formulations of conservation laws as a means to prove conservation and stability
\cite{fisher2013discretely}, and quadrature rules \cite{hicken2013summation}.
Recently, Gassner et al. applied the SBP SAT framework to a particular
\emph{discontinuous Galerkin spectral element method} (DGSEM) to prove
stability and discrete conservation for different systems of conservation laws,
see inter alia \cite{gassner2013skew, gassner2014kinetic, kopriva2014energy,
gassner2016well}.
Another extension of SBP operators has been presented by Del Rey
Fernández et al. \cite{fernandez2014generalized}, based on a numerical setting
and allowing general operators, connected with quadrature rules.

Here, we use the SBP framework in the general CPR setting. We are able to
demonstrate all well-known properties, which have already been proven,
but  we can further extend the CPR method and show conservation and stability
in a nonlinear case, see section \ref{Chapter_4}.

The paper is organised as follows. The SBP and CPR frameworks will be briefly
explained in section \ref{Chapter_2}. In the next section, we apply SBP
operators in CPR methods and revisit the results of Vincent et al.
\cite{vincent2011newclass, vincent2015extended} for constant velocity linear
advection.

In section \ref{Chapter_4}, we focus on Burgers' equation and prove both discrete
conservation and stability for a skew-symmetric formulation and Lobatto-Legendre
nodes, revisiting the results of \cite{gassner2013skew}.

Additionally, we suggest a generalisation of the CPR method to get stability for
a general SBP basis, extending the skew-symmetric formulation and being both
provably stable and conservative. Numerical test cases are used to confirm
the theoretical results. Finally, we discuss open problems and give an
outlook on future work.

\section{Existing formulations for SBP operators and CPR methods}\label{Chapter_2}

Both \emph{finite difference} (FD) SBP methods and CPR schemes are designed as
semidiscretisations of hyperbolic conservation laws
\begin{equation}
  \partial_t u + \partial_x f(u) = 0,
\label{eq:conservation-law}
\end{equation}
equipped with appropriate initial and boundary conditions.

\subsection{SBP schemes}

Traditionally, SBP operators are used in the FD framework. In one space
dimension, a set of nodes including both boundary points of the element are
used to represent the solution values. Extensions to multiple dimensions are
performed via tensor products. To compute the semidiscretisation of
\eqref{eq:conservation-law}, $f(u)$ is evaluated at each node and a difference
operator is applied. The notation using vectors $\vec{u}$ for the solution
values and the differentiation matrix $\mat{D}$ is very common and results in
a finite difference approximation $\mat{D} \vec{f}$ of $\partial_x f$.

In order to be an SBP operator, the derivative matrix needs to be written as
$D = P^{-1} Q$, $Q + Q^T = B = \diag{-1, 0, \dots, 0, 1}$, where $P$ is a
symmetric and positive definite matrix with associated norm
$\norm{\vec{u}}^2_P = \vec{u}^T \mat{P} \vec{u}$, approximating the $L^2$ norm,
see inter alia the review \cite{svard2014review} and references cited therein.
Boundary (both of the computational domain and between blocks) conditions
are imposed weakly, using a \emph{simultaneous-approximation-term} (SAT)
formulation (see inter alia \cite{fisher2013discretely}),
involving differences of desired and given values at boundary points.
Thus, the SBP CPR methods described in the next chapter extend these schemes.

\subsection{CPR methods}

The FR approach in one space dimension described by Huynh \cite{huynh2007flux}
uses a nodal polynomial basis of order $p$ in the standard element $[-1, 1]$.
All elements are mapped to this standard element and the computations are
performed there. Extensions to multiple dimensions are performed via tensor
products. The semidiscretisation of \eqref{eq:conservation-law} (i.e.
the computation of $\partial_x f(u)$) consists of the following steps, see also
the review \cite{huynh2014high} and references cited therein:
\begin{itemize}
  \item
    Interpolate the solution to the cell boundaries at $-1$ and $1$ (if these
    values are not already given as coefficients of the nodal basis).
  
  \item
    Compute common numerical fluxes $f^{num}$ at each cell boundary.
  
  \item
    Compute the flux $f(u)$ pointwise in each node.
    
  \item
    Interpolate the flux $f(u)$ to the boundary and add polynomial correction
    functions $g_L$, $g_R$ of degree $p+1$, multiplied by the difference of
    $f_{L/R} - f^{num}_{L/R}$ of the flux and the numerical flux at the
    corresponding boundary.
    
  \item
    Finally, compute the resulting derivative of
    $f + (f_{L} - f^{num}_{L}) g_L + (f_{R} - f^{num}_{R}) g_R$,
    using exact differentiation for the polynomial basis.
\end{itemize}

Wang and Gao \cite[equation (3.14)]{wang2009unifying} formulated the
\emph{lifting collocation penalty} (LCP) approach for the semidiscretisation
of \eqref{eq:conservation-law} on triangles as the exact derivative of the
flux $f(u)$ computed as above by pointwise evaluation at the nodes plus
additional correction terms in the form of a linear combination of the
differences between a common numerical flux and the flux $f$ at points on the
boundary of the cell.

Since the approaches are similar and can be reformulated in the other way,
the common name \emph{correction procedure via reconstruction} (CPR) was
used for both schemes. In the next section, the involved linear combination of
penalty terms at the boundary is rewritten in another way, allowing more
abstractions and compact formulations.

\section{CPR methods using SBP operators: Linear advection }\label{Chapter_3}

This chapter focuses on a new formulation of CPR methods with special
attention paid to SBP operators. Additionally, constant velocity linear
advection is used as a test case to investigate linear stability and
conservation properties of the schemes.

\subsection{The one dimensional setting}

After mapping each element to the standard element $[-1,1]$,
a CPR method can be formulated as
\begin{equation}
  \partial_t \vec{u} + \mat{D} \vec{f} + \mat{C} ( \vec{f}^{num} - \mat{R} \vec{f}) = 0.
\label{eq:CPR}
\end{equation}
Here, $\vec{u}, \vec{f}$ are the finite dimensional representation of $u$,
$f(u)$ in the standard element and $\vec{f}^{num}$ is the representation of
the numerical flux on the boundary. The linear operators representing differentiation
and restriction (interpolation) to the boundary of the standard element are
represented via the matrices $\mat{D}$ and $\mat{R}$, respectively. Other parameters
of the correction operator are encoded in the correction matrix $\mat{C}$. Thus,
for a given standard element, a CPR method is parameterized by
\begin{itemize}
  \item 
  A basis $\mathcal{B}$ for the local expansion, determining the derivative and
  restriction (interpolation) matrices $\mat{D}$ and $\mat{R}$.
  
  \item
  A correction matrix $\mat{C}$, adapted to the chosen basis.
\end{itemize}

For the representation of an SBP operator, the basis $\mathcal{B}$ has to be
associated with a (volume) quadrature rule, given by nodes $z_0, \dots, z_p$ and
appropriate positive weights $\omega_0, \dots, \omega_p$. The values of $u$ at
the nodes are the coefficients of the local expansion, i.e.
$\vec{u} = (u(z_0), \dots, u(z_p))^T$. The quadrature weights determine a
positive definite Matrix $\mat{M} = \diag{\omega_0, \dots, \omega_p}$ associated
with a (discrete) norm $\norm{u}_M^2 = \vec{u}^T \mat{M} \vec{u}$.
Besides the volume quadrature rule, there must be a quadrature rule for the
boundary, approximating the outward flux through the boundary as in the
divergence theorem. In the present one dimensional setting, this quadrature rule
is simply given by exact evaluation at the endpoints $\eval[0]{\cdot}_{-1}^{1}$.
The basis and it's associated quadrature rules must satisfy the SBP property
\begin{equation}
  \mat{M} \mat{D} + \mat{D}^T \mat{M} = \mat{R}^T \mat{B} \mat{R},
\label{eq:SBP-MRBD}
\end{equation}
in order to mimic integration by parts on a discrete level
\begin{equation}
\begin{aligned}
  & \vec{u}^T \mat{M} \mat{D} \vec{v} + (\mat{D} \vec{u})^T \mat{M} \vec{v}
\\& \approx \int_{-1}^{1} u \, \partial_x v \dif x + \int_{-1}^{1} \partial_x u \, v \dif x
    =
    \eval[2]{u \, v}_{-1}^{1}
\\& \approx (\mat{R} \vec{u})^T \mat{B} (\mat{R} \vec{u}).
\end{aligned}
\label{eq:SBP-IBP-MRBD}
\end{equation}

As an example, consider Gauß-Lobatto-Legendre integration with its associated
basis of point values at Lobatto nodes in $(-1, 1)$. Then, the restriction and
boundary integral matrices reduce to
\begin{equation}
  \mat{R} = \begin{pmatrix}
               1 & 0 & \dots & 0\\
               0 & \dots & 0 & 1
            \end{pmatrix}
  ,\quad
  \mat{B} = \begin{pmatrix}
              -1 & 0\\
               0 & 1
            \end{pmatrix}.
\end{equation}
Using the special choice $\mat{C} = \mat{M}^{-1} \mat{R}^T \mat{B}$ and defining
$\mat{\widetilde{B}} := \mat{R}^T \mat{B} \mat{R}$, i.e.
$\mat{\widetilde{B}} = \diag{-1, 0, \dots, 0, 1}$,
the CPR method of equation \eqref{eq:CPR} reduces to
\begin{equation}
  \partial_t \vec{u} + \mat{D} \vec{f}
  + \mat{M}^{-1} \mat{\widetilde{B}} ( \vec{\widetilde{f}}^{num} - \vec{f}) = 0,
\label{eq:DGSEM}
\end{equation}
where $\vec{\widetilde{f}}^{num} = (f^{num}_L, 0, \dots, 0, f^{num}_R)$ contains
the numerical flux at the left and right boundary and satisfies
$\vec{f}^{num} = \mat{R} \vec{\widetilde{f}}^{num}$. Equation \eqref{eq:DGSEM}
is the strong form of the DGSEM formulation of Gassner \cite{gassner2013skew},
which he proved to be a diagonal norm SBP operator.

\subsection{Conservation}

Consider now a CPR method given by a nodal basis of polynomials of degree $\leq p$
and an associated (symmetric) quadrature rule that is exact for polynomials of
degree $\leq 2p - 1$, for example Gauß-Legendre or Gauß-Lobatto-Legendre quadrature.
Then, due to exact integration of polynomials of the form $u \, \partial_x v$,
where $u, v$ are polynomials of degree $\leq p$, the SBP property \eqref{eq:SBP-MRBD}
automatically holds, see also \cite{kopriva2010quadrature}.
Let $\vec{1}$ denote the representation of the constant function $x \mapsto 1$
in the chosen basis, i. e. $\vec{1} = (1, \dots, 1)^T$ for a nodal polynomial basis.

In order to investigate conservation properties in the continuous setting,
the function $u$ is multiplied with the constant function $x \mapsto 1$ and
integrated over the interval $(a, b)$, resulting in
\begin{equation}
  \od{}{t} \int_{a}^{b} u \dif x
  = - \int_{a}^{b} \partial_x f(u) \dif x
  = \eval[2]{f(u)}_{a}^{b}
  = f_R - f_L.
\end{equation}
Mimicking this derivation in the semidiscrete setting (in the standard element) leads to
\begin{equation}
  \od{}{t} \int_{-1}^{1} u \dif x
  = \od{}{t} \vec{1}^T \mat{M} \vec{u}
  = - \vec{1}^T \mat{M} \left( \mat{D} \vec{f} + \mat{C} ( \vec{f}^{num} - \mat{R} \vec{f}) \right).
\end{equation}
Using the SBP property \eqref{eq:SBP-MRBD} results in
\begin{equation}
  \od{}{t} \vec{1}^T \mat{M} \vec{u}
  =   \vec{1}^T \mat{D}^T \mat{M} \vec{f}
    - \vec{1}^T \mat{R}^T \mat{B} \mat{R} \vec{f}
    - \vec{1}^T \mat{M} \mat{C} ( \vec{f}^{num} - \mat{R} \vec{f}).
\end{equation}
Since discrete differentiation is exact for polynomials of degree $\leq p$
and especially for constant functions, $\mat{D} \vec{1} = 0$, we get
\begin{equation}
  \od{}{t} \vec{1}^T \mat{M} \vec{u}
  = - ( \vec{1}^T \mat{R}^T \mat{B} - \vec{1}^T \mat{M} \mat{C} ) \mat{R} \vec{f}
    - \vec{1}^T \mat{M} \mat{C} \vec{f}^{num}.
\label{eq:CPR-conservation-derivation}
\end{equation}

\begin{lem}
  If the assumptions of this subsection are complied with and the correction operator
  of the CPR method satisfies
  $\vec{1}^T \mat{M} \mat{C} = \vec{1}^T \mat{R}^T \mat{B}$,
  then the scheme is conservative (across elements).
\label{lem:CPR-conservation}
\end{lem}
\begin{proof}
  Inserting the condition into \eqref{eq:CPR-conservation-derivation} gives
  \begin{equation}
    \od{}{t} \vec{1}^T \mat{M} \vec{u}
    = - \vec{1}^T \mat{R}^T \mat{B} \vec{f}^{num}
    = f^{num}_R - f^{num}_L,
  \end{equation}
  due to exact evaluation of the boundary integral for polynomials of degree
  $\leq p$. Summing up the contributions of all elements and bearing in mind
  that the numerical flux at the boundary point between two adjacent elements
  is the same for both, biased only by a factor of $-1$ for one element but
  not for the other, results in the global equality
  \begin{equation}
    \od{}{t} \int_{a}^{b} u \dif x
    = \vec{1}^T \mat{M} \vec{u}
    = f_R - f_L
  \end{equation}
  also for the numerical scheme.
\end{proof}
Assuming periodic boundary conditions therefore leads to global conservation 
\begin{equation}
  \od{}{t} \int_{a}^{b} u \dif x = \od{}{t} \vec{1}^T \mat{M} \vec{u} = 0.
\end{equation}

Lemma \ref{lem:CPR-conservation} proofs conservation across elements.
On a sub-element level, conservation for diagonal-norm SBP operators
(including boundary nodes) has been proven in \cite{fisher2013discretely}
in the context of the Lax-Wendroff theorem.

\subsection{Linear stability}

Specializing on a certain type of flux, namely the flux of linear advection with
constant velocity $1$, the conservation law reduces to
\begin{equation}
  \partial_t u + \partial_x u = 0.
\label{eq:linear-constant-advection}
\end{equation}
In the continuous setting, proving stability with respect to the $L^2$ norm is
simply an application of integration by parts. Multiplying equation
\eqref{eq:linear-constant-advection} by the solution $u$ and integrating over
the domain leads to
\begin{equation}
  \int u \, \partial_t u
  = - \int u \, \partial_x u
  = - \eval[2]{\frac{1}{2} u^2} + \int \partial_x u \, u.
\end{equation}
Summing up the first and the last equality results in
\begin{equation}
  \od{}{t} \norm{u}_{L^2}^2
  = - \eval[2]{\frac{1}{2} u^2},
\end{equation}
allowing an estimate of the solution's norm in terms of the initial and
boundary conditions, i.e. \emph{well-posedness}.
Assuming compact support or periodic boundary conditions simplifies the estimate to
$\od{}{t} \norm{u}_{L^2}^2 = 0$.

Mimicking this manipulations in the discrete setting of an SBP CPR method reads as
\begin{equation}
  \int u \, \partial_t u
  \approx \vec{u}^T \mat{M} \od{}{t} \vec{u}
  = - \vec{u}^T \mat{M} \left( \mat{D} \vec{u} + \mat{C} ( \vec{f}^{num} - \mat{R} \vec{u}) \right).
\end{equation}
Applying the SBP property \eqref{eq:SBP-MRBD} as in the previous section results in
\begin{equation}
  \vec{u}^T \mat{M} \od{}{t} \vec{u}
  = \vec{u}^T \mat{D}^T \mat{M} \vec{u} - \vec{u}^T \mat{R}^T \mat{B} \mat{R} \vec{u}
    - \vec{u}^T \mat{M} \mat{C} ( \vec{f}^{num} - \mat{R} \vec{u}).
\end{equation}
Summing up these equations and using the symmetry of the scalar product induced
by $\mat{M}$ yields
\begin{equation}
  \od{}{t} \norm{u}_M^2
  = - \vec{u}^T \mat{R}^T \mat{B} \mat{R} \vec{u}
    - 2 \vec{u}^T \mat{M} \mat{C} ( \vec{f}^{num} - \mat{R} \vec{u}).
\end{equation}
Assuming now the special form
$\mat{C} = \mat{M}^{-1} \mat{R}^T \mat B$
simplifies the last equation to
\begin{equation}
\begin{aligned}
  \od{}{t} \norm{u}_M^2
  & = \vec{u}^T \mat{R}^T \mat{B} \mat{R} \vec{u} 
      - 2 \vec{u}^T \mat{R}^T \mat B \vec{f}^{num}
\\& = \vec{u}^T \mat{R}^T \mat{B} ( \mat{R} \vec{u} - 2 \vec{f}^{num}).
\end{aligned}
\end{equation}
Due to exact evaluation of the boundary terms for $\vec{u}$,
representing a polynomial of degree $\leq p$, this can be written as
\begin{equation}
  \od{}{t} \norm{u}_M^2
  = u_R ( u_R - 2 f^{num}_R) - u_L ( u_L - 2 f^{num}_L),
\end{equation}
where the indices $R$ and $L$ indicate values at the right and left boundary,
respectively.
Therefore, a similar estimate of the norm of the numerical solution in terms of
boundary data and the numerical flux is possible. Assuming again periodic boundary
conditions or compact support reduces the global rate of change to a sum of local
contributions of the form $u_- (u_- - 2 f^{num}) - u_+ (u_+ - 2 f^{num})$,
where $f^{num}$ is the common numerical flux and $u_-$ is the value $u_R$ on the
right boundary of the left element and $u_+$ is appropriately defined.
Using a standard numerical flux of the form
\begin{equation}
  f^{num}(u_-, u_+)
  = \frac{u_+ + u_-}{2} - \alpha (u_+ - u_-),
\label{eq:num-flux-linear-constant-advection}
\end{equation}
recovering a central scheme for $\alpha = 0$ and a fully upwind scheme for
$\alpha = 1$, yields
\begin{equation}
\begin{aligned}
  &u_- (u_- - 2 f^{num}) - u_+ (u_+ - 2 f^{num})
\\&= u_-^2 - u_+^2 - u_-(u_- + u_+) + 2 \alpha u_-(u_+ - u_-)
\\&\quad                + u_+(u_- + u_+) - 2 \alpha u_+(u_+ - u_-)
\\&= - 2 \alpha (u_+ - u_-)^2.
\end{aligned}
\end{equation}
Thus, $\alpha \geq 0$ ensures $\od{}{t} \norm{u}_M \leq 0$ and therefore
\emph{stability}, the discrete analogue to well-posedness.

The basic idea of Jameson \cite{jameson2010proof} to show linear stability is using the
equivalence of norms in finite dimensional vector spaces and showing stability
not in a regular $L^2$ norm, but a kind of Sobolev norm involving derivatives,
as also explained in \cite{allaneau2011connections} and used in
\cite{vincent2011newclass, vincent2015extended} to derive linearly stable
FR methods.
Although the ansatz here is very different, some calculations
are similar and in the end, the same schemes will be derived. The difference is,
that Vincent et al. \cite{vincent2015extended} used continuous integral norms for their
derivations whereas this setting uses fully discrete norms adapted to the solution
point coordinates. Therefore, they could not recognize any influence of the solution
points on the stability properties in the linear case. For the nonlinear case,
the influence of these nodes was stressed in \cite{jameson2012nonlinear}.

Following these ideas, stability is investigated for a discrete norm given by
$\mat{M} + \mat{K}$, where $\mat{M}$ is the matrix associated as usual with the
quadrature rule given by the polynomial basis and $\mat{K}$ is a symmetric matrix
satisfying $\mat{M} + \mat{K} > 0$, i.e. positive definite. Then, the rate of change of
the discrete norm $\norm{u}_{M+K}^2 = \vec{u}^T (\mat{M} + \mat{K}) \vec{u}$ can
be computed as
\begin{equation}
  \vec{u}^T (\mat{M} + \mat{K}) \od{}{t} \vec{u}
  = - \vec{u}^T (\mat{M} + \mat{K})
       \left( \mat{D} \vec{u} + \mat{C} ( \vec{f}^{num} - \mat{R} \vec{u}) \right),
\end{equation}
which can also be written as
\begin{equation}
\begin{aligned}
  \vec{u}^T (\mat{M} + \mat{K}) \od{}{t} \vec{u}
  = &- \vec{u}^T \mat{K} \mat{D} \vec{u}
     - \vec{u}^T (\mat{M} + \mat{K}) \mat{C} ( \vec{f}^{num} - \mat{R} \vec{u})
\\&  + \vec{u}^T \mat{D}^T \mat{M} \vec{u}
     - \vec{u}^T \mat{R}^T \mat{B} \mat{R} \vec{u},       
\end{aligned}
\end{equation}
due to the SBP property \eqref{eq:SBP-MRBD}. Again, adding the last two
equations yields
\begin{equation}
  \begin{aligned}
  & \od{}{t} \norm{u}_{M+K}^2
\\& = - 2 \vec{u}^T \mat{K} \mat{D} \vec{u}
      - 2 \vec{u}^T (\mat{M} + \mat{K}) \mat{C} ( \vec{f}^{num} - \mat{R} \vec{u})
\\&\quad - \vec{u}^T \mat{R}^T \mat{B} \mat{R} \vec{u}.  
\end{aligned}
\end{equation}
The last term contains only boundary values and is thus unproblematic. The second
term can be rendered as a boundary term by enforcing the correction matrix to
be $\mat{C} = (\mat{M} + \mat{K})^{-1} \mat{R}^T \mat{B}$, analogously to the
previous procedure. Then, the only term remaining to be estimated is the first one.

In the following sections, the multiple parameter family of FR methods of
Vincent et al. \cite{vincent2015extended} will be reconsidered using the view of SBP operators.
These parameters force the first term to vanish, because $\mat{K} \mat{D}$ is
chosen to be antisymmetric. Then, using $\mat{C} = (\mat{M} + \mat{K})^{-1} \mat{R}^T \mat{B}$,
the last equation can be written as
\begin{equation}
\begin{aligned}
  \od{}{t} \norm{u}_{M+K}^2
  & = - 2 \vec{u}^T \mat{R}^T \mat{B} ( \vec{f}^{num} - \mat{R} \vec{u})
      - \vec{u}^T \mat{R}^T \mat{B} \mat{R} \vec{u}
\\& = \vec{u}^T \mat{R}^T \mat{B} ( \mat{R} \vec{u} - 2 \vec{f}^{num}),
\end{aligned}
\end{equation}
allowing the same estimates as before, leading to linear stability. This proves
the following
\begin{lem}[see also {\cite[Thm.1 ]{vincent2015extended}}]
\label{lem:CPR-linearly-stable}
  If the SBP CPR method is given by
  $\mat{C} = (\mat{M} + \mat{K})^{-1} \mat{R}^T \mat{B}$, where
  $\mat{M} + \mat{K}$ is positive definite and $\mat{K} \mat{D}$ is antisymmetric,
  then the method is linearly stable in the discrete norm $\norm{\cdot}_{M+K}$
  induced by $\mat{M} + \mat{K}$.
\end{lem}

\subsection{Symmetry}

In order to recognize the FR method associated with an SBP CPR method, it suffices
to identify the correction matrix $\mat{C}$ with the derivatives of the left
and right correction function $g_L, g_R$. Using again a nodal polynomial basis
with symmetric nodes $\xi_0, \dots, \xi_p$ in the standard element, writing
\begin{equation}
  \mat{C} = \begin{pmatrix}
              g_L'(\xi_0)   & g_R'(\xi_0)
           \\ \vdots        & \vdots
           \\ g_L'(\xi_p)   & g_R'(\xi_p)
            \end{pmatrix}
\end{equation}
provides the required identification of SBP CPR parameters and FR correction
functions. Note that $g_L(-1) = 1 = g_R(1)$ is required, so that the integration
constant is fixed. The symmetry property $g_R(\xi) = g_L(-\xi)$ (and therefore
also $g_R'(\xi) = - g_L'(-\xi)$) should be satisfied for the correction
procedure in order not to get any bias to one direction. Translated to the CPR
method, this requires
\begin{equation}
  \mat{C} = \begin{pmatrix}
              g'(\xi_0)   & -g'(\xi_p)
           \\ \vdots        & \vdots
           \\ g'(\xi_p)   & -g'(\xi_0)
            \end{pmatrix},
\label{eq:symmetry-C}
\end{equation}
dropping the index for $g_L$ and using the symmetry of $g_L$, $g_R$, and
the nodes $\xi_0, \dots, \xi_p$.

Assume that the nodal basis is associated with a symmetric quadrature that is
exact for polynomials of degree $\leq p$. Then, a coordinate transformation to
Legendre polynomials, i.e. from a nodal basis to a modal basis, is given by
the Vandermonde matrix $\mat{V}$ with $V_{i,j} = \phi_j(\xi_i)$, where
$\phi_j, j=0, \dots, p$ are the Legendre polynomials. Writing matrices and vectors
with respect to the Legendre basis using $\hat\cdot$, the transformation is
$\mat{V} \hat{\vec{u}} = \vec{u}$. Therefore, the operator matrices like the
derivative matrix transform according to $\hat{\mat{D}} = \mat{V}^{-1} \mat{D} \mat{V}$
and the matrices associated with bilinear forms like $\mat{M}$ and $\mat{K}$ can
be computed as $\hat{\mat{M}} = \mat{V}^T \mat{M} \mat{V}$.

Because the transformation from Lagrange to Legendre polynomials does not change
the basis of the boundary, which is still a nodal basis for a quadrature
(indeed, in this one dimensional setting, it is an exact evaluation), the
modal correction matrix is $\hat{\mat{C}} = \mat{V}^{-1} \mat{C}$, i.e.
\begin{equation}
  \hat{\mat{C}}
  = \mat{V}^{-1} \mat{C}
  = \mat{V}^{-1} ( \vec{g_L'} \,,\, \vec{g_R'} )
  = ( \hat{\vec{g_L'}} \,,\, \hat{\vec{g_R'}} ).
\label{eq:C-Legendre}
\end{equation}
Because of the alternating symmetry and antisymmetry of the Legendre polynomials,
the symmetry condition \eqref{eq:symmetry-C} is translated to
\begin{equation}
  \hat{\mat{C}}
  = ( \hat{\vec{g_L'}} \,,\, \hat{\vec{g_R'}} )
  = \begin{pmatrix}
      -c_o            & c_0
   \\  c_1            & c_1
   \\ \vdots          & \vdots
   \\ (-1)^{p+1} c_p  & c_p
    \end{pmatrix},
\end{equation}
for some coefficients $c_0, \dots, c_p$. Using $\hat{\mat{C}} = \mat{V}^{-1} \mat{C}$,
this becomes
\begin{equation}
\begin{aligned}
   \begin{pmatrix}
       -c_0            & c_0
    \\  c_1            & c_1
    \\ \vdots          & \vdots
    \\ (-1)^{p+1} c_p  & c_p
   \end{pmatrix}
  &= \hat{\mat{C}}
   = \mat{V}^{-1} \mat{C}
\\&= \mat{V}^{-1} (\mat{M} + \mat{K})^{-1} \mat{R}^T \mat{B}
\\&= \mat{V}^{-1} \mat{V} (\hat{\mat{M}} + \hat{\mat{K}})^{-1} \mat{V}^{-T} \mat{V}^{T} \mat{R}^T \mat{B}
\\&= (\hat{\mat{M}} + \hat{\mat{K}})^{-1} \mat{R}^T \mat{B}.
\end{aligned}
\end{equation}
The modal restriction matrix is given by the values of the Legendre polynomials
$\phi_i$ at $-1$ and $1$, i.e. $\phi_i(-1) = (-1)^i$ and $\phi_i(1) = 1$.
Therefore, the symmetry condition reduces to
\begin{equation}
\begin{aligned}
  \begin{pmatrix}
      -c_0            & c_0
   \\  c_1            & c_1
   \\ \vdots          & \vdots
   \\ (-1)^{p+1} c_p  & c_p
  \end{pmatrix}
  &= (\hat{\mat{M}} + \hat{\mat{K}})^{-1} 
    \begin{pmatrix}
       1        & 1
    \\ -1       & 1
    \\ \vdots   & \vdots
    \\ (-1)^{p} & 1
    \end{pmatrix}
    \begin{pmatrix}
       -1 & 0
    \\ 0  & 1
    \end{pmatrix}
\\&= (\hat{\mat{M}} + \hat{\mat{K}})^{-1}
    \begin{pmatrix}
       -1         & 1
    \\  1         & 1
    \\ \vdots     & \vdots
    \\ (-1)^{p+1} & 1
    \end{pmatrix}
\end{aligned}
\label{eq:symmetry-using-Legendre}
\end{equation}
A sufficient condition for this equality in analogy to
\cite[Thm. 2]{vincent2015extended} is given by

\begin{lem}
\label{lem:CPR-symmetry}
  If for $\hat{\mat{J}} = \diag{-1, 1, \dots, (-1)^{p+1}}$ the condition
  \begin{equation}
    \hat{\mat{J}} (\hat{\mat{M}} + \hat{\mat{K}})
    = (\hat{\mat{M}} + \hat{\mat{K}}) \hat{\mat{J}}
  \end{equation}
  is satisfied, than the SBP CPR method is symmetric in the sense of equation
  \eqref{eq:symmetry-C}.
\end{lem}
\begin{proof}
  Comparing the rows of equation \eqref{eq:symmetry-using-Legendre} leads to
  the conditions
  \begin{equation}
    \begin{pmatrix}
        c_0
    \\  c_1
    \\ \vdots
    \\  c_p
    \end{pmatrix}
    = (\hat{\mat{M}} + \hat{\mat{K}})^{-1}
      \begin{pmatrix}
         1
      \\ 1
      \\ \vdots
      \\ 1
      \end{pmatrix}
  \end{equation}
  \begin{equation}
  \begin{aligned}
    \hat{\mat{J}}
    \begin{pmatrix}
        c_0
    \\  c_1
    \\ \vdots
    \\  c_p
    \end{pmatrix}
    =
    \begin{pmatrix}
       -c_0
    \\  c_1
    \\ \vdots
    \\ (-1)^{p+1} c_p
    \end{pmatrix}
    = (\hat{\mat{M}} + \hat{\mat{K}})^{-1}
      \begin{pmatrix}
         -1
      \\  1
      \\ \vdots
      \\ (-1)^{p+1}
      \end{pmatrix}
    = (\hat{\mat{M}} + \hat{\mat{K}})^{-1} \hat{\mat{J}}
      \begin{pmatrix}
         1
      \\ 1
      \\ \vdots
      \\ 1
      \end{pmatrix}.
  \end{aligned}
  \end{equation}
  The first condition determines the coefficients $c_0, \dots, c_p$ and the
  second one is automatically satisfied if $(\hat{\mat{M}} + \hat{\mat{K}})$
  and $\hat{\mat{J}}$ commute.
\end{proof}

If the quadrature is exact of order $2p - 1$, $\hat{\mat{M}}$ is still a diagonal
matrix, because the Legendre polynomials are orthogonal. The entries with index
$0$ to $p-1$ are the correct norms of the corresponding Legendre polynomials and
the last entry may be changed. For Gauß-Lobatto-Legendre quadrature, the last
entry is $\hat{\mat{M}}_{p,p} = \frac{2}{p}$, as used in \cite{gassner2011comparison}.
For Gauß-Legendre quadrature, the last entry is the correct value $\frac{2}{2p-1}$,
because the quadrature is exact for polynomials of degree $\leq 2p + 1$.
In this case a new result shows that the condition of Lemma
\ref{lem:CPR-symmetry} is automatically satisfied:

\begin{lem}
\label{lem:CPR-symmetry-satisfied}
  If the SBP CPR method is associated with a quadrature of order $2p - 1$,
  $\mat{K} \mat{D} + \mat{D}^T \mat{K}^T = 0$ and $\mat{M} + \mat{K}$ is positive
  definite ($\mat{M}$ is positive definite by definition), then the symmetry condition
  $\hat{\mat{J}} (\hat{\mat{M}} + \hat{\mat{K}}) = (\hat{\mat{M}} + \hat{\mat{K}}) \hat{\mat{J}}$
  of Lemma \ref{lem:CPR-symmetry} is satisfied.
\end{lem}
\begin{proof}
  Because the quadrature is exact for polynomials of order $\leq 2p - 1$, the
  modal mass matrix $\hat{\mat{M}}$ is diagonal and commutes with $\hat{\mat{J}}$.
  Therefore, it suffices to prove the commutativity with $\hat{\mat{K}}$.
 
  In the following part of the proof, the notation using $\hat{\cdot}$ and
  $\mat{\cdot}$ is dropped due to simplicity.
  
  Proof for $J K = K J$ by induction on $p$: For $p = 1$, the relevant matrices are
  \begin{equation}
    J = \begin{pmatrix}
          -1 & 0
        \\ 0 & 1
        \end{pmatrix}
    ,\,
    D = \begin{pmatrix}
           0 & 1
        \\ 0 & 0
        \end{pmatrix}
    ,\,
    K = \begin{pmatrix}
           k_{00} & k_{01}
        \\ k_{01} & k_{11}
        \end{pmatrix}.
  \end{equation}
  Therefore, $K D + D^T K = 0$ implies
  \begin{equation}
    \begin{pmatrix}
       0 & k_{00}
    \\ 0 & k_{01}
    \end{pmatrix}
    +
    \begin{pmatrix}
       0      & 0
    \\ k_{00} & k_{01}
    \end{pmatrix}
    = 0,
  \end{equation}
  i.e. $k_{00} = k_{01} = 0$. Using this results in $J K J = K$.
  
  $p \to p+1\colon \quad J_+ K_+ J_+ = K_+$ has to be proven using the result for
  $K, J$. The matrices are given by
  \begin{equation}
    J_+ = \begin{pmatrix}
             J  &  0
          \\ 0  &  (-1)^p
          \end{pmatrix}
    ,\,
    D_+ = \begin{pmatrix}
             D  &  d
          \\ 0  &  0
          \end{pmatrix}
    ,\,
    K_+ = \begin{pmatrix}
             K    &  k
          \\ k^T  &  \kappa
          \end{pmatrix}.
  \end{equation}
  Using induction, the equation is
  \begin{equation}
  \begin{aligned}
    J_+ K_+ J_+
    &= J_+
      \begin{pmatrix}
         K J    &  (-1)^p k
      \\ k^T J  &  (-1)^p \kappa
      \end{pmatrix}
    = \begin{pmatrix}
         J K J         &  (-1)^p J k
      \\ (-1)^p k^T J  &  (-1)^{2p} \kappa
      \end{pmatrix}
  \\&= \begin{pmatrix}
         K             &  (-1)^p J k
      \\ (-1)^p k^T J  &  \kappa
      \end{pmatrix}.
  \end{aligned}
  \end{equation}
  Thus, it remains to show $k = (-1)^p J k$ using $D_+^T K_+ + K_+ D_+ = 0$, i.e.
  \begin{equation}
    \begin{pmatrix}
         D^T K  &  D^T k
      \\ d^T K  &  d^T k
    \end{pmatrix}
    +
    \begin{pmatrix}
         K D    &  K d
      \\ k^T D  &  k^T d
    \end{pmatrix}
    =
    \begin{pmatrix}
         D^T K + K D    &  D^T k + K d
      \\ d^T K + k^T D  &  2 d^T k
    \end{pmatrix}
    = 0
  \label{eq:DtK+KD}
  \end{equation}
  $k = (-1)^p J k$ can be written as $\diag{1+(-1)^p, \dots, 2, 0, 2} \, k = 0$.
  That is, $k_p = k_{p-2} = \dots = 0$ is to be proven.
  
  Calculating the product of $J$ and $D$ results in
  \begin{equation}
  \begin{aligned}
    &J \, D \, J
    \\&= 
      \begin{pmatrix}
         1      &  0     &  0     &     0  &  \dots
      \\ 0      &  -1    &  0     &     0  &  \dots
      \\ 0      &  0     &  1     &     0  &  \dots
      \\ 0      &  0     &  0     &     -1 &  \dots
      \\ \vdots & \vdots & \vdots & \vdots & \ddots
      \end{pmatrix}
      \begin{pmatrix}
         0      &  1     &  0     &     1  &  0     &  \dots
      \\ 0      &  0     &  3     &     0  &  3     &  \dots
      \\ 0      &  0     &  0     &     5  &  0     &  \dots
      \\ 0      &  0     &  0     &     0  &  7     &  \dots
      \\ \vdots & \vdots & \vdots & \vdots & \vdots & \ddots
      \end{pmatrix}
      \begin{pmatrix}
         1      &  0     &  0     &     0  &  \dots
      \\ 0      &  -1    &  0     &     0  &  \dots
      \\ 0      &  0     &  1     &     0  &  \dots
      \\ 0      &  0     &  0     &     -1 &  \dots
      \\ \vdots & \vdots & \vdots & \vdots & \ddots
      \end{pmatrix}
  \\& =
      \begin{pmatrix}
         0      &  1     &  0     &     1  &  0     &  \dots
      \\ 0      &  0     &  -3    &     0  &  -3    &  \dots
      \\ 0      &  0     &  0     &     5  &  0     &  \dots
      \\ 0      &  0     &  0     &     0  &  -7    &  \dots
      \\ \vdots & \vdots & \vdots & \vdots & \vdots & \ddots
      \end{pmatrix}
      \begin{pmatrix}
         1      &  0     &  0     &     0  &  \dots
      \\ 0      &  -1    &  0     &     0  &  \dots
      \\ 0      &  0     &  1     &     0  &  \dots
      \\ 0      &  0     &  0     &     -1 &  \dots
      \\ \vdots & \vdots & \vdots & \vdots & \ddots
      \end{pmatrix}
      = -D.
  \end{aligned}
  \end{equation}
  Using $J^2 = I$, $J^T = J$, and $D^T J = - J D^T$, yields
  \begin{equation}
    D^T (I - (-1)^p J) k
    = D^T k - (-1)^p D^T J k
    = D^T k + (-1)^p J D^T k.
  \end{equation}
  Using equation \eqref{eq:DtK+KD} together with $J K = K J$ results in
  \begin{equation}
  \begin{aligned}
    D^T (I - (-1)^p J) k
      &= (I + (-1)^p J) D^T k
       = - (I + (-1)^p J) K d
    \\&= - K (I + (-1)^p J) d.
  \end{aligned}
  \end{equation}
  Since $(I + (-1)^p J) = \diag{\dots, 0, 2, 0}$ and $d = (\dots, *, 0, *)^T$, 
  $(I + (-1)^p J) d = 0$ and therefore also $D^T (I - (-1)^p J) k = 0$. Here
  and in the following, $*$ is a placeholder for an arbitrary real number.
  Because of the implication
  \begin{equation}
    0 = D^T x = 
    \begin{pmatrix}
         0      &  0     &  0     &     0  &  \dots
      \\ 1      &  0     &  0     &     0  &  \dots
      \\ 0      &  3     &  0     &     0  &  \dots
      \\ 1      &  0     &  5     &     0  &  \dots
      \\ 0      &  3     &  0     &     7  &  \dots
      \\ \vdots & \vdots & \vdots & \vdots & \vdots & \ddots
    \end{pmatrix}
    \begin{pmatrix}
      x_0 \\ \vdots \\ x_p
    \end{pmatrix}
    \implies x =
    \begin{pmatrix}
      0 \\ \vdots \\ 0 \\ x_p
    \end{pmatrix}
  \end{equation}
  and $I - (-1)^p J = \diag{\dots, 2, 0, 2, 0, 2}$, one can deduce that $k$ can
  be written as
  $k = (\dots, 0, *, 0, *, *)^T$. Finally, using $d^T k = 0$ from equation
  \eqref{eq:DtK+KD} yields $k_p = k_{p-2} = \dots = 0$ and finishes the proof.
\end{proof}

\subsection{Summary}

The results of the previous sections are summed up in the following
\begin{thm}
\label{thm:CPR-1D-linear}
  Let a one dimensional CPR method be given by a nodal basis $\mathcal{B}$
  of polynomials of   degree $\leq p$, associated with a quadrature,
  given by symmetric nodes   $z_0, \dots, z_p \in (-1,1)$ and positive
  weights $\omega_0, \dots, \omega_p > 0$, that is exact for polynomials of
  degree $\leq 2p - 1$. Let
  \begin{itemize}
    \item 
    $\mat{M} = \diag{\omega_0, \dots, \omega_p} > 0$ be the (positive definite
    and diagonal) mass matrix associated with a bilinear volume quadrature,
    
    \item
    $\mat{R}$ be the restriction operator, performing an interpolation to the
    boundary,
    
    \item
    $\mat{B} = \diag{-1, 1}$ be the boundary matrix, associated with an integral
    along the outer normal of the boundary, and
    
    \item
    $\mat{D}$ be the discrete derivative matrix, associated with the divergence
    operator,
  \end{itemize}
  satisfying the SBP property
  $\mat{M} \mat{D} + \mat{D}^T \mat{M} = \mat{R}^T \mat{B} \mat{R}$.
  Then the following results are valid:
  
  \begin{enumerate}
    \item
    If $\vec{1}^T \mat{M} \mat{C} = \vec{1}^T \mat{R}^T \mat{B}$, where
    $\vec{1}$ is the representation of the constant function $\xi \mapsto 1$,
    then the SBP CPR method with correction parameters $\mat{C}$ is conservative
    (see Lemma \ref{lem:CPR-conservation}).
    
    \item
    If $\mat{C} = (\mat{M} + \mat{K})^{-1} \mat{R}^T \mat{B}$, where
    $\mat{M} + \mat{K}$ is positive definite and $\mat{K} \mat{D}$ is antisymmetric,
    then the SBP CPR method given by $\mat{C}$ is linearly stable in the discrete
    norm $\norm{\cdot}_{M+K}$ induced by $\mat{M} + \mat{K}$
    (see Lemma \ref{lem:CPR-linearly-stable}).
    
    \item
    If again $\mat{C} = (\mat{M} + \mat{K})^{-1} \mat{R}^T \mat{B}$, where
    $\mat{M} + \mat{K}$ is positive definite and $\mat{K} \mat{D}$ is antisymmetric,
    then the SBP CPR method given by $\mat{C}$ is associated with a FR scheme using
    symmetric correction functions $g_L(\xi) = g_R(-\xi)$
    (see Lemmata \ref{lem:CPR-symmetry} and \ref{lem:CPR-symmetry-satisfied}).
  \end{enumerate}
\end{thm}

\subsection{The one parameter family of Vincent et al. (2011)}

The approach of Vincent et al. \cite{vincent2011newclass} can be formulated as enforcing
$\mat{K} \mat{D} = 0$ by setting $\mat{K} = c (\mat{D}^p)^T \mat{D}^p$,
because $\mat{D}^{p+1} = 0$ (polynomials of degree $\leq p$).
However, in this work the ansatz $\mat{K} = \kappa (\mat{D}^p)^T \mat{M} \mat{D}^p$
is chosen to allow an interpretation in terms of discrete norms. Additionally,
the transformation of the matrices during a change of the basis is only handled
consistently in this way. In the following section, Gauß- and Lobatto-Legendre
quadrature rules accompanied by the associated nodal polynomial basis of
degree $\leq p$ are considered. Therefore, the leading assumptions of Theorem
\ref{thm:CPR-1D-linear} are satisfied.

For concrete computations, again a change to the Legendre basis is advantageous.
In these coordinates, the derivative matrix to the power of $p$ is simply
\begin{equation}
  \hat{\mat{D}}^p = \begin{pmatrix}
                      0 & \dots & 0 & p! \, a_p
                    \\0 & \dots & 0 & 0
                    \\\vdots & \vdots & \vdots & \vdots
                    \\ 0 & 0 & 0 & 0
                    \end{pmatrix},
  \quad
  a_p = \frac{(2p)!}{2^p (p!)^2},
\end{equation}
referring to the leading coefficient of the Legendre polynomial of degree $p$
in the same way as Vincent et al. \cite{vincent2011newclass}  as $a_p$. Therefore, using 
$\hat{\mat{M}} = \diag{2, *, \dots, *}$, the ansatz for $\hat{\mat{K}}$ becomes
\begin{equation}
  \hat{\mat{K}} = \kappa \begin{pmatrix}
                            0 & \dots & 0 & 0
                         \\ \vdots & \vdots & \vdots & \vdots
                         \\ 0 & \dots & 0 & 0
                         \\ 0 & \dots & 0 & 2 a_p^2 (p!)^2
                         \end{pmatrix}.
\end{equation}
The choice of the basis influences further computations through the mass matrix
$\hat{\mat{M}}$. In the following, variables associated with the Gauß-Legendre
and Lobatto-Legendre basis are denoted using a superscript $G$ and $L$, respectively.
Gauß-Legendre quadrature is exact for polynomials of degree $\leq 2p+1$ and
Lobatto-Legendre quadrature is exact for polynomials of degree $\leq 2p - 1$.
Therefore, $\hat{\mat{M}}^G$ and $\hat{\mat{M}}^L$ are both diagonal and
the last entry of $\hat{\mat{M}}^L$ is $\frac{2}{p}$ in accordance with
\cite{gassner2011comparison}:
\begin{equation}
\begin{aligned}
  &\hat{\mat{M}}^G = \diag{2, \frac{2}{3}, \dots, \frac{2}{2p-1}, \frac{2}{2p+1}},
  \\
  &\hat{\mat{M}}^L = \diag{2, \frac{2}{3}, \dots, \frac{2}{2p-1}, \frac{2}{p}}.
\end{aligned}
\label{eq:MG-and-ML}
\end{equation}
Therefore, $\mat{M}^G + \mat{K}$ and $\mat{M}^L + \mat{K}$ are positive definite
if and only if
\begin{equation}
  \kappa > \kappa_-^G := - \frac{1}{(2p+1) a_p^2 (p!)^2}
  , \quad
  \kappa > \kappa_-^L := - \frac{1}{p a_p^2 (p!)^2},
\label{eq:kappa-range}
\end{equation}
respectively. Therefore, the associated SBP CPR methods given by
$\mat{C} = (\mat{M} + \mat{K})^{-1} \mat{R}^T \mat{B}$
are linearly stable and conservative by Theorem \ref{thm:CPR-1D-linear} if $\kappa$
is chosen accordingly to \eqref{eq:kappa-range}. In addition, they are conservative,
since
\begin{equation}
\begin{aligned}
  \hat{\vec{1}}^T \hat{\mat{M}} \hat{\mat{C}}
  &= \hat{\vec{1}}^T \hat{\mat{M}}
     (\hat{\mat{M}} + \hat{\mat{K}})^{-1} \hat{\mat{R}}^T \hat{\mat{B}}
\\&= (1, 0, \dots, 0) \diag{2, *, \dots, *} \diag{\frac{1}{2}, *, \dots, *}
     \hat{\mat{R}}^T \hat{\mat{B}}
   = \hat{\vec{1}}^T \hat{\mat{R}}^T \hat{\mat{B}}.
\end{aligned}
\end{equation}
To compare the resulting methods with the ones obtained in \cite{vincent2011newclass},
equation \eqref{eq:C-Legendre} can be used. To compute $\mat{C}$ explicitly, the
restriction matrix $\mat{R}$ has to be computed in the Legendre basis. Describing
interpolation to the boundary, using $\phi_i(1) = 1$ and $\phi_i(-1) = (-1)^i$
it can be written as
\begin{equation}
  \hat{\mat{R}} = \begin{pmatrix}
                    1 & -1 & 1 & \dots & (-1)^p
                  \\1 &  1 & 1 & \dots & 1
                  \end{pmatrix}.
\end{equation}
Therefore, computing $\mat{C} = (\mat{M} + \mat{K})^{-1} \mat{R}^T \mat{B}$
explicitly results in
\begin{equation}
\begin{aligned}
  \hat{\mat{C}}^{G/L}
  &= \diag{\frac{1}{2}, \dots, \frac{2}{2p-1}, *^{G/L}}
    \begin{pmatrix}
      -1 & 1
    \\ 1 & 1
    \\ \vdots & \vdots
    \\ (-1)^{p+1} & 1
    \end{pmatrix}
\\&= \begin{pmatrix}
      -\frac{1}{2}           &  \frac{1}{2}
    \\ \frac{3}{2}           &  \frac{3}{2}
    \\ \vdots                &  \vdots
    \\ (-1)^p \frac{2p-1}{2} &  \frac{2p-1}{2}
    \\ (-1)^{p+1} *^{G/L}    &  *^{G/L}
    \end{pmatrix}
  = (\hat{\vec{g_L'}}^{G/L}, \, \hat{\vec{g_R'}}^{G/L}),
\end{aligned}
\label{eq:C-one-parameter-family-kappa}
\end{equation}
where $*^G = \left(\frac{2}{2p+1} + 2 \kappa a_p^2 (p!)^2 \right)^{-1}$
and $*^L = \left(\frac{2}{p} + 2 \kappa a_p^2 (p!)^2 \right)^{-1}$.
The (symmetric) correction functions of \cite{vincent2011newclass} are given by
\begin{equation}
  g_L = \frac{(-1)^p}{2}
        \left( \phi_p - \frac{\eta_p \phi_{p-1} + \phi_{p+1}}{1 + \eta_p} \right)
  ,\,
  g_R = \frac{1}{2}
        \left( \phi_p + \frac{\eta_p \phi_{p-1} + \phi_{p+1}}{1 + \eta_p} \right),
\label{eq:vincent-et-al-one-parameter-g}
\end{equation}
where
\begin{equation}
  \eta_p = c \frac{2p+1}{2} a_p^2 (p!)^2.
\end{equation}
Therefore, in order to compare the results, it remains to compute the derivatives
of \eqref{eq:vincent-et-al-one-parameter-g}. The derivative matrix of size
$(p+2) \times (p+2)$ for even $p$ is given as
\begin{equation}
  \hat{\mat{D}}_+
  = \begin{pmatrix}
       0  &  1  &  0  &  1       &  \dots  &  1       &  0       &  1
    \\    &  0  &  3  &  0       &  \dots  &  0       &  3       &  0
    \\    &     &  0  &  5       &  \dots  &  5       &  0       &  5
    \\    &     &     &  \ddots  &  \dots  &  \vdots  &  \vdots  &  \vdots
    \\    &     &     &          &         &          &  0       &  2p+1
    \\    &     &     &          &         &          &          &  0
    \end{pmatrix}
\end{equation}
and as
\begin{equation}
  \hat{\mat{D}}_+
  = \begin{pmatrix}
       0  &  1  &  0  &  1       &  \dots  &  1       &  0
    \\    &  0  &  3  &  0       &  \dots  &  0       &  3
    \\    &     &  0  &  5       &  \dots  &  5       &  0 
    \\    &     &     &  \ddots  &  \dots  &  \vdots  &  \vdots
    \\    &     &     &          &         &  0       &  2p+1
    \\    &     &     &          &         &          &  0
    \end{pmatrix}
\end{equation}
for odd $p$. Multiplication with
\begin{equation}
  \hat{\vec{g_L}}
  = \frac{(-1)^p}{2}  \left(
                        0,
                        \dots,
                        0,
                         -\frac{\eta_p}{1 + \eta_p},
                         1,
                         -\frac{1}{1 + \eta_p}
                      \right)^T
\end{equation}
results for both even and odd $p$ in the same coefficients with indices $0$ to
$p-1$ as in \eqref{eq:C-one-parameter-family-kappa} and thus, in order to get
the same methods, the last coefficient has to be the same, resulting in the equation
\begin{equation}
  *^{G/L}
  = \frac{2p+1}{2} \frac{1}{1 + \eta_p}
  = \frac{2p+1}{2} \frac{1}{1 + c \frac{2p+1}{2} a_p^2 (p!)^2}.
\label{eq:compare-one-parameter-coefficients}
\end{equation}
Inserting $*^{G/L}$ from above, this results in
\begin{equation}
  (*^G)^{-1}
  = \frac{2}{2p+1} + 2 \kappa^G a_p^2 (p!)^2
  = \frac{2}{2p+1} + c a_p^2 (p!)^2
\end{equation}
and
\begin{equation}
  (*^L)^{-1}
  = \frac{2}{p} + 2 \kappa^L a_p^2 (p!)^2
  = \frac{2}{2p+1} + c a_p^2 (p!)^2,
\end{equation}
respectively, and therefore the parameter $c$ of \cite{vincent2011newclass}
can be expressed as
\begin{equation}
  c
  = 2 \kappa^G
  = 2 \kappa^L + c_{Hu},
  \quad c_{Hu} = 2 \frac{p+1}{p (2p+1) a_p^2 (p!)^2},
\end{equation}
where $c_{Hu}$ is the coefficient
recovering the scheme named $g_2$ by Huynh \cite{huynh2007flux}. This scheme is exactly
the same as the DGSEM scheme with a nodal basis at Lobatto-Legendre nodes and a
lumped mass matrix used in \cite{gassner2013skew, gassner2014kinetic,
kopriva2014energy, gassner2016well} and proven to be an SBP scheme.

\subsection{The multi parameter family of Vincent et al. (2015)}

The results for the multi parameter family of linearly stable and conservative
schemes of Vincent et al. \cite{vincent2015extended} are similar to those obtained in the previous
section about the one parameter family -- as expected, since the one parameter
family is contained in the extended range of schemes.

The calculations of \cite{vincent2015extended} used an exact mass matrix (in
the Legendre basis) and are thus valid for Gauß-Legendre points. Using Lobatto-Legendre
quadrature will result in a transformed parameter space, recovering the same schemes
as before, similar to the previous section. In contrast to their results, the
solution point coordinates are considered to be an important parameter of an
SBP CPR method and thus included in the analysis. Therefore, discrete norms
are investigated and stability results are stated in these discrete norms.

\subsection{Numerical examples}

In order validate the implementation, the numerical experiments presented in
\cite{vincent2011newclass} are repeated. The conservation law solved is the
linear advection equation \eqref{eq:linear-constant-advection} with constant
velocity $1$ in one space dimension in the interval $[-1, 1]$ with periodic
boundary conditions. The initial condition is
\begin{equation}
  u_0(x) = e^{-20 x^2}.
\end{equation}
Several SBP CPR methods with $N = 10$ equally spaced elements of order $p = 3$
are utilised as semidscretisations and the classical fourth order Runge-Kutta method
with $50,000$ steps is used to obtain the solution in the time interval $[0, 20]$,
i.e. ten traversals of the initial data are regarded.

Results for a Lobatto-Legendre basis and the central numerical flux are shown
in Figures \ref{fig:vincent2011newclass-lobatto-central-1},
\ref{fig:vincent2011newclass-lobatto-central-2} and
\ref{fig:vincent2011newclass-lobatto-central-3}. Four different values of the
parameter $c$ for the correction matrix $\mat{C}$ are used,
the same as presented in \cite{vincent2011newclass}:
$c = c_-/2 = \frac{-1}{(2p+1) (a_p \, p!)^2} < 0$
(a negative parameter near the boundary value $c_-$ for stable schemes),
$c = c_0 = 0$ (no additional matrix $\mat{K}$ for exact integration, i.e. in
the framework of \cite{vincent2011newclass}, corresponding to Gauß-Legendre points in
the framework presented here), $c = c_{SD} = \frac{2p}{(2p+1) (p+1) (a_p \, p!)^2}$
(recovering a spectral difference method), and
$c = c_{Hu} = \frac{2(p+1)}{(2p+1) p (a_p \, p!)^2}$ (using the correction
functions named $g_2$ by Huynh \cite{huynh2007flux}, corresponding to the DGSEM
of \cite{gassner2013skew}).
Figure \ref{fig:vincent2011newclass-lobatto-central-1} consists
of plots of the solution at $t = 20$ (in blue) and the initial profile
at $t = 0$ (in green). In Figure \ref{fig:vincent2011newclass-lobatto-central-2},
the squared $L_2$ norms computed via Gauß (blue) and Lobatto (green) quadrature
in the time interval $[0, 20]$ are plotted. Finally, Figure 
\ref{fig:vincent2011newclass-lobatto-central-3} provides a zoomed in view of the
time interval $[0, 0.8]$. The solutions are visually the same as those obtained
in \cite{vincent2011newclass}. Since $c_{Hu}$ corresponds to the correction
function $g_2$ of \cite{huynh2007flux} and the corresponding SBP CPR method is
the same as the DGSEM of \cite{gassner2013skew}, the energy computed via
Lobatto quadrature remains constant for this choice of $c$. The results obtained
by using a Gauß-Legendre basis look very much the same at this resolution and
are consequently not printed.

In the CPR framework, the solution is approximated as a piecewise
polynomial function. Thus, derived quantities like norms are computed exactly
of approximately for these polynomials on each element. Therefore, Gauß-Legendre
or Lobatto-Legendre quadrature rules are natural choices to compute $L_2$ norms.
However, as shown in the previous sections, each choice of correction parameter
for a CPR method is associated with a natural norm / scalar product, given
by $\mat{M} + \mat{K}$. For $c = c_0 = 0$ and $c = c_{Hu}$, these scalar
products are given by Gauß and Lobatto quadrature, respectively. Using a central
flux, energy \emph{in this specific norm} is conserved. By equivalence of norms
in finite dimensional spaces, energy computed via other quadrature rules is
bounded, but not necessarily conserved or non-increasing. This can be seen in
Figures \ref{fig:vincent2011newclass-lobatto-central-2} and
\ref{fig:vincent2011newclass-lobatto-central-3}: The natural quadrature rules
(Gauß for $c = c_0$ and Lobatto for $c = c_{Hu}$) yield exactly conserved
energy, whereas other quadrature rules result in bounded oscillations of energy.
Computing the norms $\norm{\cdot}_{M+K}$ for $c = c_-/2$ and $c = c_{SD}$,
the same conservation of energy is obtained, but not plotted here. However,
as the solution $\vec{u}$ in each element represents a polynomial and not
just some point values as in traditional FD methods, Gauß and Lobatto
integration are standard choices to compute $L_2$ norms.

Using an upwind flux instead of a central flux, the corresponding results are
shown in Figures \ref{fig:vincent2011newclass-lobatto-upwind-1},
\ref{fig:vincent2011newclass-lobatto-upwind-2} and
\ref{fig:vincent2011newclass-lobatto-upwind-3}. As before, the
results look like the ones of \cite{vincent2011newclass}. The only
difference is a stronger dissipation of energy for $c = c_- / 2$ and
$c = c_{Hu}$. This may be caused by the different time integrator and unknown
values for the time steps used in \cite{vincent2011newclass}.
Again, the results obtained by using a Gauß-Legendre basis look very much the
same.

As before, only the energy computed via the norm associated with the
chosen correction is necessarily non-increasing. Due to dissipation by the
upwind numerical flux, the corresponding energy decays. Other choices of
quadrature rules still yield oscillating energy, decaying in the large.

A convergence study for a fixed number of elements $N = 10$ and varying
polynomial degree $p$ is plotted in Figure \ref{fig:vincent2011newclass-p-10}.
The corresponding numerical values (rounded to two significant digits)
are printed in Table \ref{tab:vincent2011newclass-p-10}.
Both Gauß-Legendre and Lobatto-Legendre bases with an upwind numerical flux
for different values of $c$ are compared. In addition, the natural choice
$\kappa = 0$ for each basis is considered, i.e. $c_0$ and $c_{Hu}$ for Gauß
and Lobatto quadrature, respectively. Nearly all parameters are the same as
before, but the number of time steps is increased to $50,000$. For fixed $c$,
the results for Gauß-Legendre and Lobatto-Legendre are similar, but for the
natural choice $\kappa = 0$, Gauß-Legendre is clearly superior. All plots show
clearly an approximately exponential decay of the error with increasing $p$ up
to about $p = 10$. There, higher precision than $64$ bit floating point will
probably lead to further decay.

Similar plots for a fixed polynomial degree $p = 4$ and varying number of
elements $N$ can be found in Figure \ref{fig:vincent2011newclass-N-4}
with corresponding numerical values (rounded to two significant digits) in Table
\ref{tab:vincent2011newclass-N-4}.
As before, for fixed $c$, the results are similar but for the natural choice
$\kappa = 0$, the Gauß-Legendre basis is clearly superior.
Note that both studies used the same total number of degrees of freedom and
a high polynomial degree is superior compared to more number of elements for
high precision. In this study, the limit error is not reached for any of the
plotted number of elements.

\begin{figure}[!hp]
  \centering
  \begin{subfigure}[b]{0.49\textwidth}
    \includegraphics[width=\textwidth]{%
      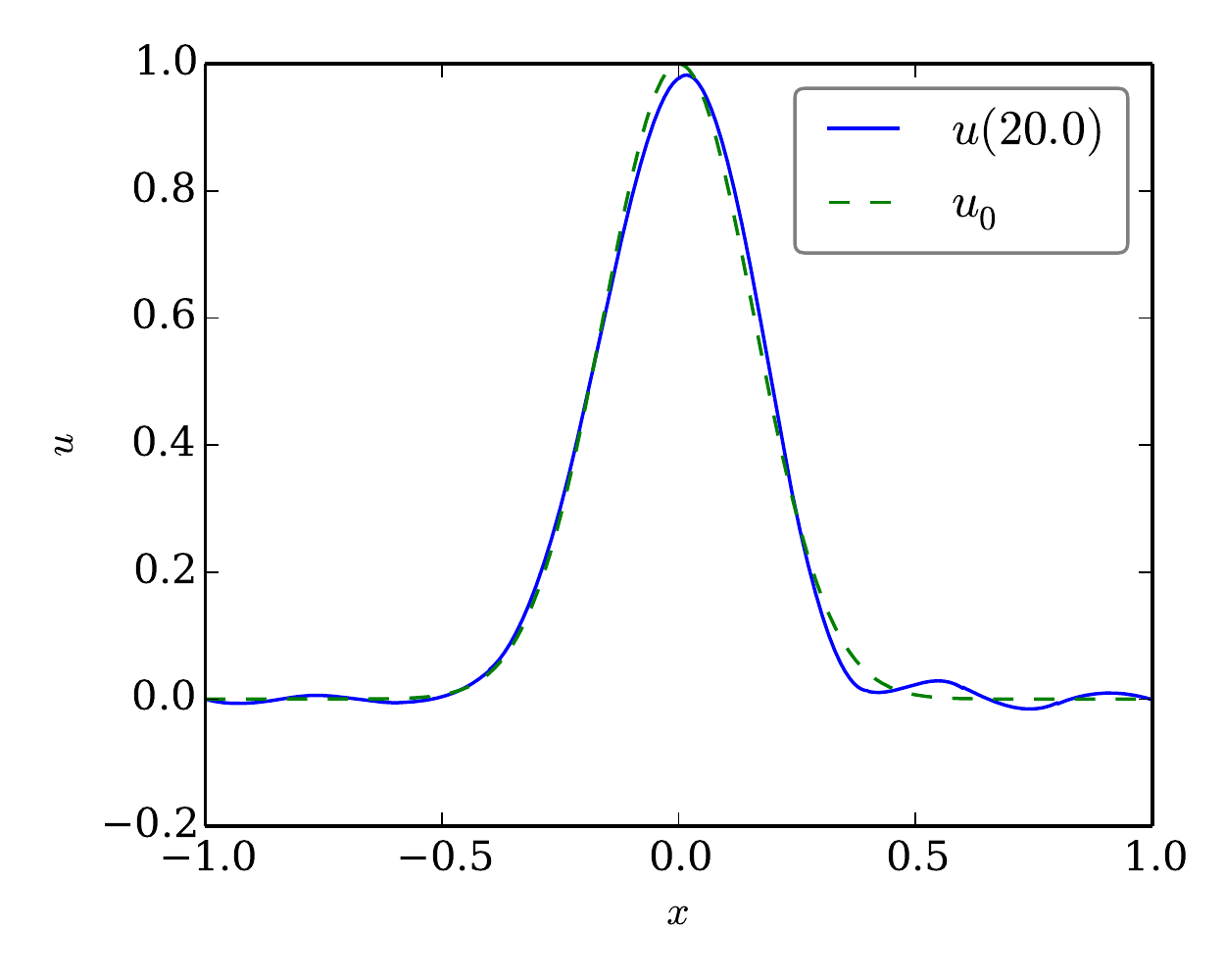}
    \caption{$c = c_- / 2$.}
  \end{subfigure}%
  ~
  \begin{subfigure}[b]{0.49\textwidth}
    \includegraphics[width=\textwidth]{%
      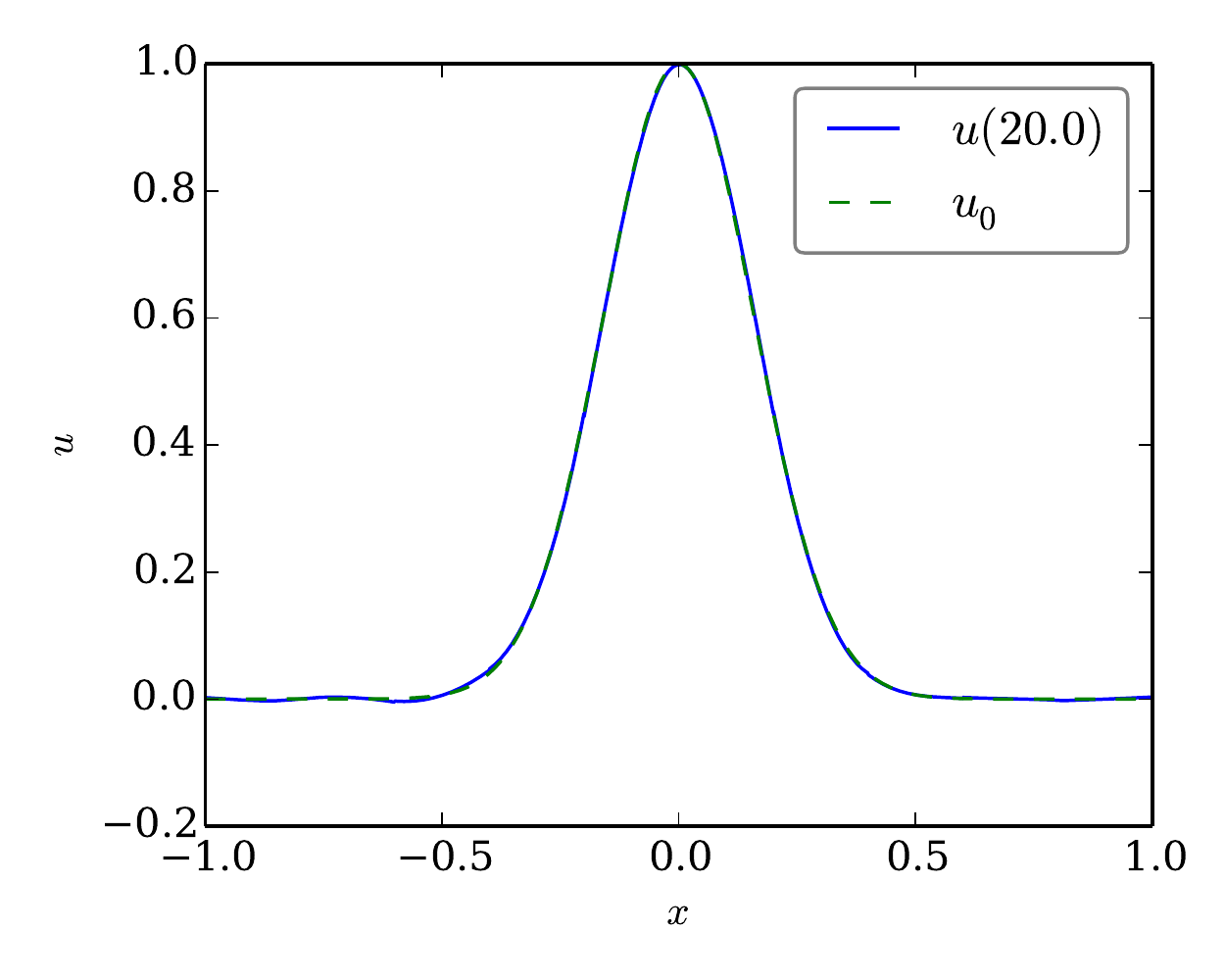}
    \caption{$c = c_0$.}
  \end{subfigure}%
  \\
  \begin{subfigure}[b]{0.49\textwidth}
    \includegraphics[width=\textwidth]{%
      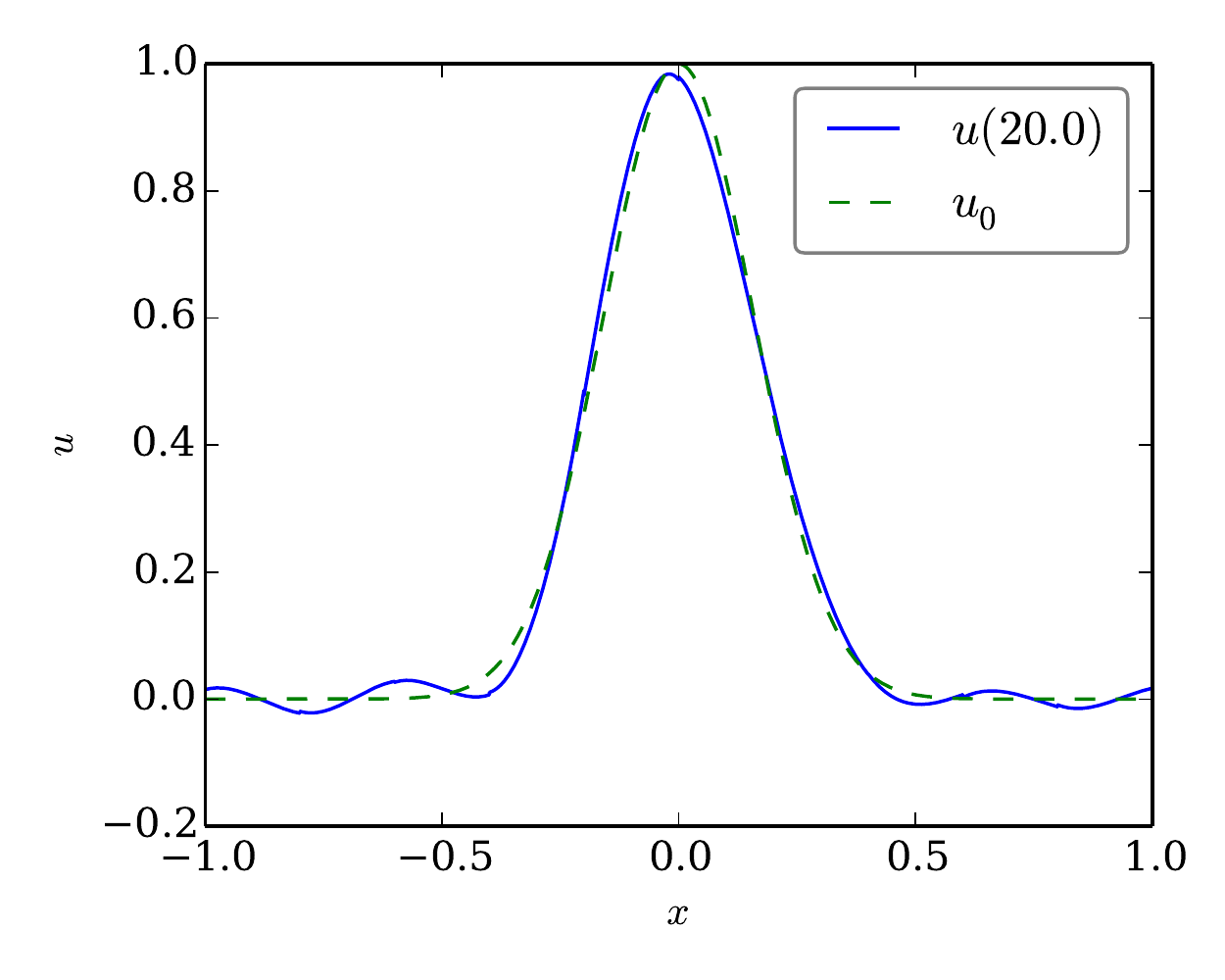}
    \caption{$c = c_{SD}$.}
  \end{subfigure}%
  ~
  \begin{subfigure}[b]{0.49\textwidth}
    \includegraphics[width=\textwidth]{%
      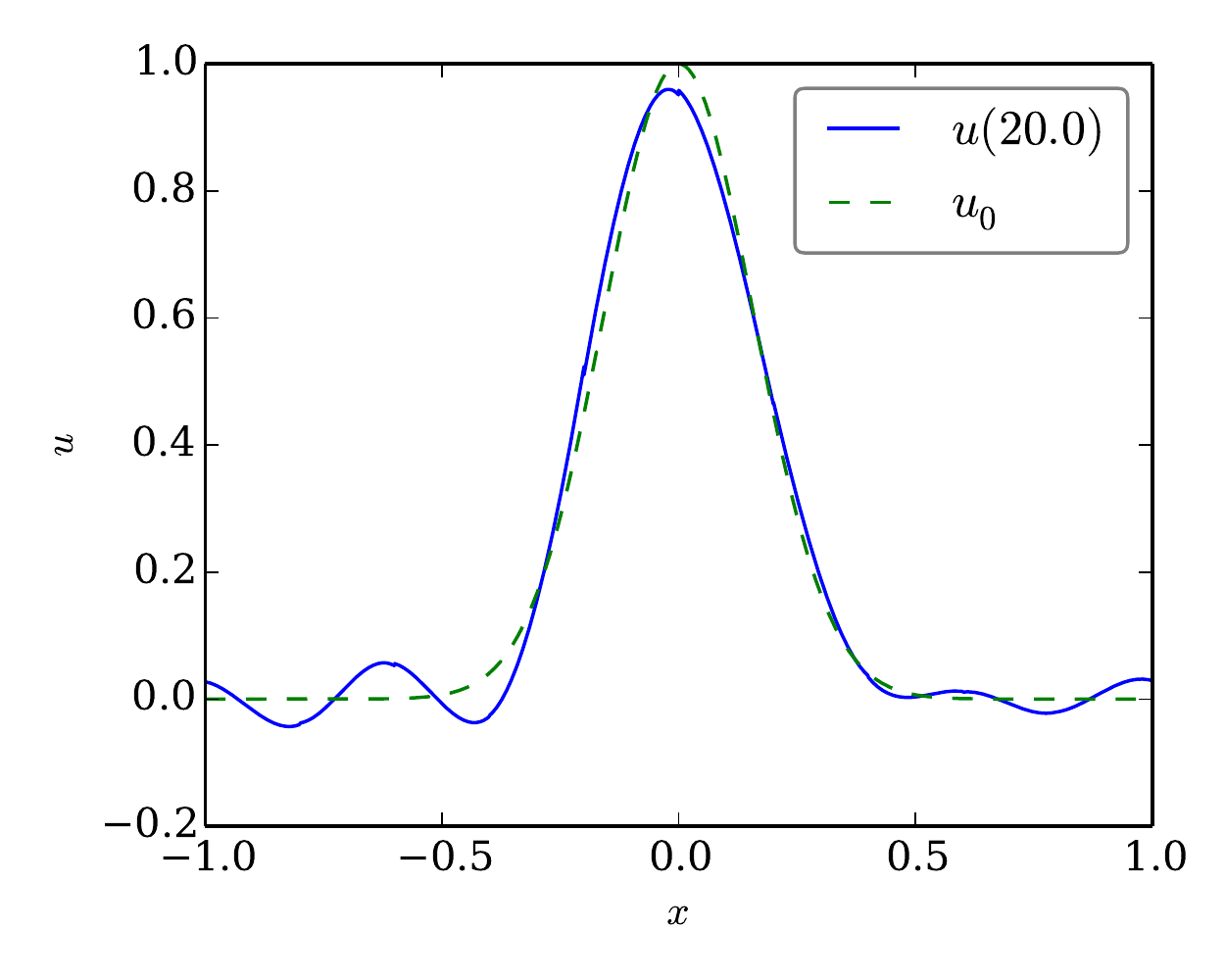}
    \caption{$c = c_{Hu}$.}
  \end{subfigure}%
  \caption{The numerical solution of constant velocity linear advection
            using SBP CPR methods with $10$ elements, a Lobatto-Legendre
            basis of order $p = 3$ and a central numerical flux.
            Different values of $c$ are used for the correction matrix
            $\mat{C}$. The initial Gaussian profile $u_0$ is shown in green,
            the numerical solution is plotted in blue.}
  \label{fig:vincent2011newclass-lobatto-central-1}
\end{figure}

\begin{figure}[!hp]
  \centering
  \begin{subfigure}[b]{0.49\textwidth}
    \includegraphics[width=\textwidth]{%
      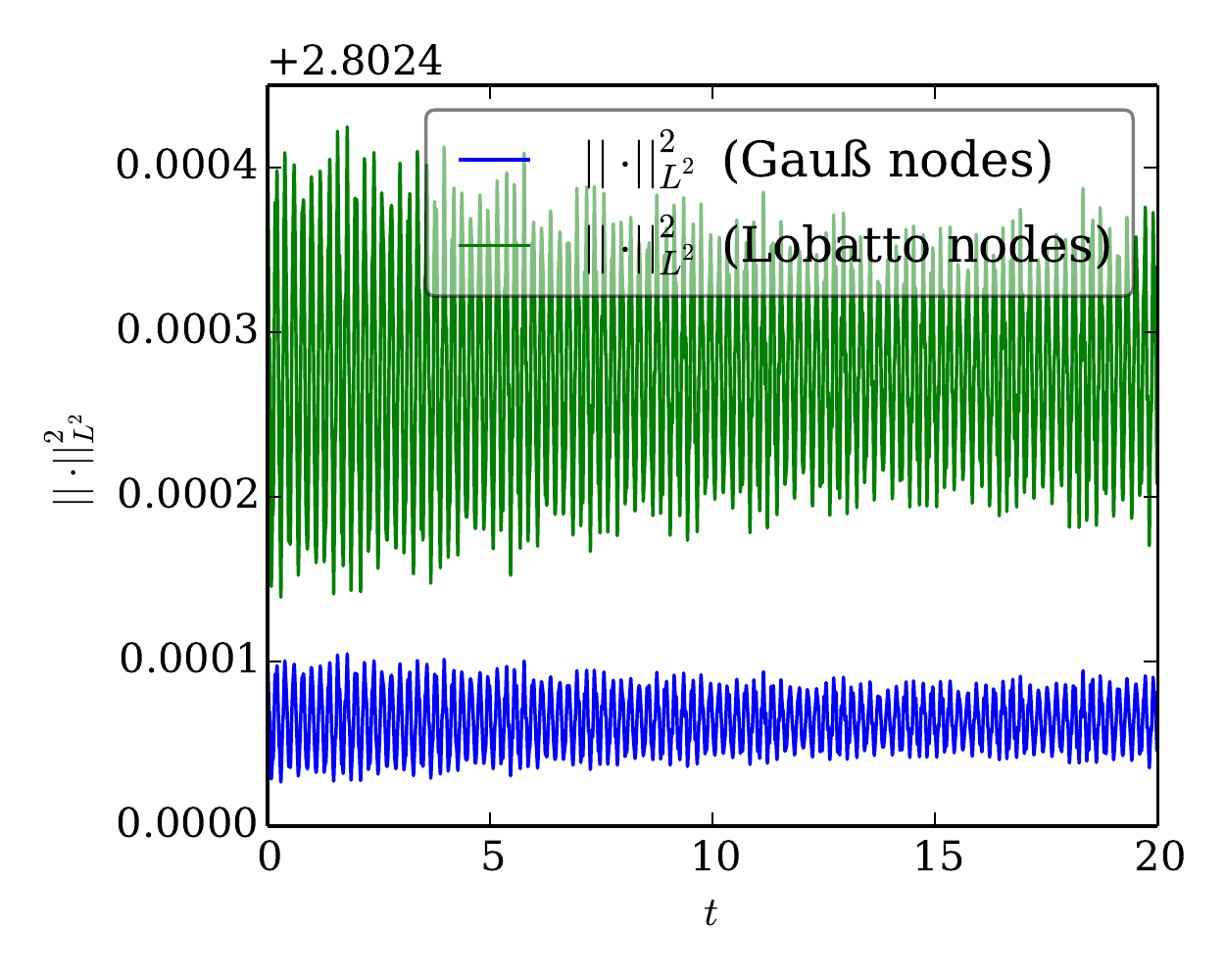}
    \caption{$c = c_- / 2$.}
  \end{subfigure}%
  ~
  \begin{subfigure}[b]{0.49\textwidth}
    \includegraphics[width=\textwidth]{%
      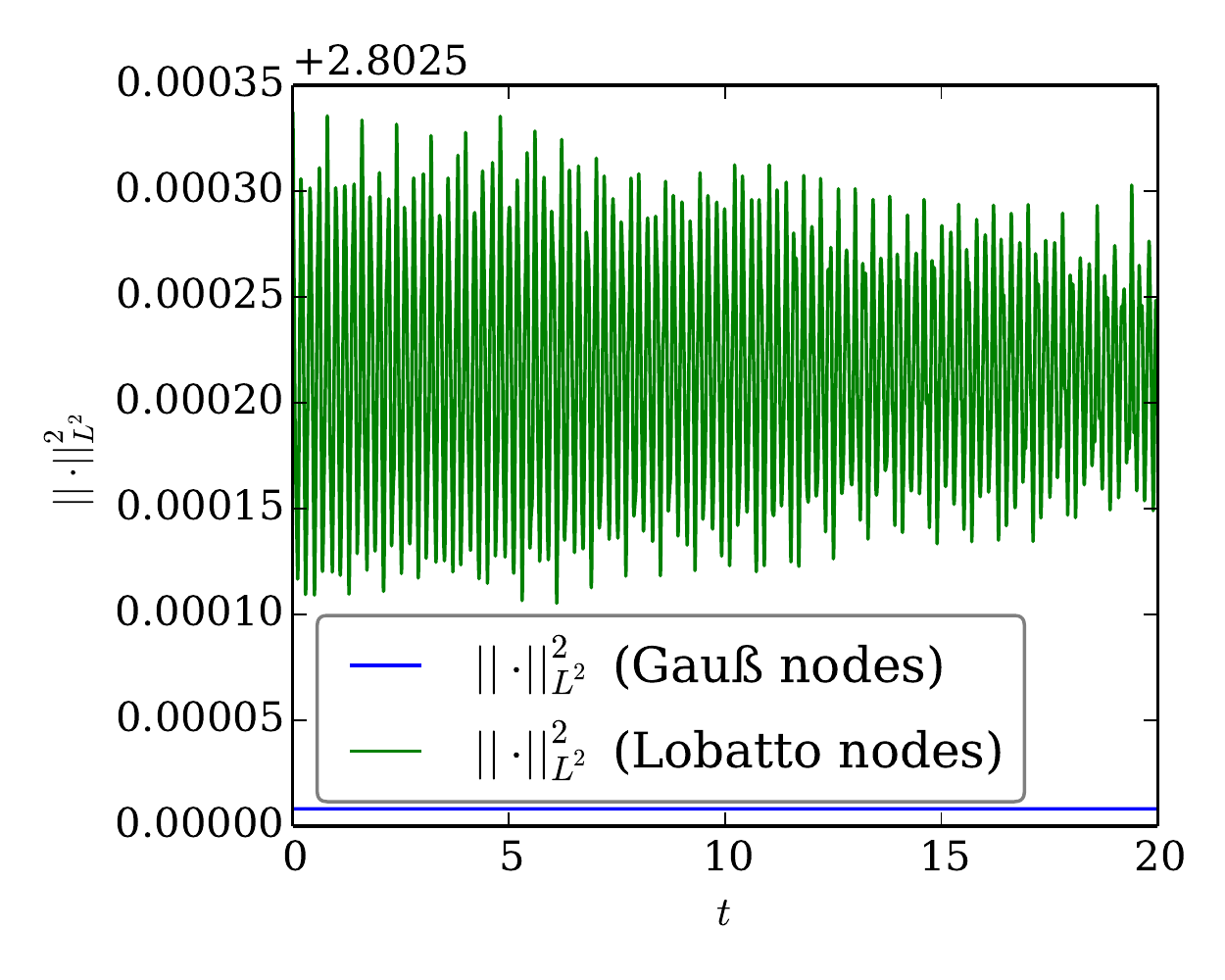}
    \caption{$c = c_0$.}
  \end{subfigure}%
  \\
  \begin{subfigure}[b]{0.49\textwidth}
    \includegraphics[width=\textwidth]{%
      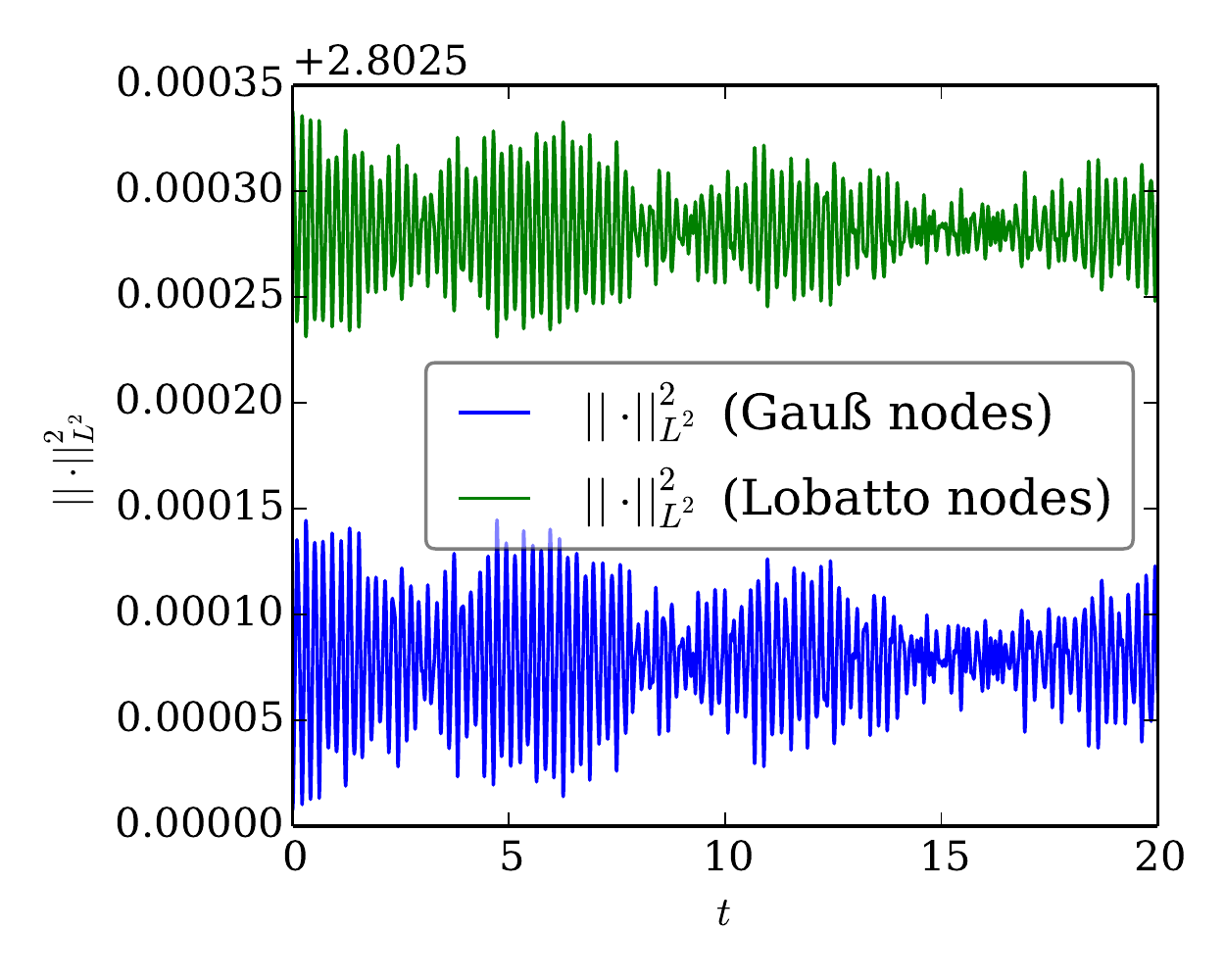}
    \caption{$c = c_{SD}$.}
  \end{subfigure}%
  ~
  \begin{subfigure}[b]{0.49\textwidth}
    \includegraphics[width=\textwidth]{%
      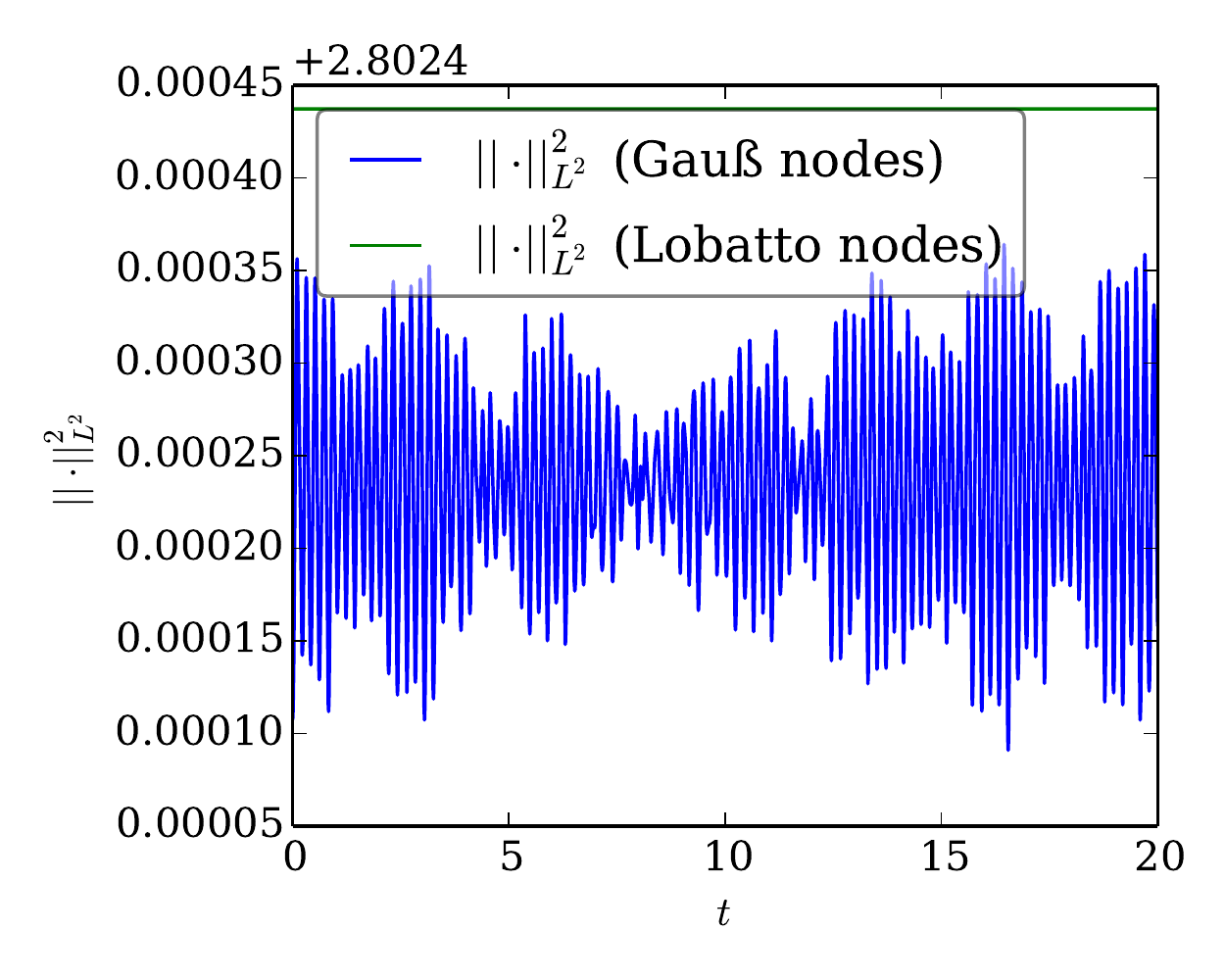}
    \caption{$c = c_{Hu}$.}
  \end{subfigure}%
  \caption{Energy of numerical solution of constant velocity linear
            advection using SBP CPR methods with $10$ elements, a Lobatto-Legendre
            basis of order $p = 3$ and a central numerical flux.
            Different values of $c$ are used for the correction matrix
            $\mat{C}$. The discrete energy $\norm{\vec{u}}^2$ is computed using
            Gauß-Legendre (blue) and Lobatto-Legendre (green) quadrature with
            $p + 1 = 4$ nodes in the full time interval $[0, 20]$.}
  \label{fig:vincent2011newclass-lobatto-central-2}
\end{figure}

\begin{figure}[!hp]
  \centering
  \begin{subfigure}[b]{0.49\textwidth}
    \includegraphics[width=\textwidth]{%
      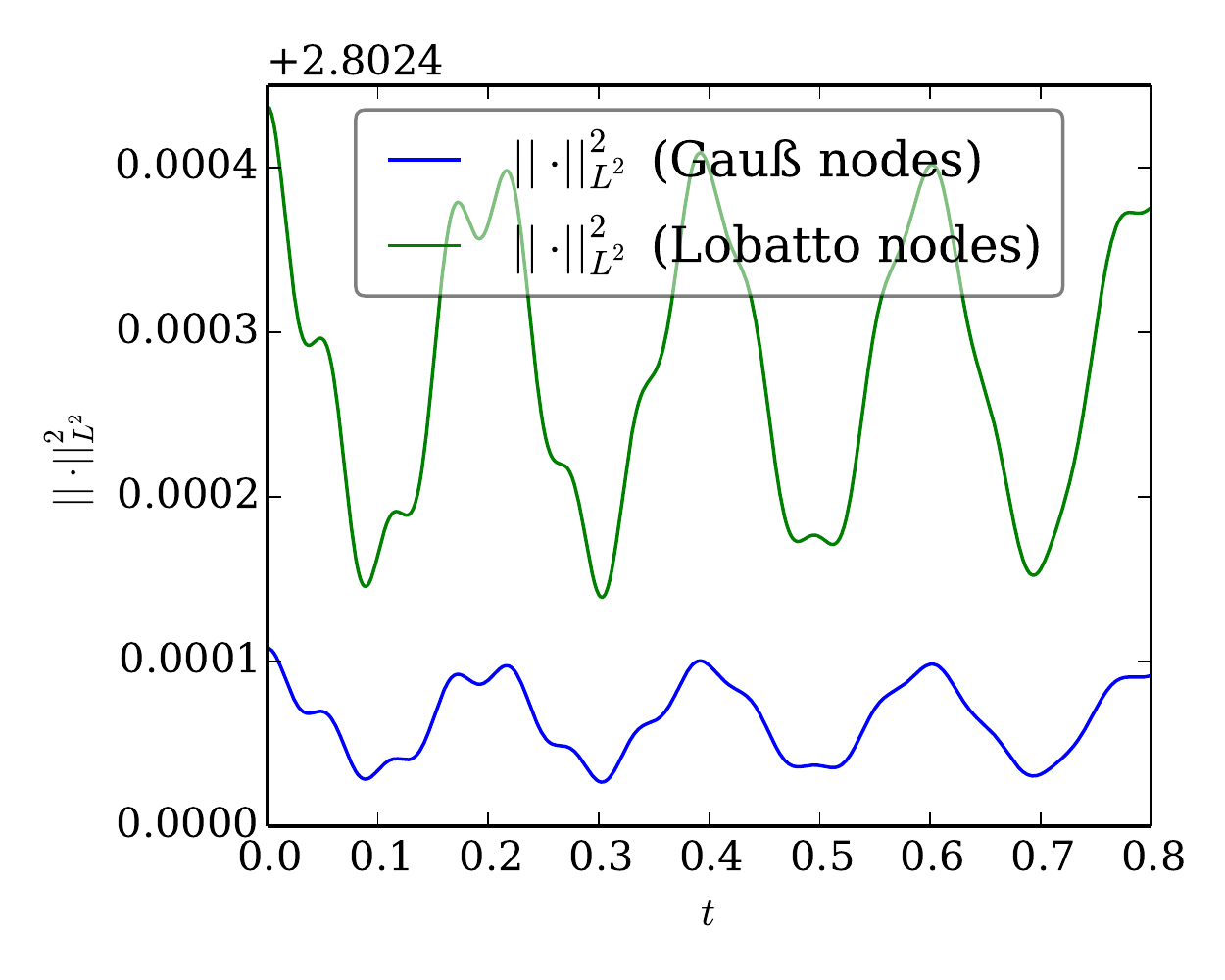}
    \caption{$c = c_- / 2$.}
  \end{subfigure}%
  ~
  \begin{subfigure}[b]{0.49\textwidth}
    \includegraphics[width=\textwidth]{%
      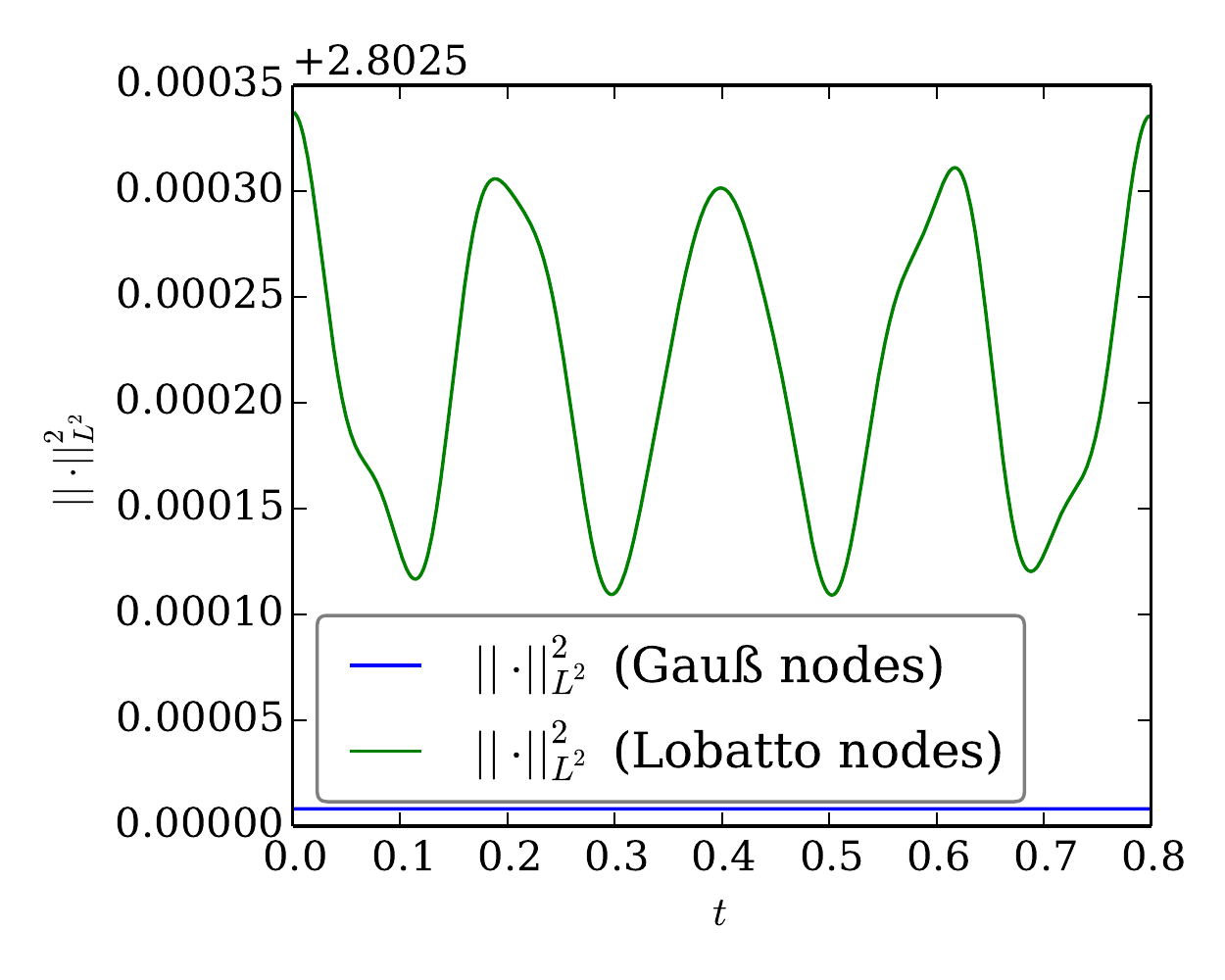}
    \caption{$c = c_0$.}
  \end{subfigure}%
  \\
  \begin{subfigure}[b]{0.49\textwidth}
    \includegraphics[width=\textwidth]{%
      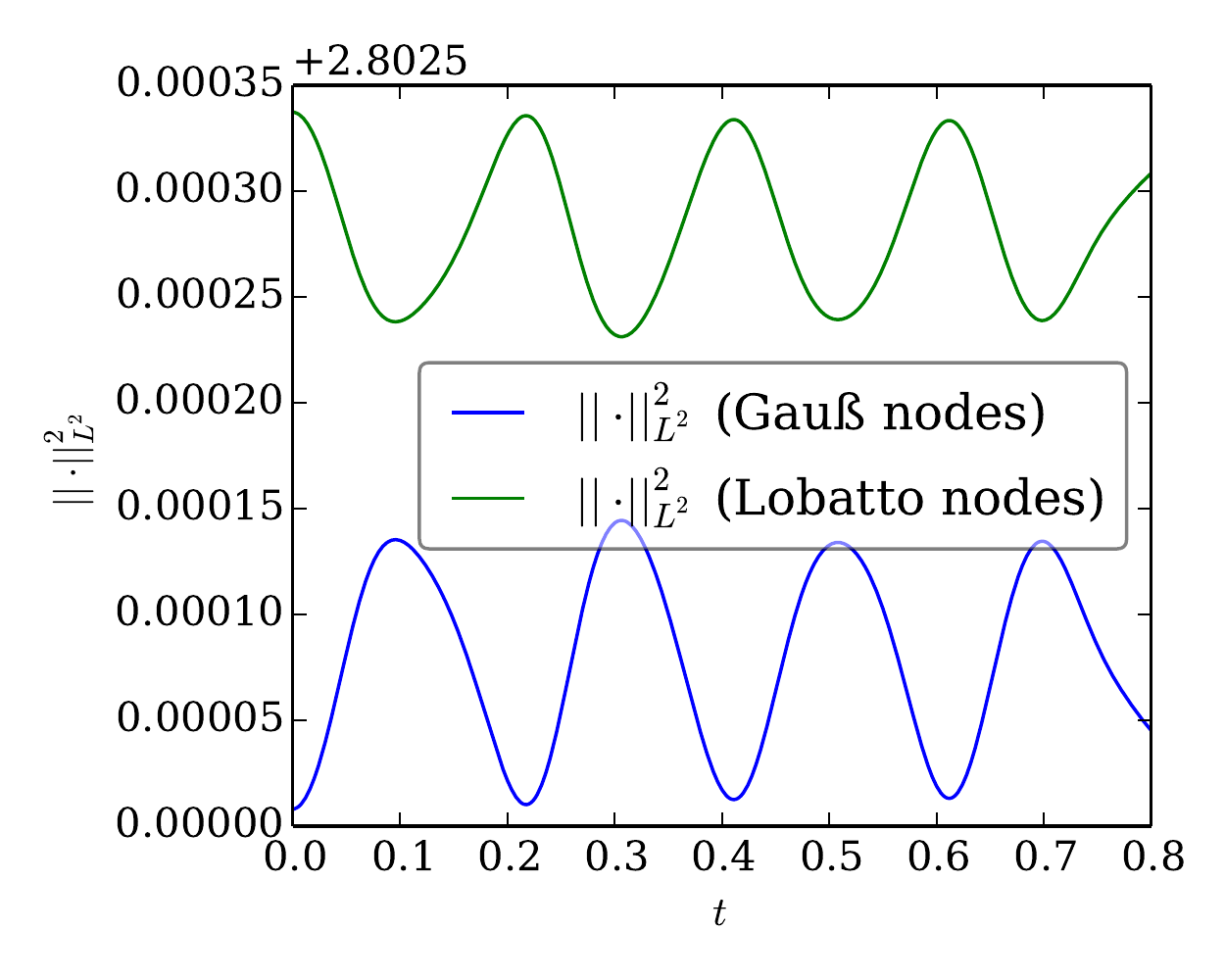}
    \caption{$c = c_{SD}$.}
  \end{subfigure}%
  ~
  \begin{subfigure}[b]{0.49\textwidth}
    \includegraphics[width=\textwidth]{%
      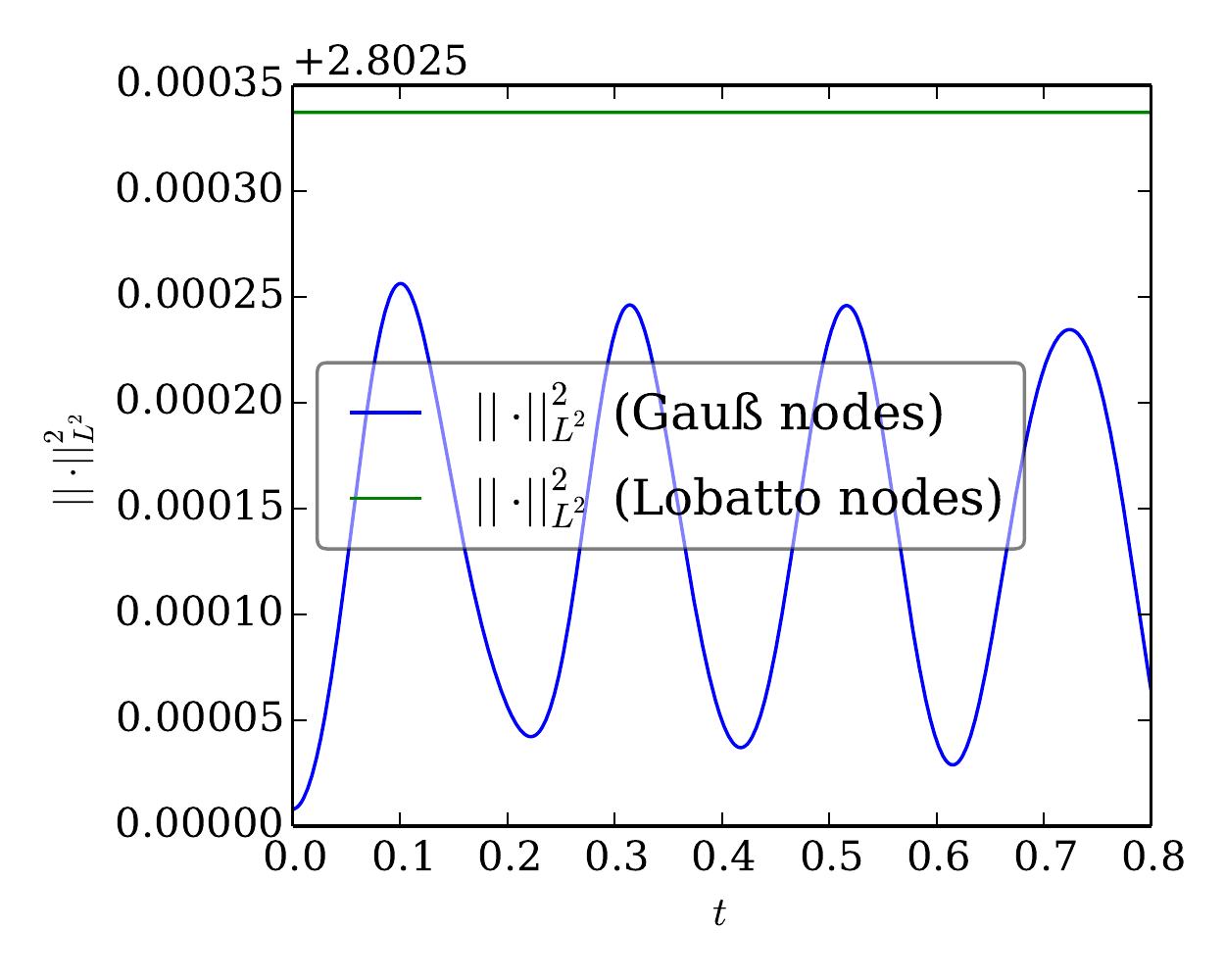}
    \caption{$c = c_{Hu}$.}
  \end{subfigure}%
  \caption{Energy of numerical solution of constant velocity linear
            advection using SBP CPR methods with $10$ elements, a Lobatto-Legendre
            basis of order $p = 3$ and a central numerical flux.
            Different values of $c$ are used for the correction matrix
            $\mat{C}$. The discrete energy $\norm{\vec{u}}^2$ is computed using
            Gauß-Legendre (blue) and Lobatto-Legendre (green) quadrature with
            $p + 1 = 4$ nodes in the time interval $[0, 0.8]$.}
  \label{fig:vincent2011newclass-lobatto-central-3}
\end{figure}

\begin{figure}[!hp]
  \centering
  \begin{subfigure}[b]{0.49\textwidth}
    \includegraphics[width=\textwidth]{%
      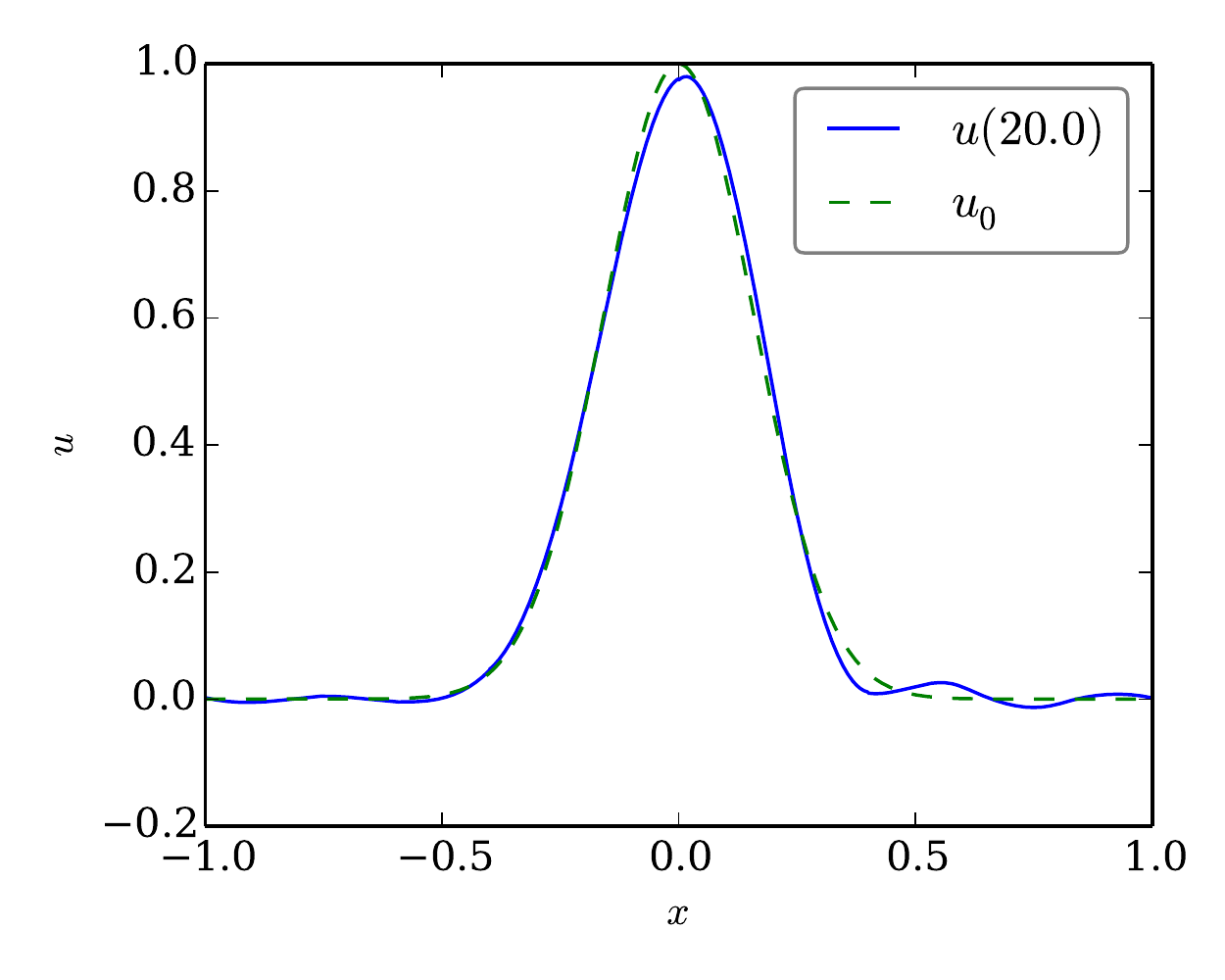}
    \caption{$c = c_- / 2$.}
  \end{subfigure}%
  ~
  \begin{subfigure}[b]{0.49\textwidth}
    \includegraphics[width=\textwidth]{%
      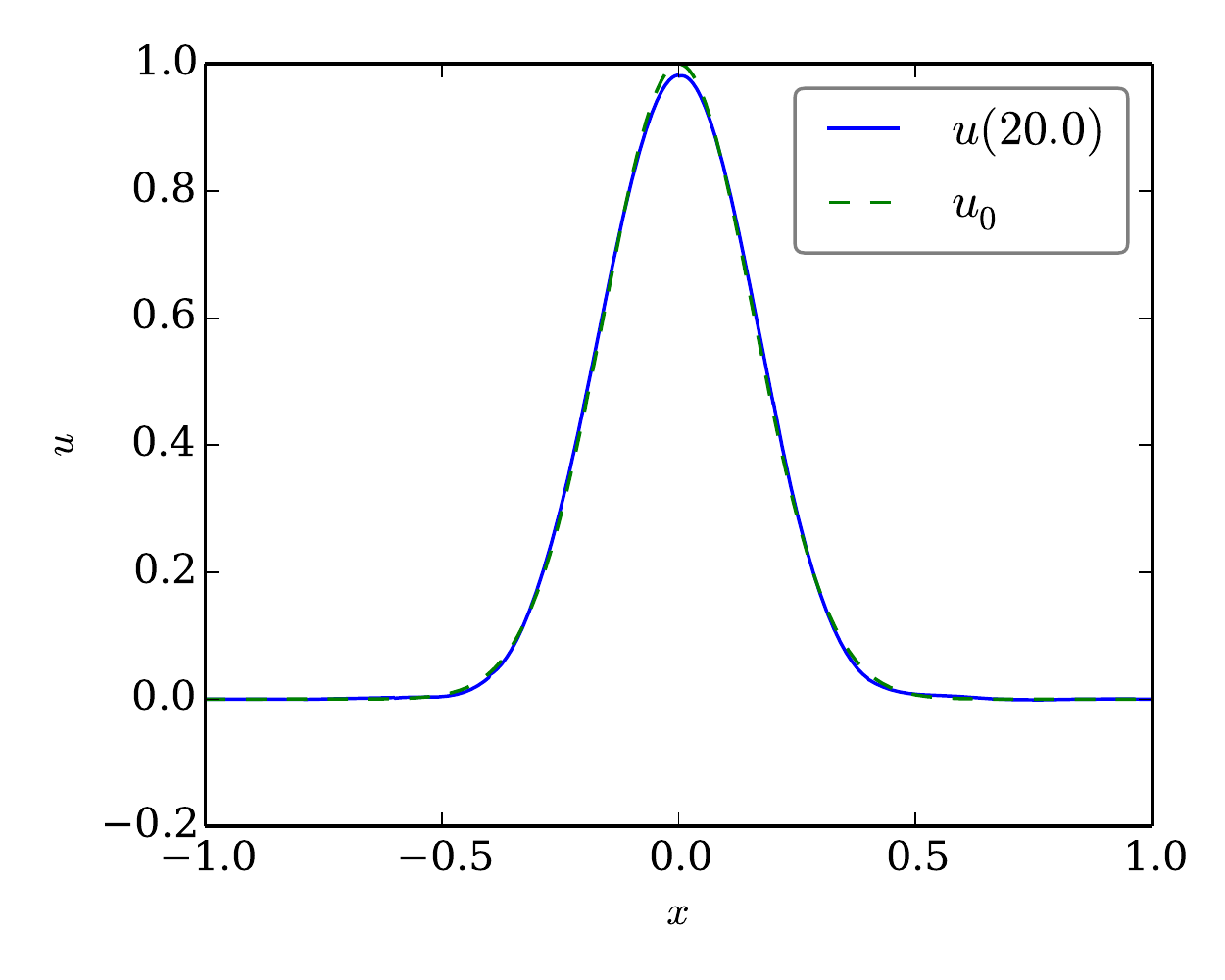}
    \caption{$c = c_0$.}
  \end{subfigure}%
  \\
  \begin{subfigure}[b]{0.49\textwidth}
    \includegraphics[width=\textwidth]{%
      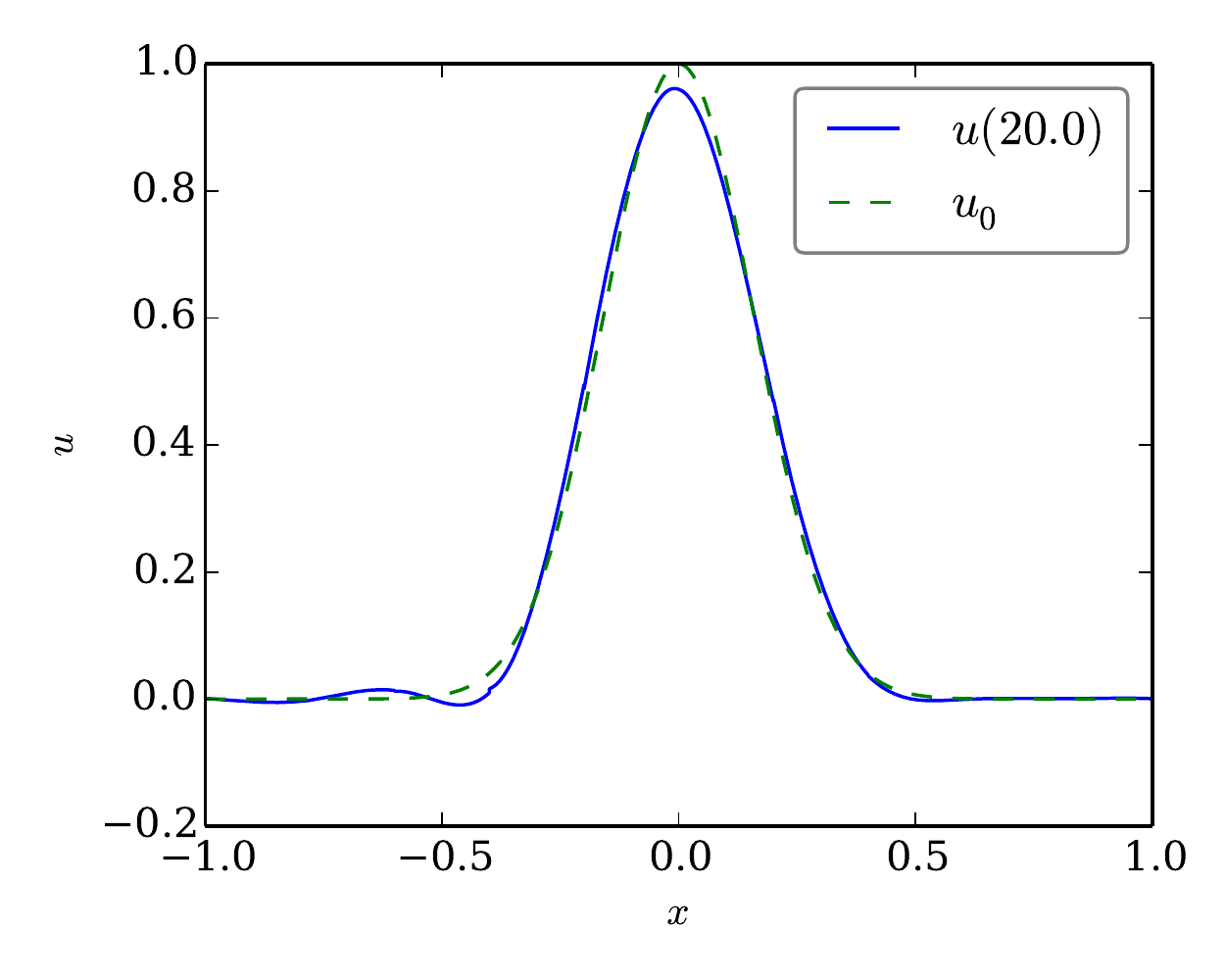}
    \caption{$c = c_{SD}$.}
  \end{subfigure}%
  ~
  \begin{subfigure}[b]{0.49\textwidth}
    \includegraphics[width=\textwidth]{%
      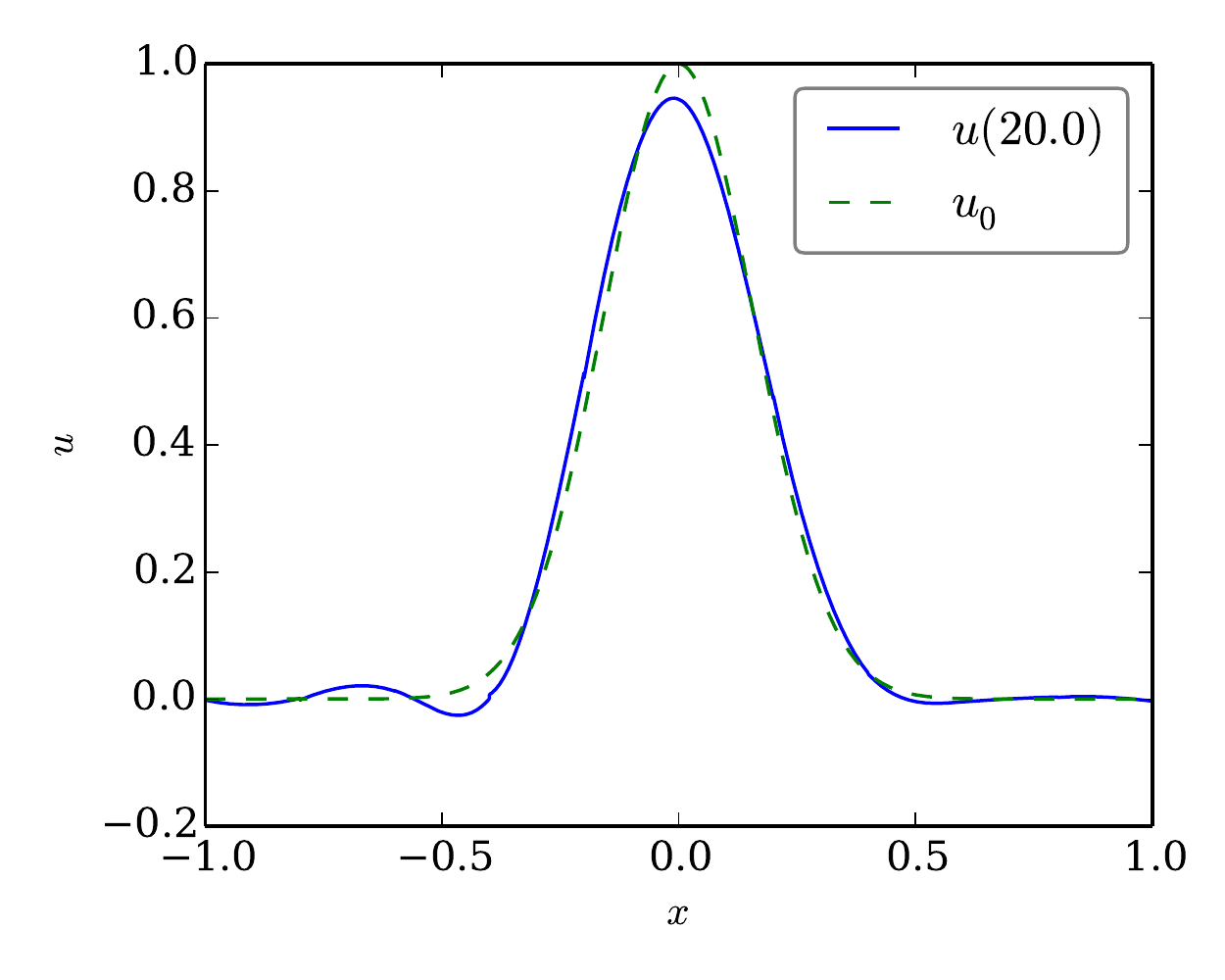}
    \caption{$c = c_{Hu}$.}
  \end{subfigure}%
  \caption{The numerical solution of constant velocity linear advection
            using SBP CPR methods with $10$ elements, a Lobatto-Legendre
            basis of order $p = 3$ and a upwind numerical flux.
            Different values of $c$ are used for the correction matrix
            $\mat{C}$. The initial Gaussian profile $u_0$ is shown in green,
            the numerical solution is plotted in blue.}
  \label{fig:vincent2011newclass-lobatto-upwind-1}
\end{figure}

\begin{figure}[!hp]
  \centering
  \begin{subfigure}[b]{0.49\textwidth}
    \includegraphics[width=\textwidth]{%
      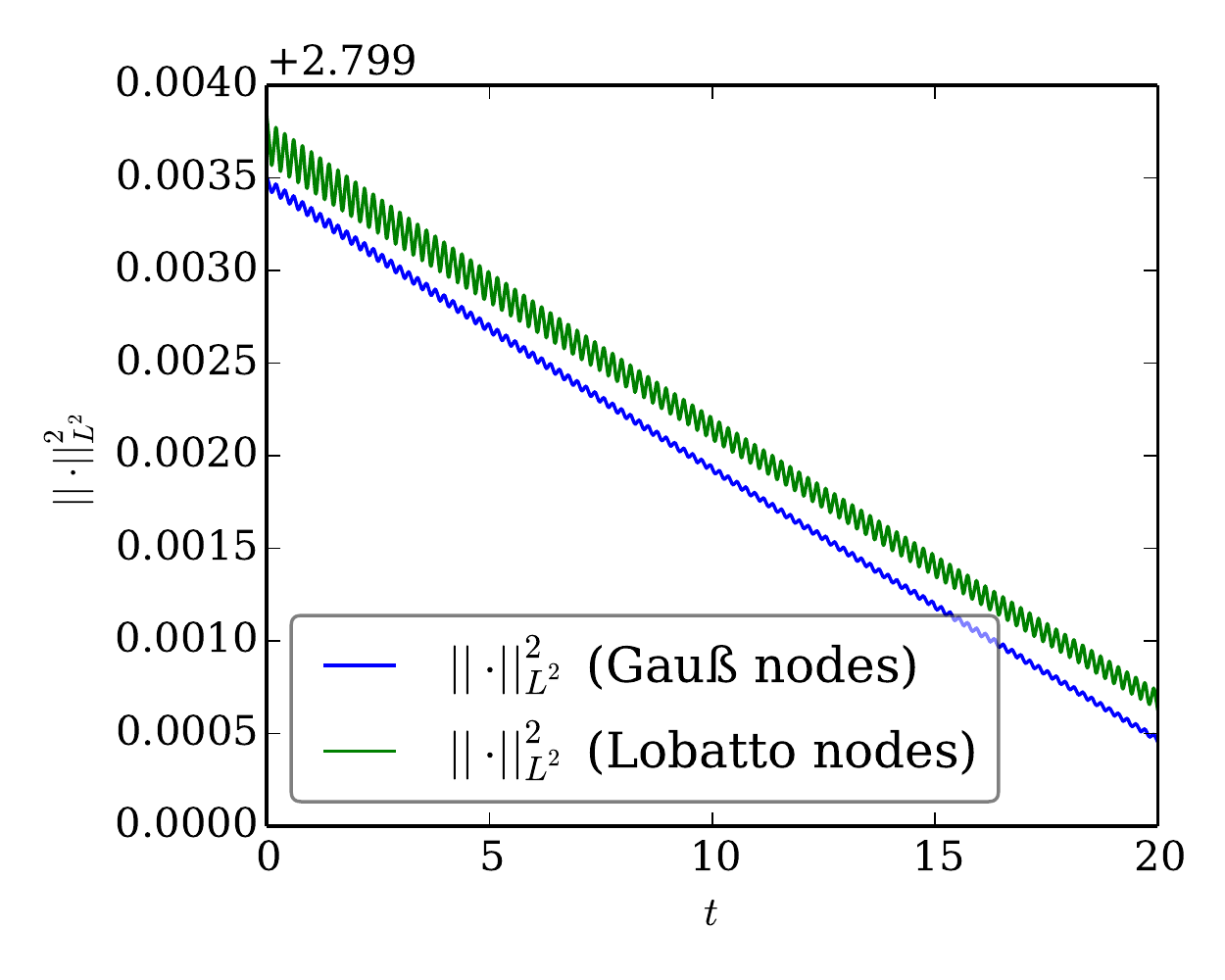}
    \caption{$c = c_- / 2$.}
  \end{subfigure}%
  ~
  \begin{subfigure}[b]{0.49\textwidth}
    \includegraphics[width=\textwidth]{%
      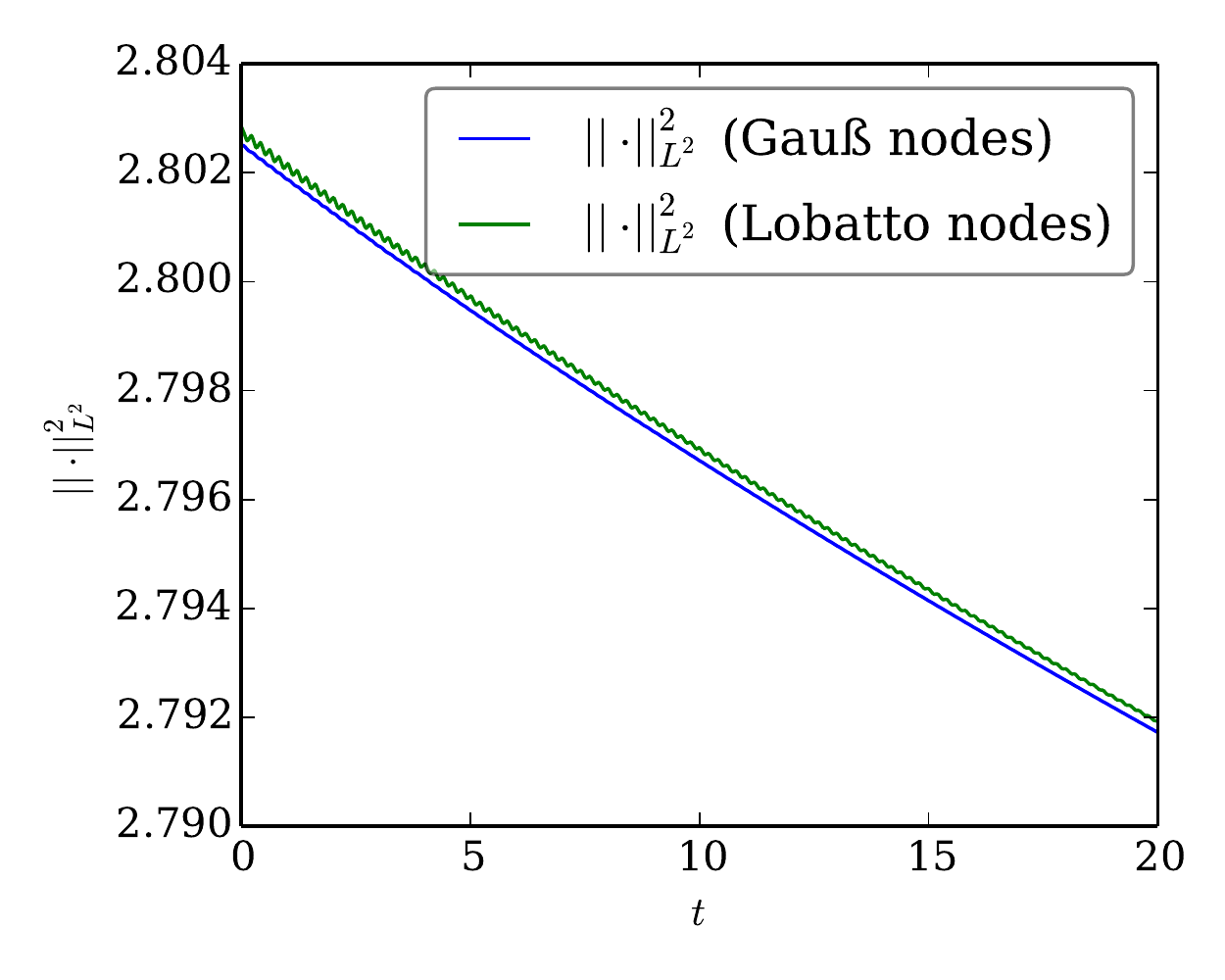}
    \caption{$c = c_0$.}
  \end{subfigure}%
  \\
  \begin{subfigure}[b]{0.49\textwidth}
    \includegraphics[width=\textwidth]{%
      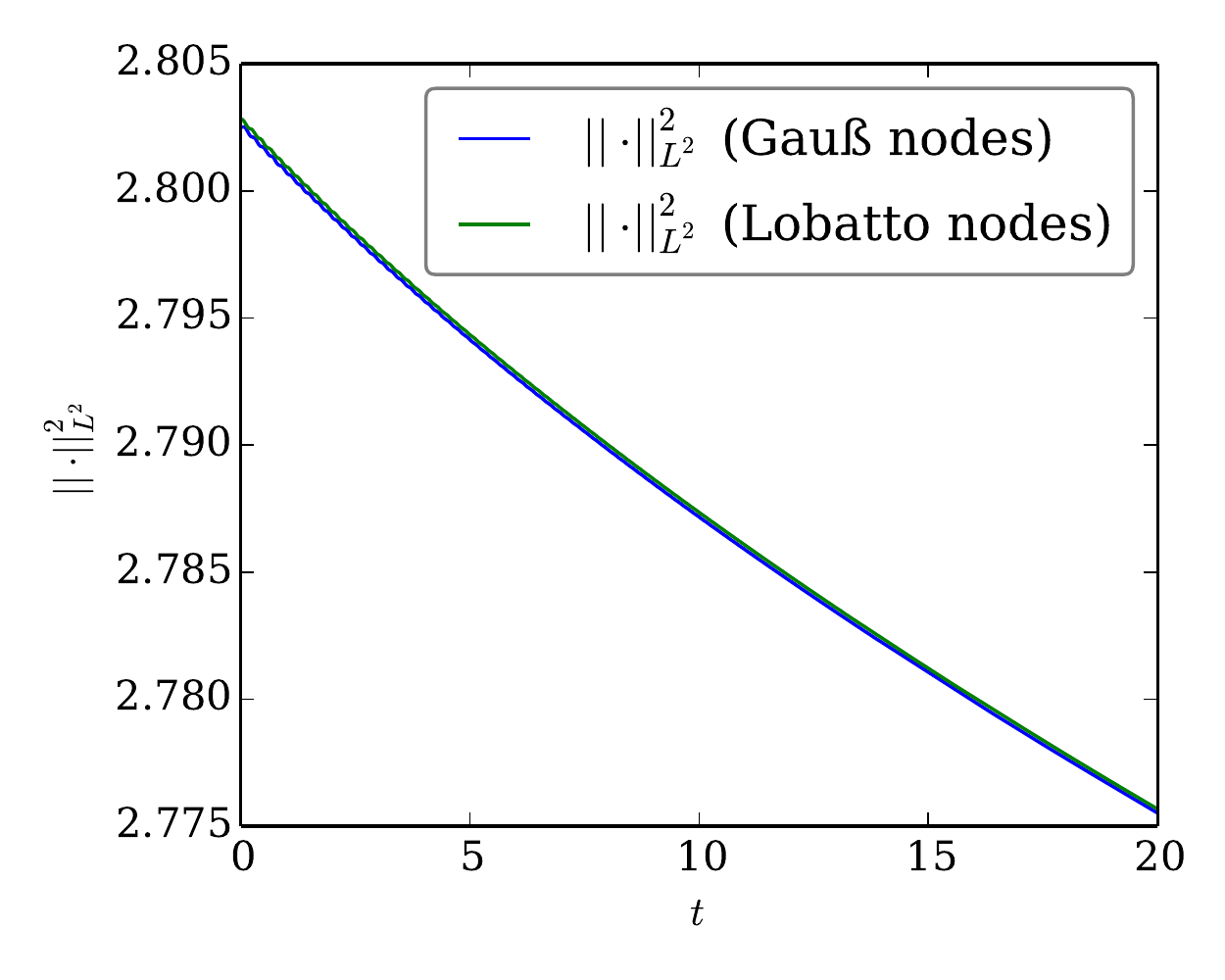}
    \caption{$c = c_{SD}$.}
  \end{subfigure}%
  ~
  \begin{subfigure}[b]{0.49\textwidth}
    \includegraphics[width=\textwidth]{%
      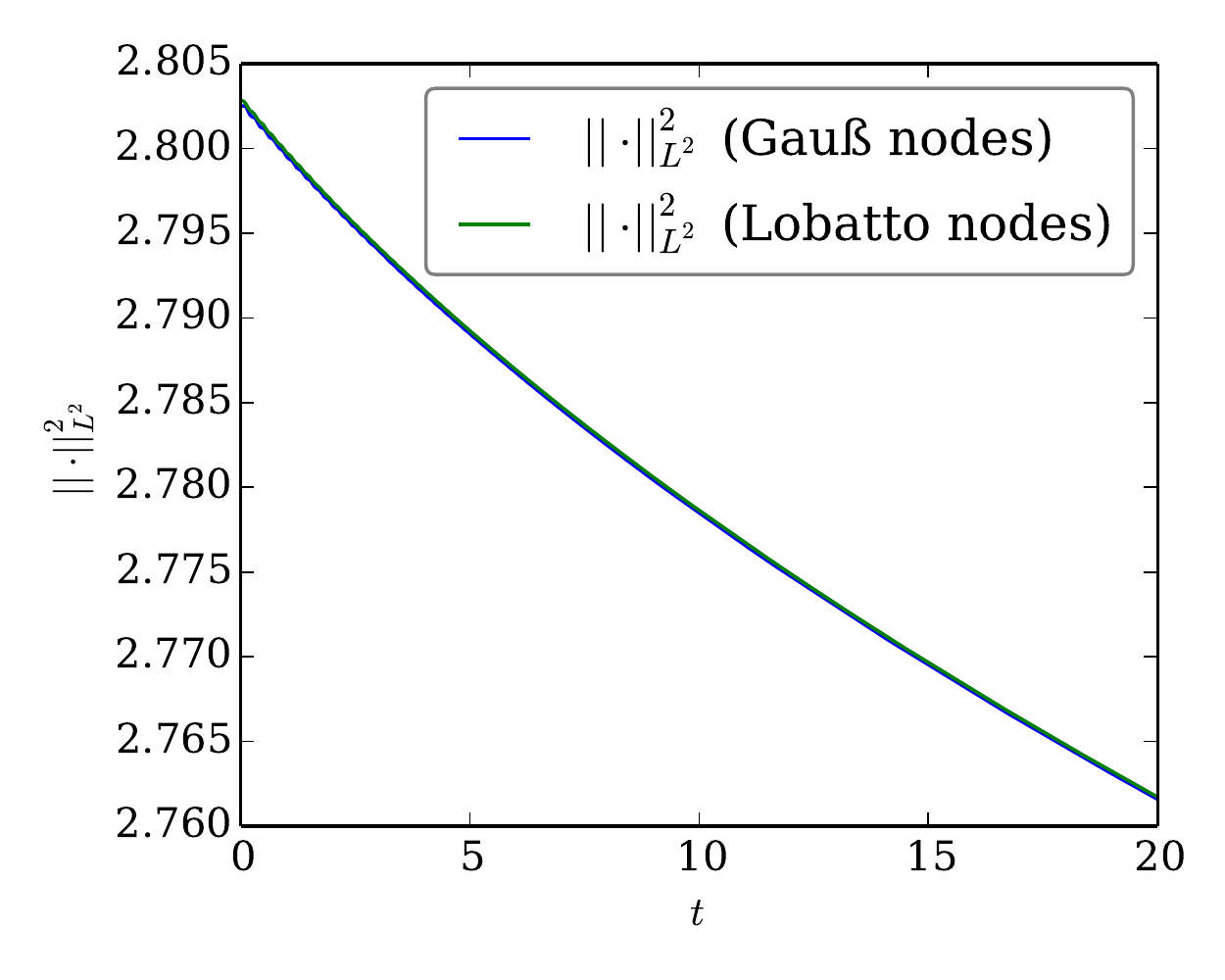}
    \caption{$c = c_{Hu}$.}
  \end{subfigure}%
  \caption{Energy of numerical solution of constant velocity linear
            advection using SBP CPR methods with $10$ elements, a Lobatto-Legendre
            basis of order $p = 3$ and a upwind numerical flux.
            Different values of $c$ are used for the correction matrix
            $\mat{C}$. The discrete energy $\norm{\vec{u}}^2$ is computed using
            Gauß-Legendre (blue) and Lobatto-Legendre (green) quadrature with
            $p + 1 = 4$ nodes in the full time interval $[0, 20]$.}
  \label{fig:vincent2011newclass-lobatto-upwind-2}
\end{figure}

\begin{figure}[!hp]
  \centering
  \begin{subfigure}[b]{0.49\textwidth}
    \includegraphics[width=\textwidth]{%
      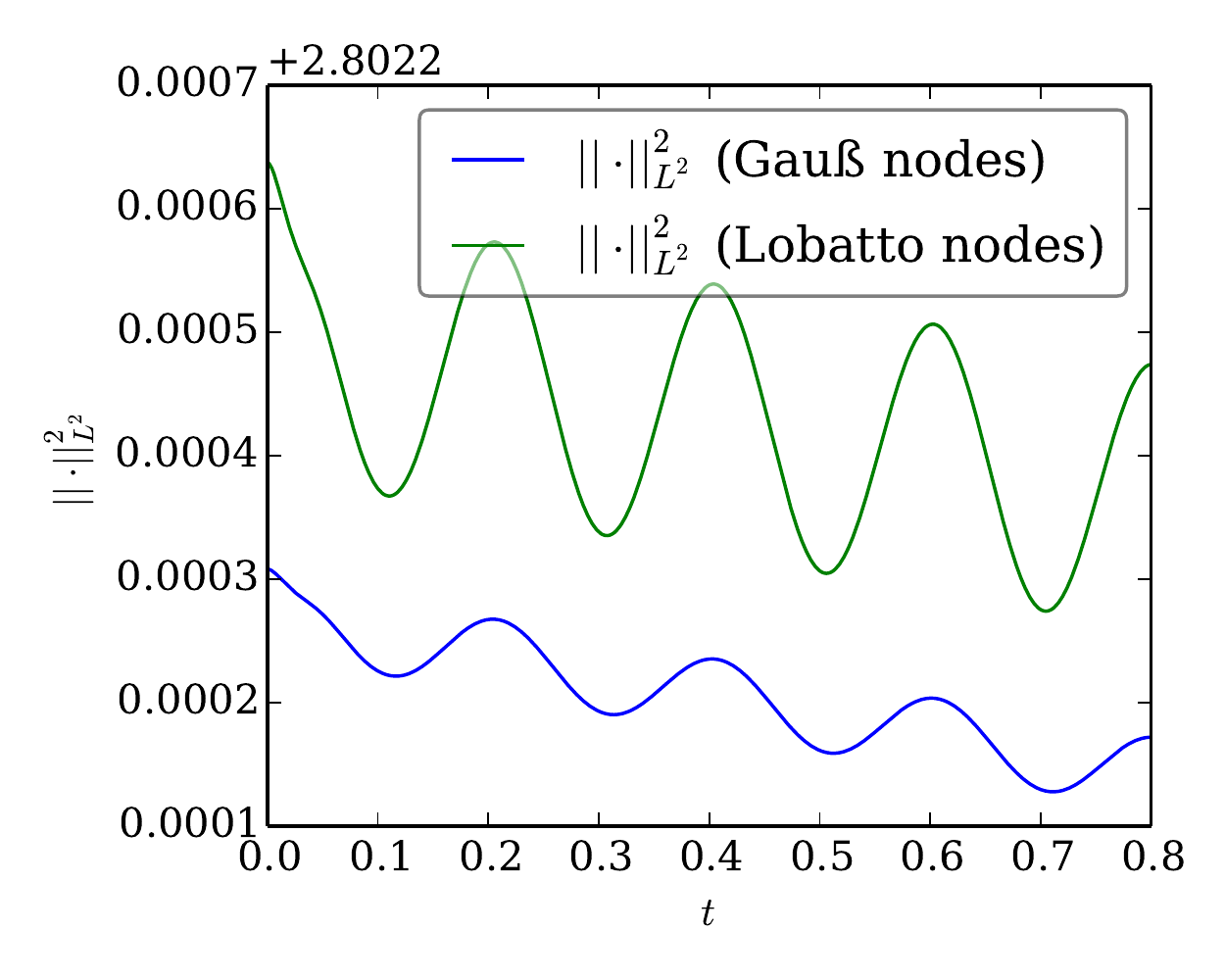}
    \caption{$c = c_- / 2$.}
  \end{subfigure}%
  ~
  \begin{subfigure}[b]{0.49\textwidth}
    \includegraphics[width=\textwidth]{%
      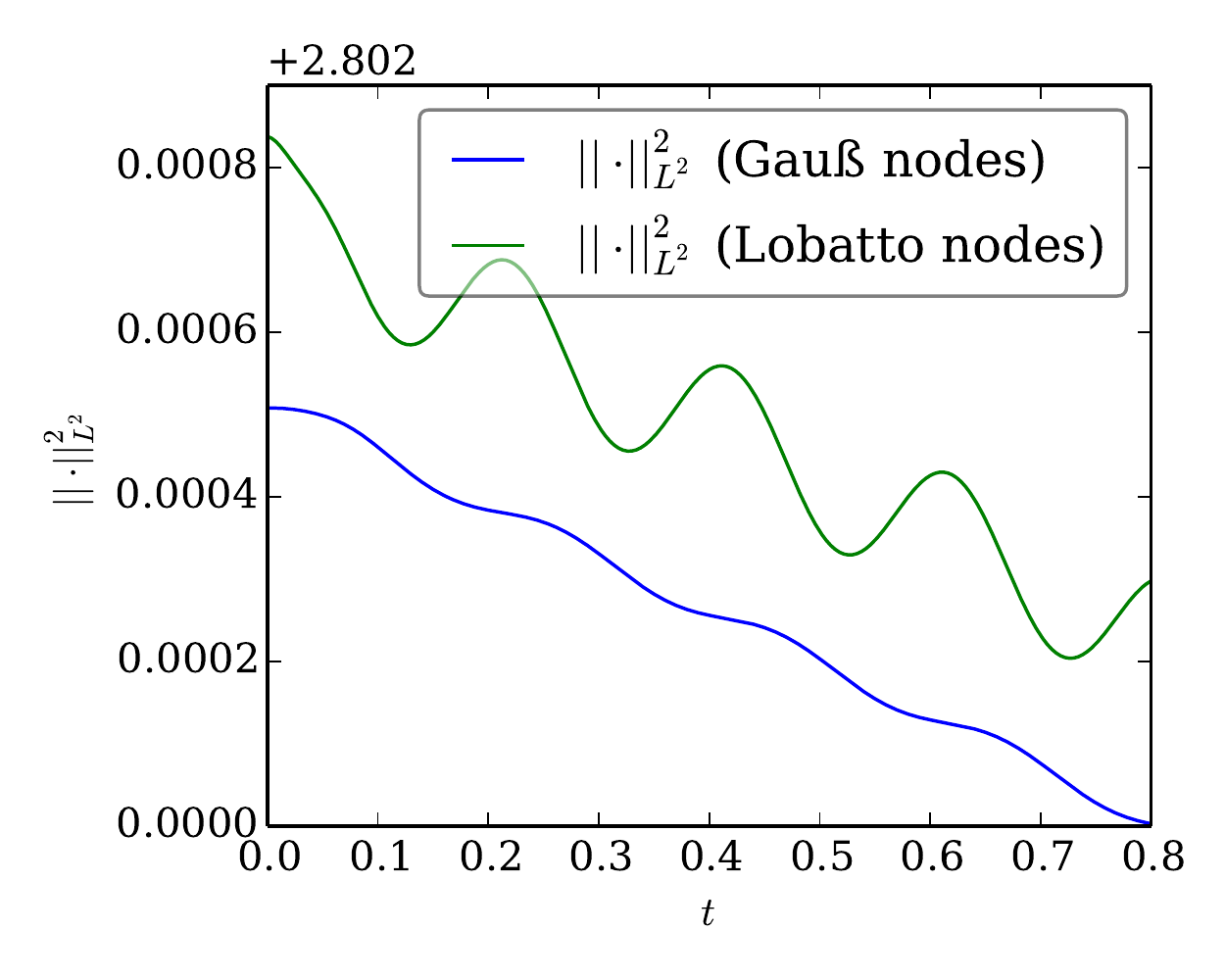}
    \caption{$c = c_0$.}
  \end{subfigure}%
  \\
  \begin{subfigure}[b]{0.49\textwidth}
    \includegraphics[width=\textwidth]{%
      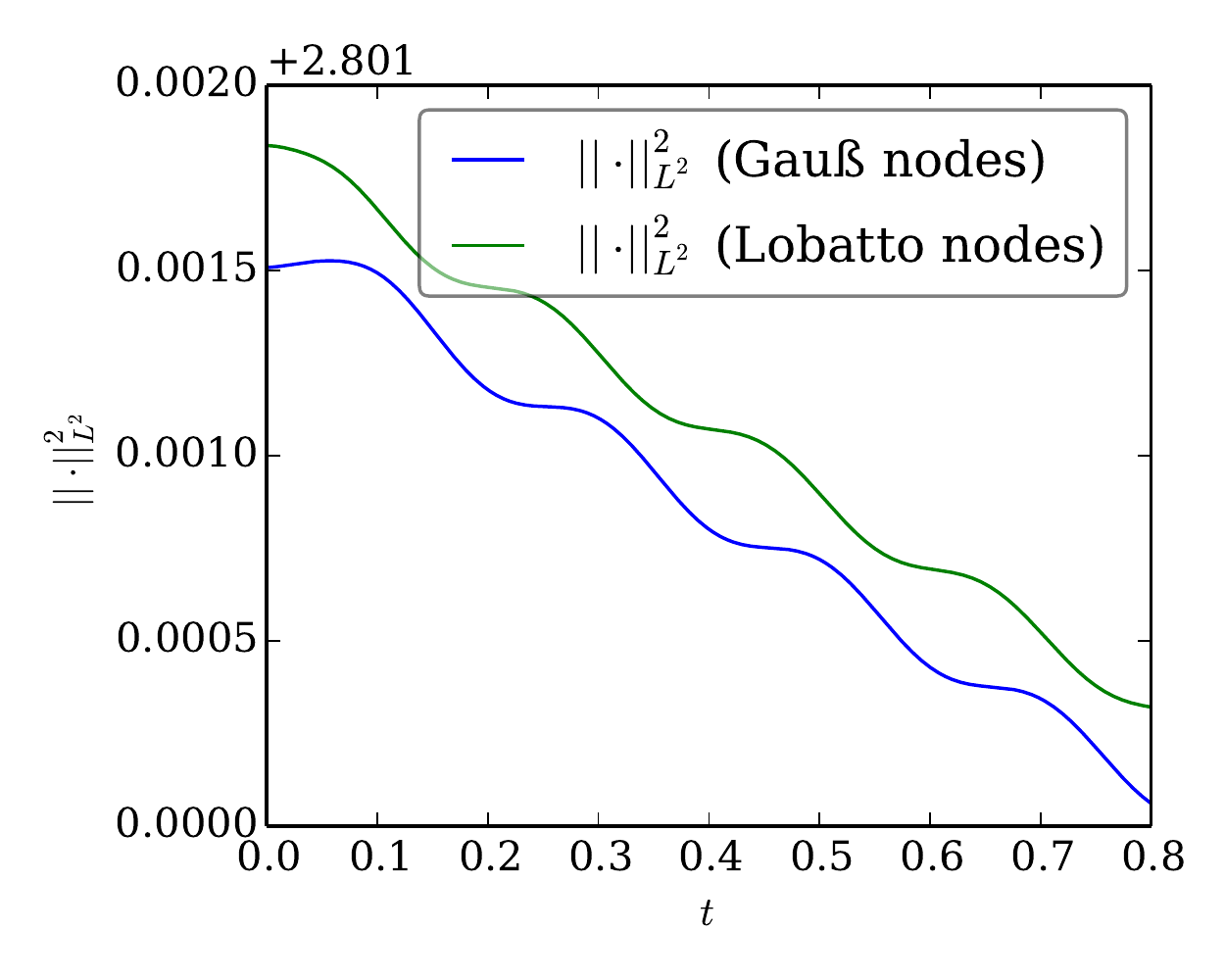}
    \caption{$c = c_{SD}$.}
  \end{subfigure}%
  ~
  \begin{subfigure}[b]{0.49\textwidth}
    \includegraphics[width=\textwidth]{%
      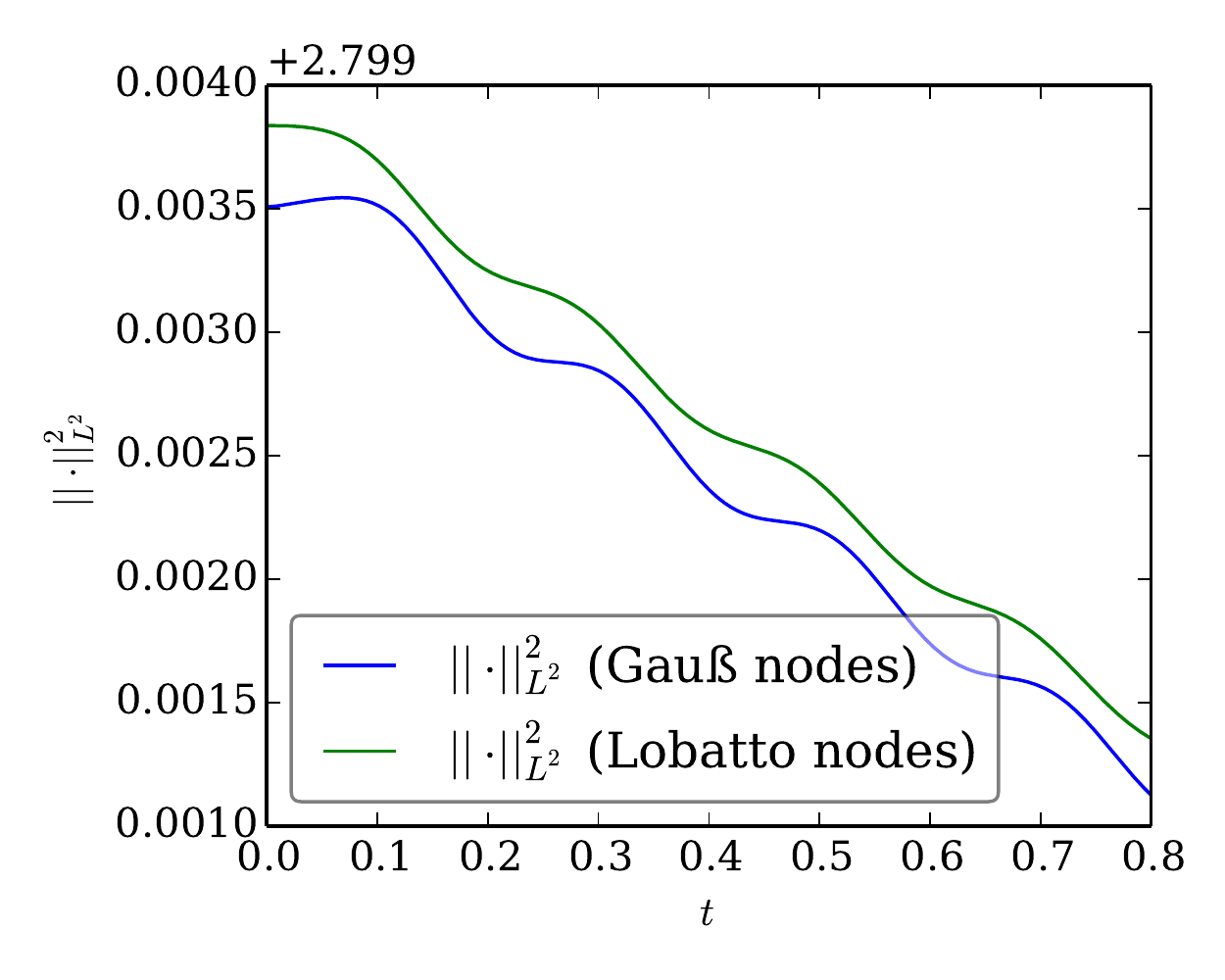}
    \caption{$c = c_{Hu}$.}
  \end{subfigure}%
  \caption{Energy of numerical solution of constant velocity linear
            advection using SBP CPR methods with $10$ elements, a Lobatto-Legendre
            basis of order $p = 3$ and a upwind numerical flux.
            Different values of $c$ are used for the correction matrix
            $\mat{C}$. The discrete energy $\norm{\vec{u}}^2$ is computed using
            Gauß-Legendre (blue) and Lobatto-Legendre (green) quadrature with
            $p + 1 = 4$ nodes in the time interval $[0, 0.8]$.}
  \label{fig:vincent2011newclass-lobatto-upwind-3}
\end{figure}

\begin{figure}[!hp]
  \centering
  \begin{subfigure}[b]{0.45\textwidth}
    \includegraphics[width=\textwidth]{%
      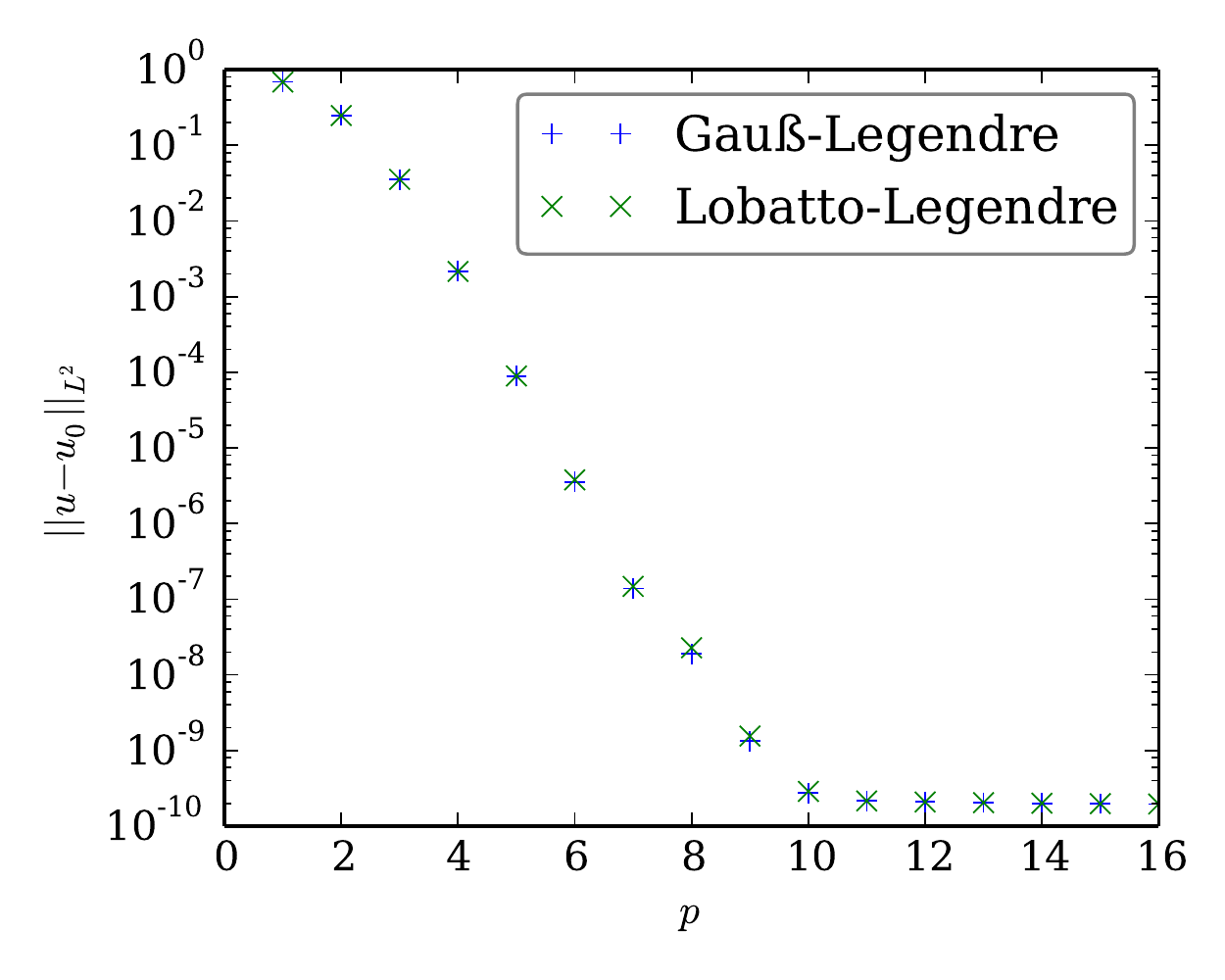}
    \caption{$c = c_- / 2$}
  \end{subfigure}%
  ~
  \begin{subfigure}[b]{0.45\textwidth}
    \includegraphics[width=\textwidth]{%
      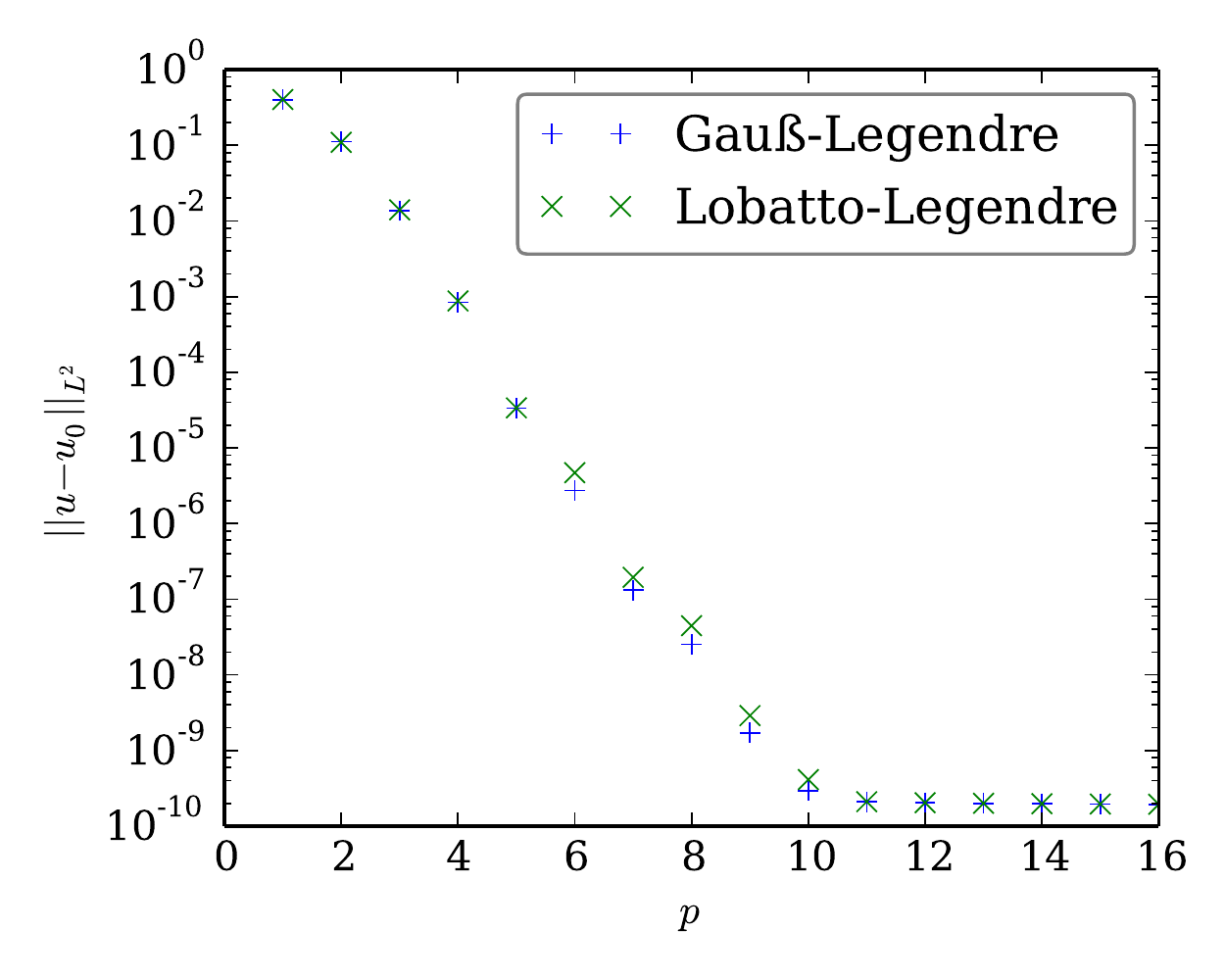}
    \caption{$c = c_0$}
  \end{subfigure}%
  ~\\
  \begin{subfigure}[b]{0.45\textwidth}
    \includegraphics[width=\textwidth]{%
      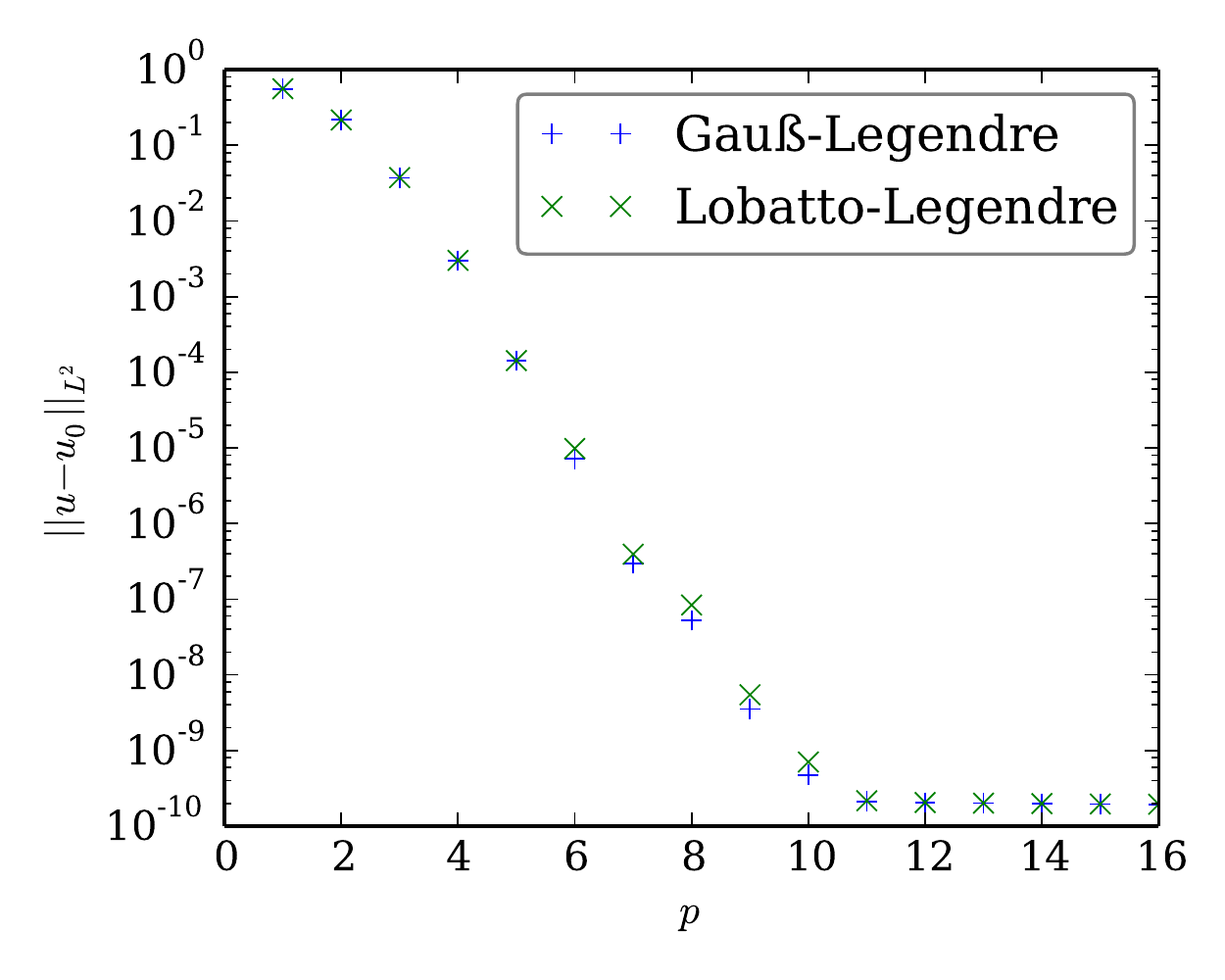}
    \caption{$c = c_{SD}$}
  \end{subfigure}%
  ~
  \begin{subfigure}[b]{0.45\textwidth}
    \includegraphics[width=\textwidth]{%
      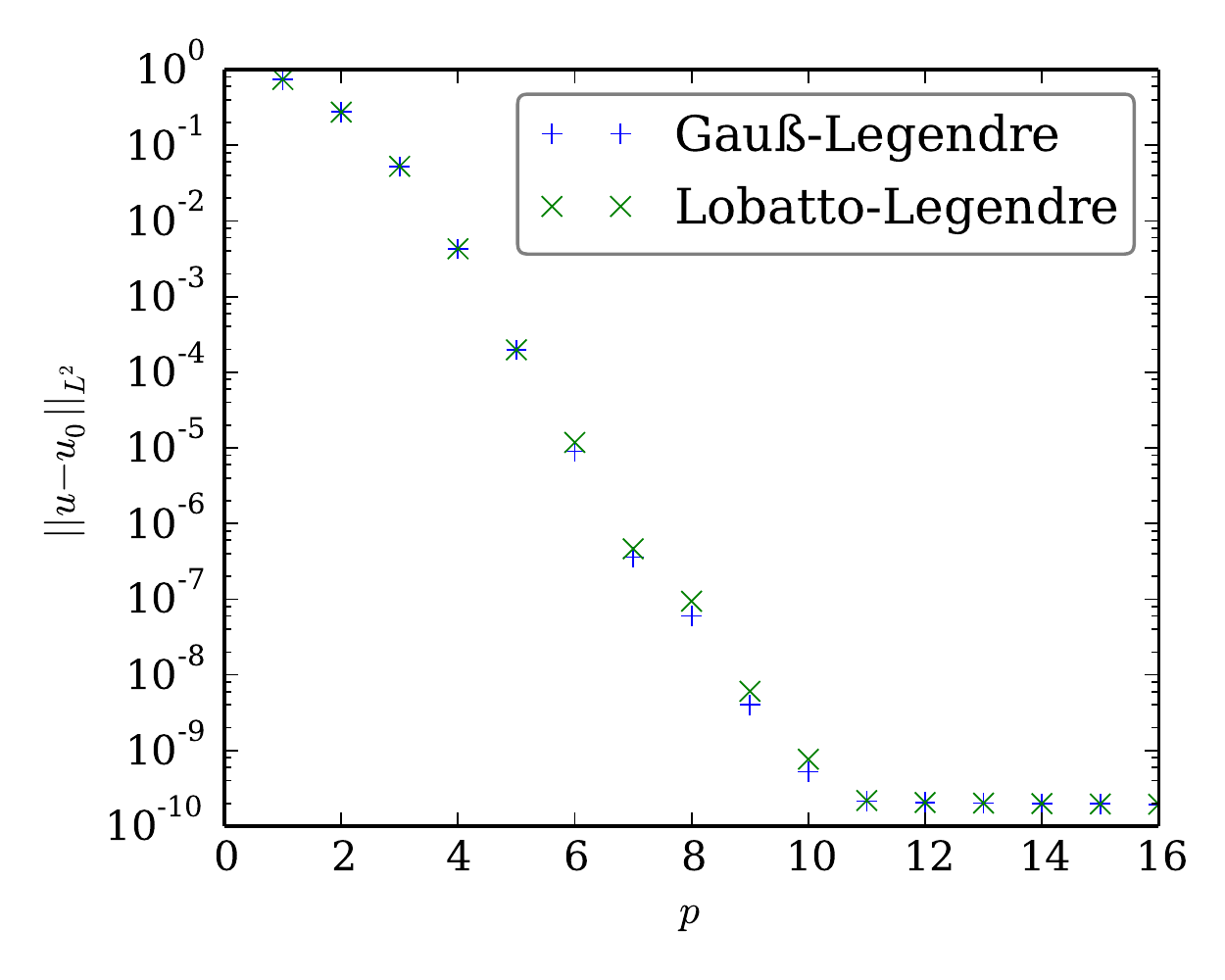}
    \caption{$c = c_{Hu}$}
  \end{subfigure}%
  ~\\
  \begin{subfigure}[b]{0.45\textwidth}
    \includegraphics[width=\textwidth]{%
      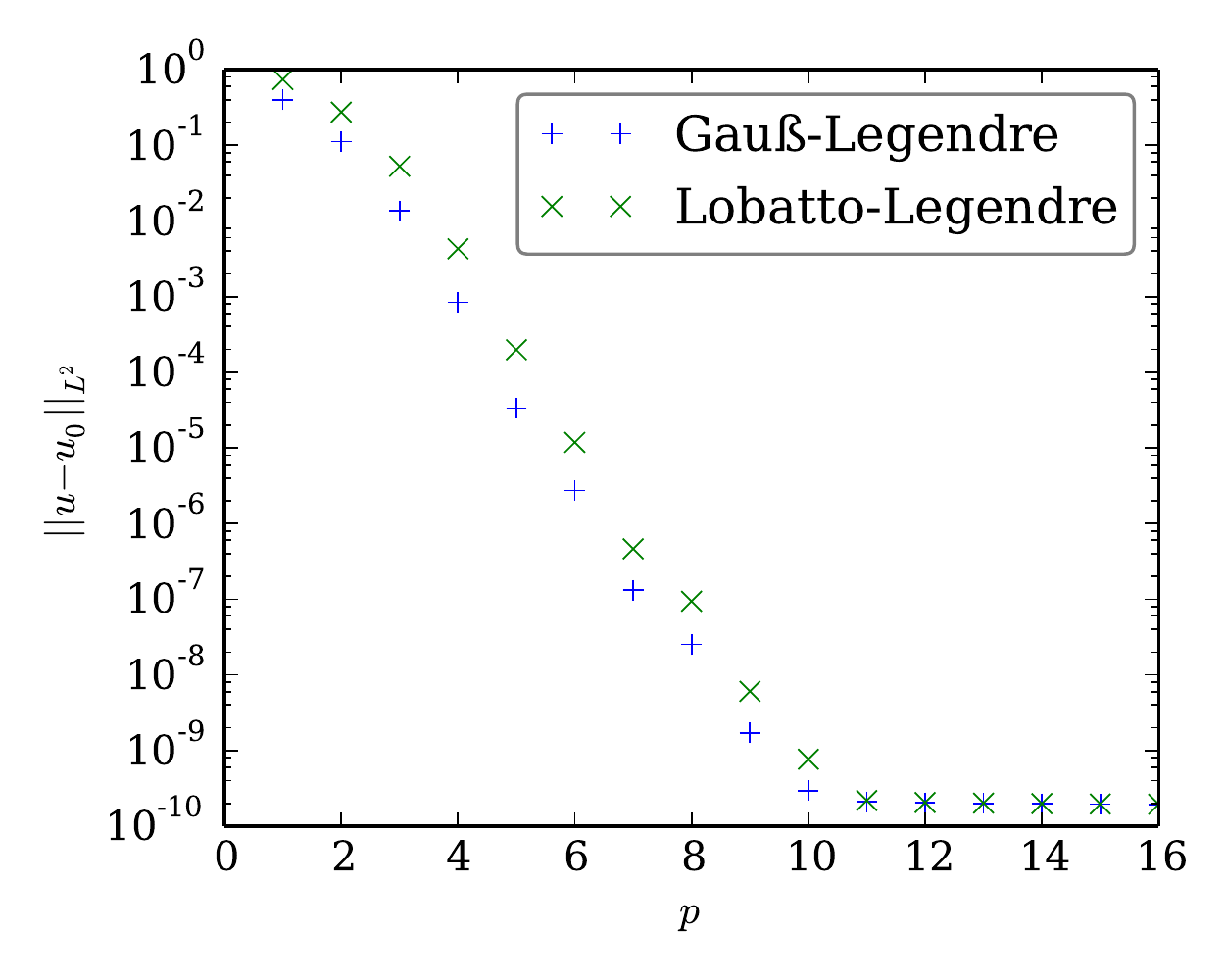}
    \caption{$\kappa = 0$}
  \end{subfigure}%
  \caption{$L^2$ errors of $u(20)$ for constant velocity linear advection
            using SBP CPR methods with $N = 10$ elements, a Gauß-Legendre and
            Lobatto-Legendre bases of varying order $p$ and an upwind numerical
            flux. Different values of $c$ are used for the correction matrix
            $\mat{C}$.}
  \label{fig:vincent2011newclass-p-10}
\end{figure}

\begin{sidewaystable}[!hp]
  \caption{$L^2$ errors of $u(20)$ for constant velocity linear advection
            using SBP CPR methods with $N = 10$ elements, Gauß-Legendre and
            Lobatto-Legendre bases of varying order $p$ and an upwind numerical
            flux. Different values of $c$ are used for the correction matrix
            $\mat{C}$. Numerical values corresponding to Figure
            \ref{fig:vincent2011newclass-p-10}.}
  \label{tab:vincent2011newclass-p-10}
  \centering
  \begin{tabular}{ccccccccc}
  \toprule
  & \multicolumn{8}{c}{$\norm{u - u_0}_{L^2}$}
  \\
  \cmidrule(r){2-9}
  & \multicolumn{2}{c}{$c = c_-/2$} & \multicolumn{2}{c}{$c = c_0$} & \multicolumn{2}{c}{$c = c_{SD}$} & \multicolumn{2}{c}{$c = c_{Hu}$}
  \\
  \cmidrule(r){2-3} \cmidrule(r){4-5} \cmidrule(r){6-7} \cmidrule(r){8-9}
  $p$ & Gauß & Lobatto & Gauß & Lobatto & Gauß & Lobatto & Gauß & Lobatto
  \\
  \midrule
  \num{1} & \num{6.88e-01} & \num{6.86e-01} & \num{4.00e-01} & \num{4.03e-01} & \num{5.55e-01} & \num{5.59e-01} & \num{7.40e-01} & \num{7.41e-01} \\
  \num{2} & \num{2.49e-01} & \num{2.47e-01} & \num{1.12e-01} & \num{1.09e-01} & \num{2.17e-01} & \num{2.16e-01} & \num{2.75e-01} & \num{2.74e-01} \\
  \num{3} & \num{3.51e-02} & \num{3.55e-02} & \num{1.36e-02} & \num{1.40e-02} & \num{3.70e-02} & \num{3.74e-02} & \num{5.22e-02} & \num{5.27e-02} \\
  \num{4} & \num{2.16e-03} & \num{2.15e-03} & \num{8.38e-04} & \num{8.75e-04} & \num{2.98e-03} & \num{3.01e-03} & \num{4.26e-03} & \num{4.28e-03} \\
  \num{5} & \num{8.92e-05} & \num{8.93e-05} & \num{3.36e-05} & \num{3.38e-05} & \num{1.42e-04} & \num{1.42e-04} & \num{1.98e-04} & \num{1.98e-04} \\
  \num{6} & \num{3.56e-06} & \num{3.77e-06} & \num{2.73e-06} & \num{4.71e-06} & \num{7.16e-06} & \num{9.81e-06} & \num{9.04e-06} & \num{1.18e-05} \\
  \num{7} & \num{1.39e-07} & \num{1.48e-07} & \num{1.32e-07} & \num{1.95e-07} & \num{2.97e-07} & \num{3.93e-07} & \num{3.61e-07} & \num{4.64e-07} \\
  \num{8} & \num{1.90e-08} & \num{2.28e-08} & \num{2.52e-08} & \num{4.46e-08} & \num{5.25e-08} & \num{8.38e-08} & \num{6.03e-08} & \num{9.40e-08} \\
  \num{9} & \num{1.33e-09} & \num{1.55e-09} & \num{1.71e-09} & \num{2.89e-09} & \num{3.53e-09} & \num{5.45e-09} & \num{4.00e-09} & \num{6.05e-09} \\
  \num{10} & \num{2.73e-10} & \num{2.88e-10} & \num{2.95e-10} & \num{4.13e-10} & \num{4.77e-10} & \num{7.05e-10} & \num{5.23e-10} & \num{7.68e-10} \\
  \num{11} & \num{2.15e-10} & \num{2.15e-10} & \num{2.09e-10} & \num{2.10e-10} & \num{2.13e-10} & \num{2.17e-10} & \num{2.14e-10} & \num{2.18e-10} \\
  \num{12} & \num{2.09e-10} & \num{2.09e-10} & \num{2.04e-10} & \num{2.04e-10} & \num{2.05e-10} & \num{2.05e-10} & \num{2.05e-10} & \num{2.06e-10} \\
  \num{13} & \num{2.04e-10} & \num{2.04e-10} & \num{2.01e-10} & \num{2.01e-10} & \num{2.01e-10} & \num{2.01e-10} & \num{2.02e-10} & \num{2.02e-10} \\
  \num{14} & \num{2.01e-10} & \num{2.01e-10} & \num{1.98e-10} & \num{1.98e-10} & \num{1.99e-10} & \num{1.99e-10} & \num{1.99e-10} & \num{1.99e-10} \\
  \num{15} & \num{1.98e-10} & \num{1.98e-10} & \num{1.96e-10} & \num{1.96e-10} & \num{1.97e-10} & \num{1.97e-10} & \num{1.97e-10} & \num{1.97e-10} \\
  \num{16} & \num{1.96e-10} & \num{1.96e-10} & \num{1.95e-10} & \num{1.95e-10} & \num{1.95e-10} & \num{1.95e-10} & \num{1.95e-10} & \num{1.95e-10} \\
  \bottomrule
  \end{tabular}
\end{sidewaystable}

\begin{figure}[!hp]
  \centering
  \begin{subfigure}[b]{0.45\textwidth}
    \includegraphics[width=\textwidth]{%
      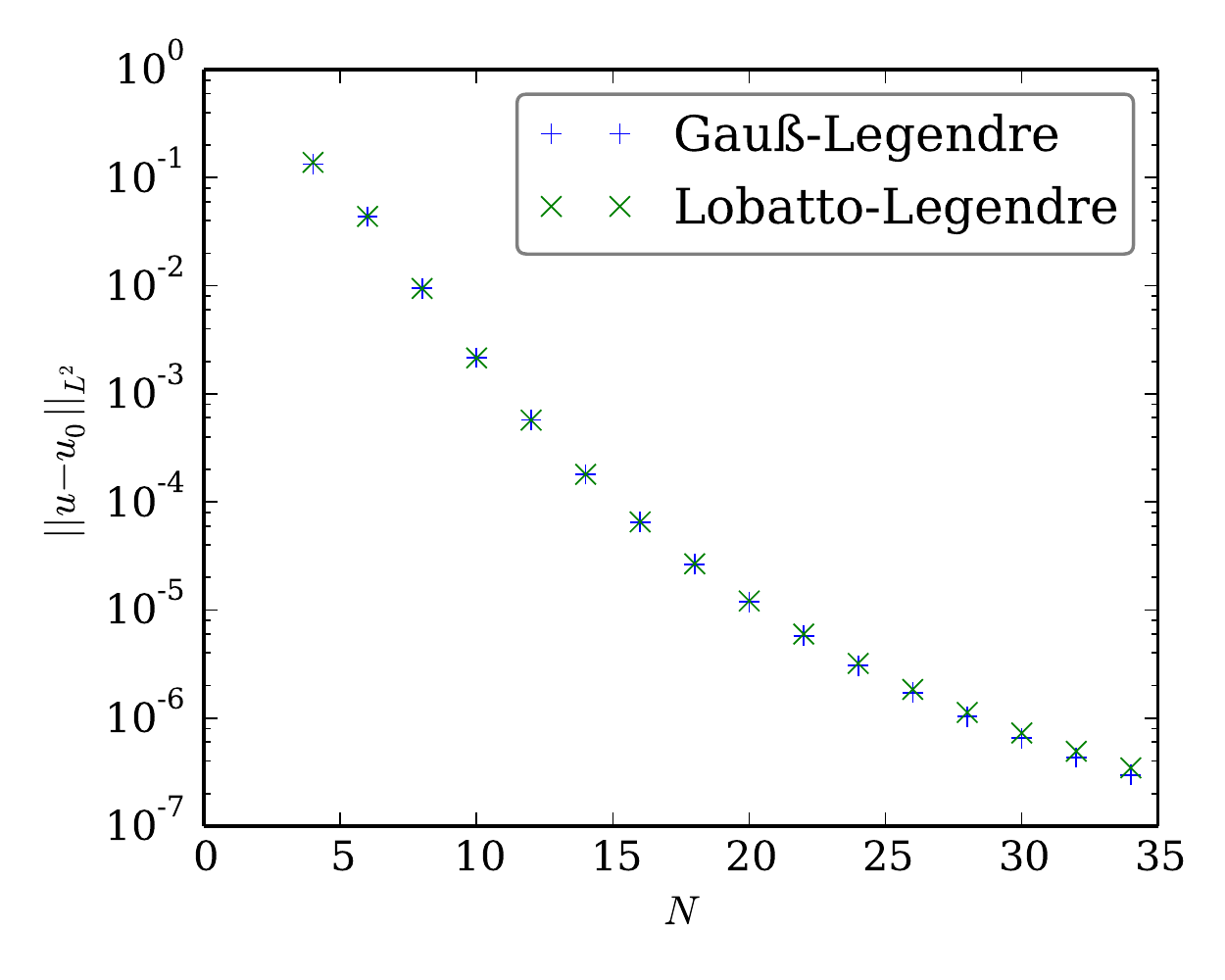}
    \caption{$c = c_- / 2$}
  \end{subfigure}%
  ~
  \begin{subfigure}[b]{0.45\textwidth}
    \includegraphics[width=\textwidth]{%
      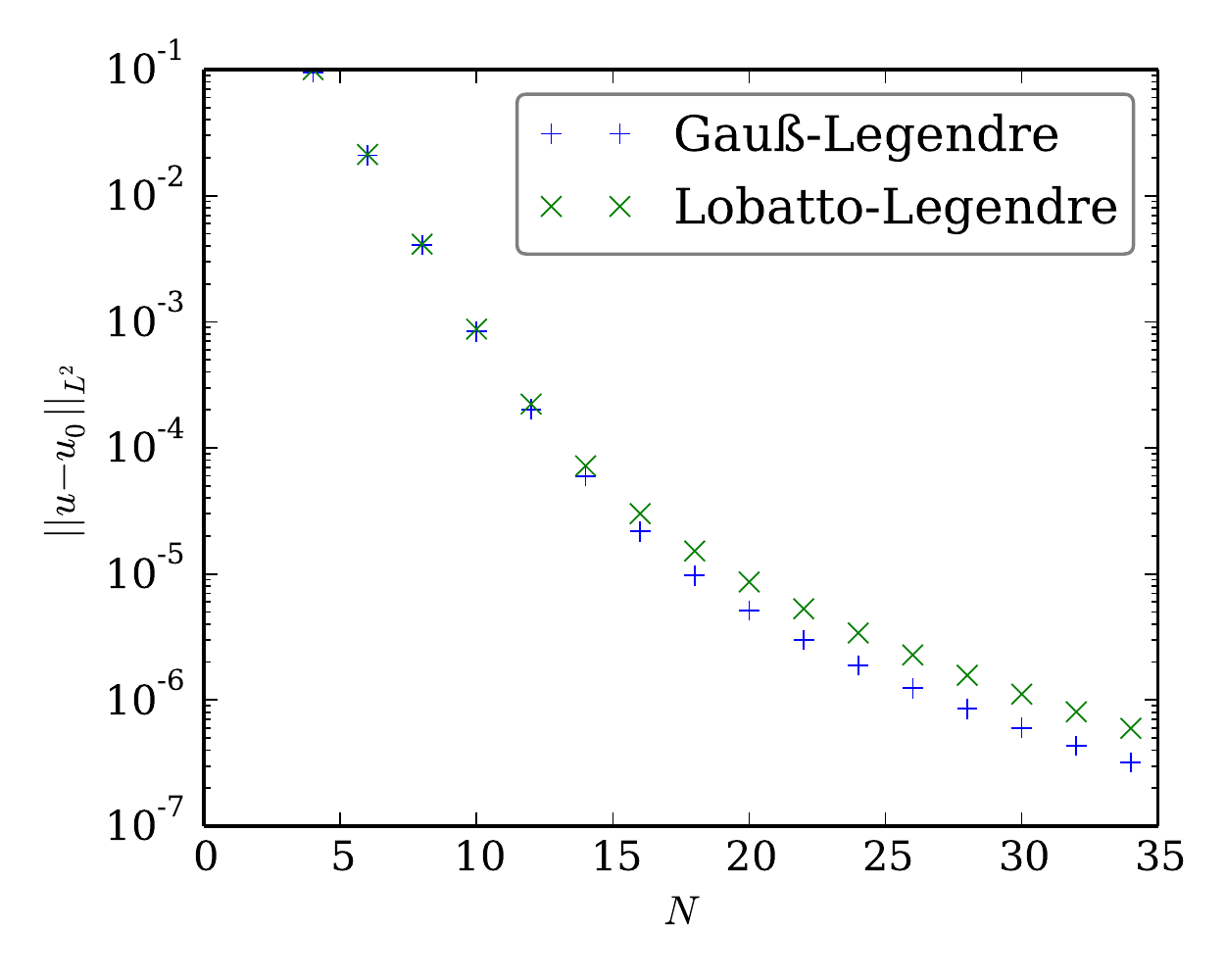}
    \caption{$c = c_0$}
  \end{subfigure}%
  ~\\
  \begin{subfigure}[b]{0.45\textwidth}
    \includegraphics[width=\textwidth]{%
      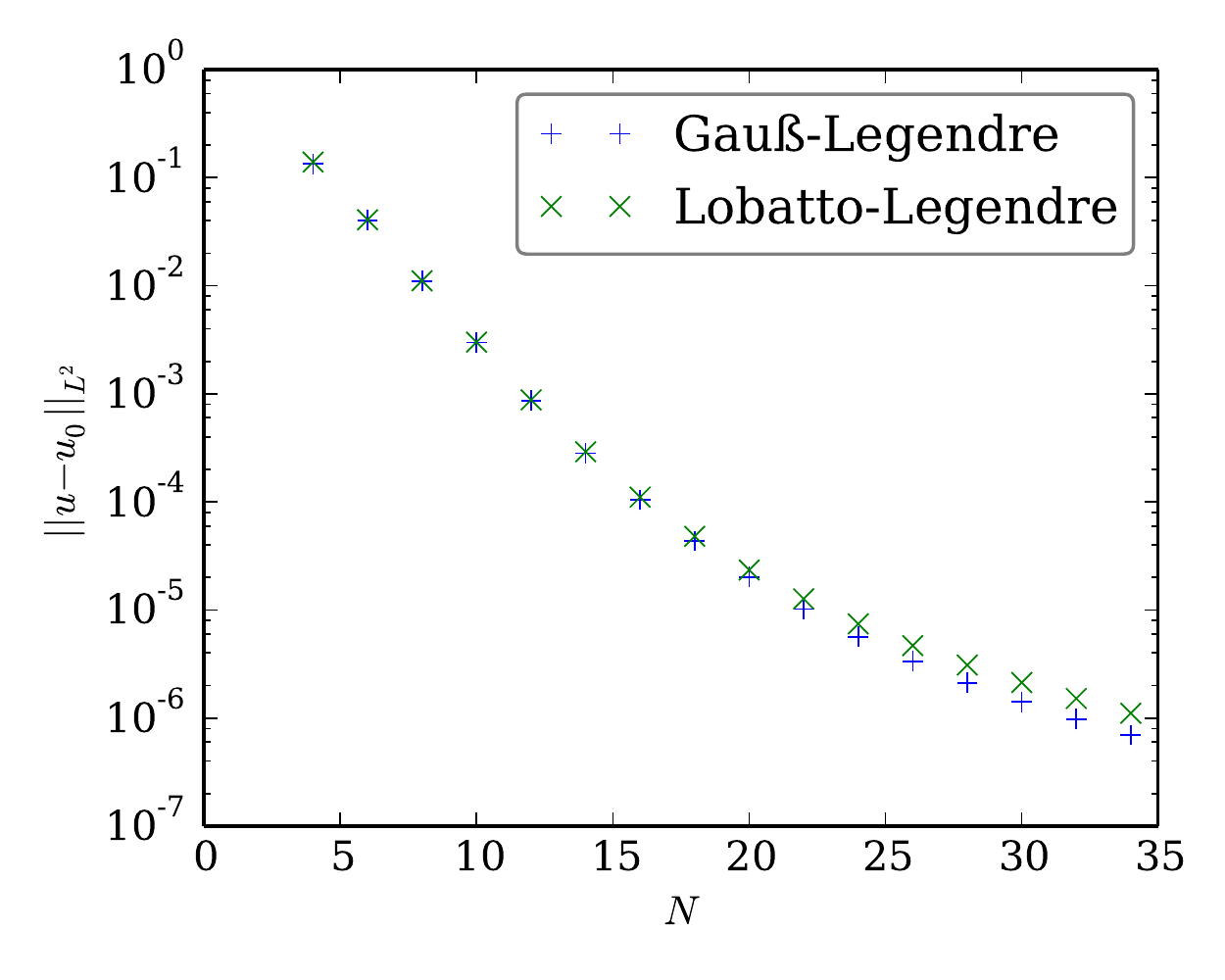}
    \caption{$c = c_{SD}$}
  \end{subfigure}%
  ~
  \begin{subfigure}[b]{0.45\textwidth}
    \includegraphics[width=\textwidth]{%
      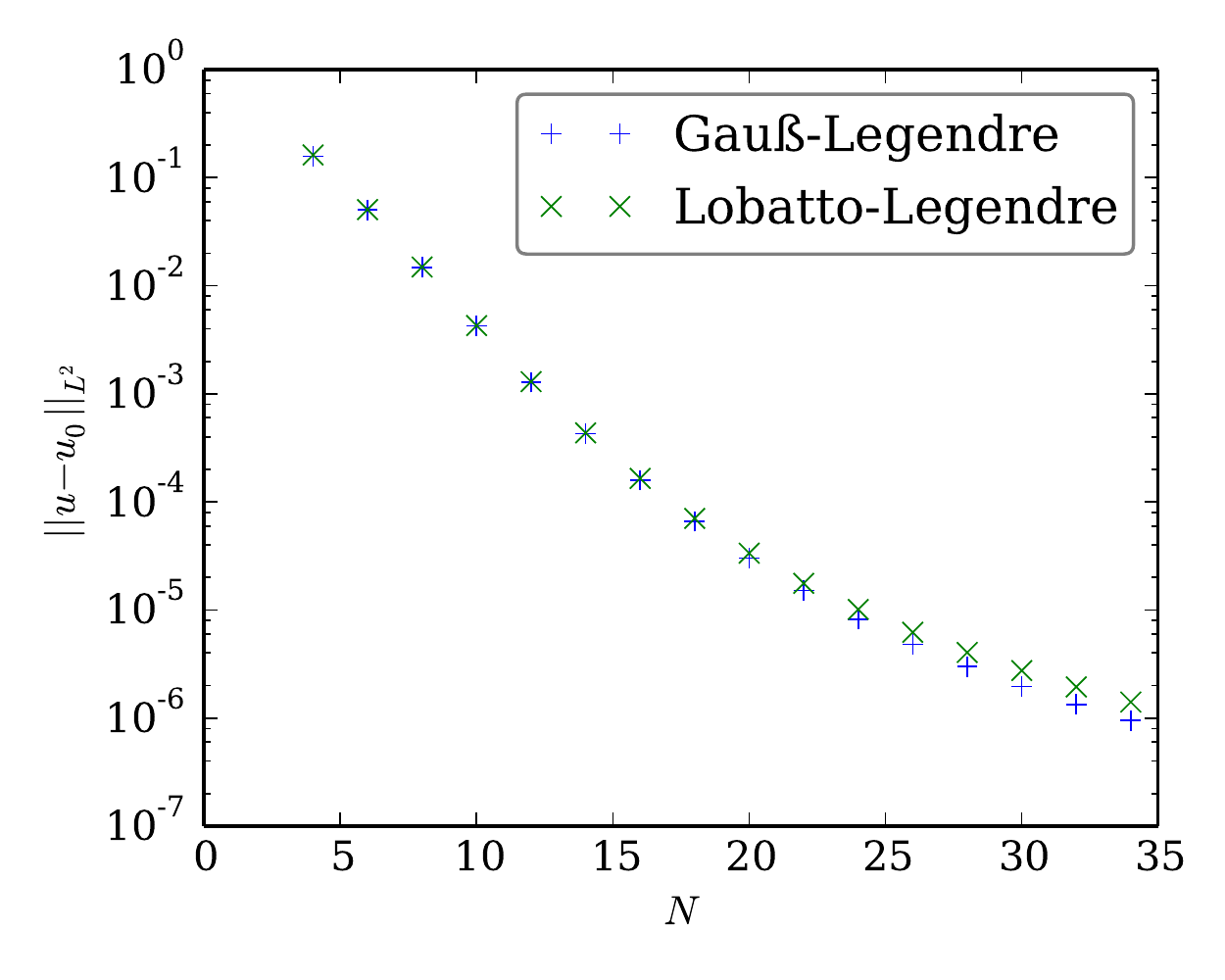}
    \caption{$c = c_{Hu}$}
  \end{subfigure}%
  ~\\
  \begin{subfigure}[b]{0.45\textwidth}
    \includegraphics[width=\textwidth]{%
      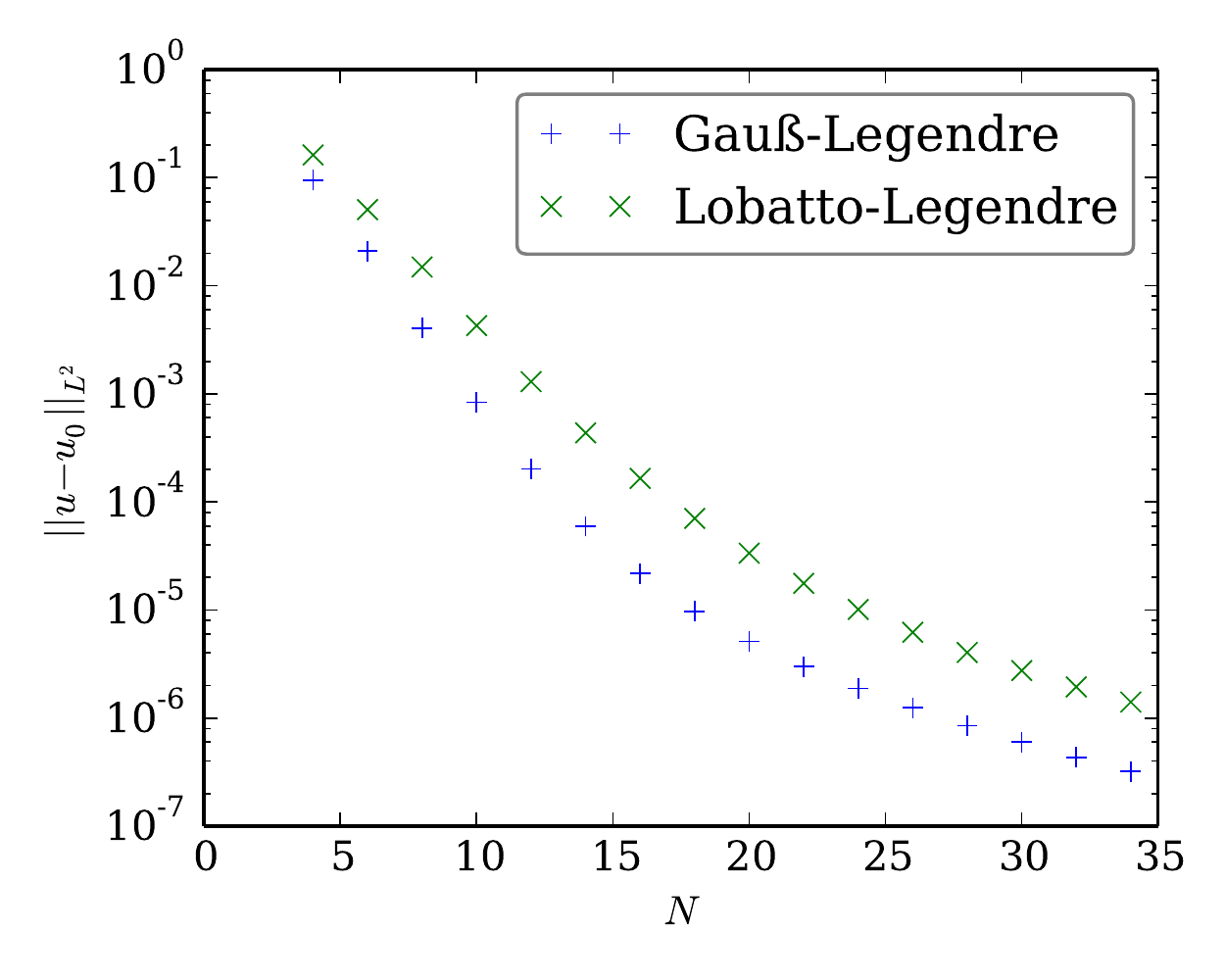}
    \caption{$\kappa = 0$}
  \end{subfigure}%
  \caption{$L^2$ errors of $u(20)$ for constant velocity linear advection
            using SBP CPR methods with varying number of elements $N$,
            Gauß-Legendre and Lobatto-Legendre bases of order $p = 4$ and an
            upwind numerical flux. Different values of $c$ are used for the
            correction matrix $\mat{C}$.}
  \label{fig:vincent2011newclass-N-4}
\end{figure}

\begin{sidewaystable}[!hp]
  \caption{$L^2$ errors of $u(20)$ for constant velocity linear advection
            using SBP CPR methods with varying number of elements $N$,
            Gauß-Legendre and Lobatto-Legendre bases of order $p = 4$ and an
            upwind numerical flux. Different values of $c$ are used for the
            correction matrix $\mat{C}$. Numerical values corresponding to Figure
            \ref{fig:vincent2011newclass-N-4}.}
  \label{tab:vincent2011newclass-N-4}
  \centering
  \begin{tabular}{ccccccccc}
  \toprule
  & \multicolumn{8}{c}{$\norm{u - u_0}_{L^2}$}
  \\
  \cmidrule(r){2-9}
  & \multicolumn{2}{c}{$c = c_-/2$} & \multicolumn{2}{c}{$c = c_0$} & \multicolumn{2}{c}{$c = c_{SD}$} & \multicolumn{2}{c}{$c = c_{Hu}$}
  \\
  \cmidrule(r){2-3} \cmidrule(r){4-5} \cmidrule(r){6-7} \cmidrule(r){8-9}
  $N$ & Gauß & Lobatto & Gauß & Lobatto & Gauß & Lobatto & Gauß & Lobatto
  \\
  \midrule
  \num{4} & \num{1.34e-01} & \num{1.39e-01} & \num{9.48e-02} & \num{9.99e-02} & \num{1.35e-01} & \num{1.40e-01} & \num{1.57e-01} & \num{1.62e-01} \\
  \num{6} & \num{4.37e-02} & \num{4.38e-02} & \num{2.09e-02} & \num{2.12e-02} & \num{4.05e-02} & \num{4.08e-02} & \num{5.04e-02} & \num{5.07e-02} \\
  \num{8} & \num{9.46e-03} & \num{9.44e-03} & \num{4.06e-03} & \num{4.13e-03} & \num{1.11e-02} & \num{1.11e-02} & \num{1.49e-02} & \num{1.49e-02} \\
  \num{10} & \num{2.16e-03} & \num{2.15e-03} & \num{8.38e-04} & \num{8.75e-04} & \num{2.98e-03} & \num{3.01e-03} & \num{4.26e-03} & \num{4.28e-03} \\
  \num{12} & \num{5.73e-04} & \num{5.71e-04} & \num{2.02e-04} & \num{2.23e-04} & \num{8.65e-04} & \num{8.80e-04} & \num{1.28e-03} & \num{1.30e-03} \\
  \num{14} & \num{1.80e-04} & \num{1.80e-04} & \num{5.95e-05} & \num{7.20e-05} & \num{2.83e-04} & \num{2.91e-04} & \num{4.27e-04} & \num{4.35e-04} \\
  \num{16} & \num{6.50e-05} & \num{6.54e-05} & \num{2.18e-05} & \num{3.01e-05} & \num{1.05e-04} & \num{1.11e-04} & \num{1.59e-04} & \num{1.65e-04} \\
  \num{18} & \num{2.64e-05} & \num{2.68e-05} & \num{9.75e-06} & \num{1.52e-05} & \num{4.36e-05} & \num{4.80e-05} & \num{6.60e-05} & \num{7.04e-05} \\
  \num{20} & \num{1.19e-05} & \num{1.21e-05} & \num{5.12e-06} & \num{8.64e-06} & \num{2.01e-05} & \num{2.34e-05} & \num{3.03e-05} & \num{3.35e-05} \\
  \num{22} & \num{5.79e-06} & \num{5.99e-06} & \num{3.00e-06} & \num{5.29e-06} & \num{1.02e-05} & \num{1.27e-05} & \num{1.52e-05} & \num{1.77e-05} \\
  \num{24} & \num{3.05e-06} & \num{3.20e-06} & \num{1.88e-06} & \num{3.41e-06} & \num{5.65e-06} & \num{7.45e-06} & \num{8.24e-06} & \num{1.01e-05} \\
  \num{26} & \num{1.72e-06} & \num{1.84e-06} & \num{1.24e-06} & \num{2.28e-06} & \num{3.36e-06} & \num{4.68e-06} & \num{4.81e-06} & \num{6.22e-06} \\
  \num{28} & \num{1.03e-06} & \num{1.13e-06} & \num{8.51e-07} & \num{1.57e-06} & \num{2.12e-06} & \num{3.10e-06} & \num{2.99e-06} & \num{4.05e-06} \\
  \num{30} & \num{6.51e-07} & \num{7.29e-07} & \num{6.00e-07} & \num{1.11e-06} & \num{1.41e-06} & \num{2.13e-06} & \num{1.96e-06} & \num{2.75e-06} \\
  \num{32} & \num{4.32e-07} & \num{4.93e-07} & \num{4.34e-07} & \num{8.08e-07} & \num{9.76e-07} & \num{1.52e-06} & \num{1.34e-06} & \num{1.94e-06} \\
  \num{34} & \num{2.98e-07} & \num{3.46e-07} & \num{3.20e-07} & \num{5.97e-07} & \num{6.98e-07} & \num{1.11e-06} & \num{9.49e-07} & \num{1.41e-06} \\
  \bottomrule
  \end{tabular}
\end{sidewaystable}

\subsection{Influence of time discretisation}

Leaving the mathematical paradise of the previous sections and entering the
real world of numerical methods, time discretisation plays an important role.
For simplicity, an explicit Euler method for an SBP CPR
semidiscretisation of the linear advection equation
\eqref{eq:linear-constant-advection} is considered. Thus, one step of size
$\Delta t$ from $\vec{u}$ to $\vec{u}^+$ (in the standard element) can be
written as
\begin{equation}
 \vec{u}^+ = \vec{u} + \Delta t \, \partial_t \vec{u}.
\label{eq:explicit-euler}
\end{equation}
Thus, the discrete norm after one time step is
\begin{equation}
\begin{aligned}
  {\vec{u}^+}^T \mat{M} \vec{u}^+
  &= \left( \vec{u} + \Delta t \, \partial_t \vec{u} \right)^T
     \mat{M}
     \left( \vec{u} + \Delta t \, \partial_t \vec{u} \right)
\\&= \vec{u}^T \mat{M} \vec{u}
     + 2 \Delta t \, \vec{u}^T \mat{M} \partial_t \vec{u}
     + \Delta t^2 \, \partial_t \vec{u}^T \mat{M} \partial_t \vec{u}.
\end{aligned}
\end{equation}
The computations leading to Lemma \ref{lem:CPR-linearly-stable} result in
an estimate of the second term $\vec{u}^T \mat{M} \partial_t \vec{u} \leq 0$.
Since $\mat{M}$ is positive definite, the third term is non-negative
(for $\Delta t > 0$). Thus, time discretisation introduces a growth of the
discrete norm not considered in the previous calculations, possibly leading
to stability issues.
In order to get a decay of the discrete norm, the difference
\begin{equation}
  {\vec{u}^+}^T \mat{M} \vec{u}^+ - \vec{u}^T \mat{M} \vec{u}
  = 2 \Delta t \left( \vec{u} + \frac{\Delta t}{2} \partial_t \vec{u} \right)^T
    \mat{M} \partial_t \vec{u}
\end{equation}
must be estimated. Inserting the SBP CPR semidiscretisation for the linear
advection equation with constant velocity yields
\begin{equation}
\begin{aligned}
  &\left( \vec{u} + \frac{\Delta t}{2} \partial_t \vec{u} \right)^T
    \mat{M} \partial_t \vec{u}
\\&= \left( \vec{u} - \frac{\Delta t}{2} \mat{D} \vec{u}
          - \frac{\Delta t}{2} \mat{C} (\vec{f}^{num} - \mat{R} \vec{u}) \right)^T
      \mat{M}
      \left( -\mat{D} \vec{u} - \mat{C} (\vec{f}^{num} - \mat{R} \vec{u})\right)
\\&= -\vec{u}^T \mat{M} \mat{D} \vec{u}
     - \vec{u}^T \mat{M} \mat{C} (\vec{f}^{num} - \mat{R} \vec{u})
\\&\quad+ \frac{\Delta t}{2} \left[
       \vec{u}^T \mat{D}^T \mat{M} \mat{D} \vec{u}
       + (\vec{f}^{num} - \mat{R} \vec{u})^T \mat{C}^T \mat{M} \mat{C} 
         (\vec{f}^{num} - \mat{R} \vec{u})
       \right] 
\\&\quad \left. + 2 \vec{u}^T \mat{D}^T \mat{M} \mat{C} (\vec{f}^{num} - \mat{R} \vec{u})
     \right].
\end{aligned}
\end{equation}
Inserting $\mat{C} = \mat{M}^{-1} \mat{R}^T \mat{B}$, the right hand side becomes
\begin{equation}
\begin{aligned}
  &-\vec{u}^T \mat{M} \mat{D} \vec{u}
     - \vec{u}^T \mat{R}^T \mat{B} (\vec{f}^{num} - \mat{R} \vec{u})
\\&+ \frac{\Delta t}{2} \left[
       \vec{u}^T \mat{D}^T \mat{M} \mat{D} \vec{u}
       + (\vec{f}^{num} - \mat{R} \vec{u})^T \mat{B}^T \mat{R} \mat{M}^{-1}
         \mat{R}^T \mat{B} (\vec{f}^{num} - \mat{R} \vec{u}) \right.
\\& \left. + 2 \vec{u}^T \mat{D}^T \mat{R}^T \mat{B} (\vec{f}^{num} - \mat{R} \vec{u})
     \right].
\end{aligned}
\end{equation}
Using the SBP property
$\mat{M} \mat{D} + \mat{D}^T \mat{M} = \mat{R}^T \mat{B} \mat{R}$ results in
\begin{equation}
\begin{aligned}
  &\vec{u}^T \mat{D}^T \mat{M} \vec{u}
   - \vec{u}^T \mat{R}^T \mat{B} \mat{R} \vec{u}
   - \vec{u}^T \mat{R}^T \mat{B} (\vec{f}^{num} - \mat{R} \vec{u})
\\&+ \frac{\Delta t}{2} \left[
       \vec{u}^T \mat{D}^T \mat{M} \mat{D} \vec{u}
       + (\vec{f}^{num} - \mat{R} \vec{u})^T \mat{B}^T \mat{R} \mat{M}^{-1}
         \mat{R}^T \mat{B} (\vec{f}^{num} - \mat{R} \vec{u}) \right.
\\&+ \left. + 2 \vec{u}^T \mat{D}^T \mat{R}^T \mat{B} (\vec{f}^{num} - \mat{R} \vec{u})
     \right].
\end{aligned}
\end{equation}
Thus, adding these expressions, the left hand side $LHS$ can be expressed as
\begin{equation}
\begin{aligned}
  &2 \, LHS
  =  \vec{u}^T \mat{R}^T \mat{B} \mat{R} \vec{u}
      - 2 \vec{u}^T \mat{R}^T \mat{B} \vec{f}^{num}
\\&\quad+ \Delta t \left[
        \vec{u}^T \mat{D}^T \mat{M} \mat{D} \vec{u}
        + (\vec{f}^{num} - \mat{R} \vec{u})^T \mat{B}^T \mat{R} \mat{M}^{-1}
          \mat{R}^T \mat{B} (\vec{f}^{num} - \mat{R} \vec{u}) \right.
\\&\quad \left. + 2 \vec{u}^T \mat{D}^T \mat{R}^T \mat{B} (\vec{f}^{num} - \mat{R} \vec{u})
      \right].
\end{aligned}
\end{equation}
If all summands on the right hand side would involve only boundary terms,
an estimate similar to those in previous sections would be possible.
Unfortunately, the volume term $\vec{u}^T \mat{D}^T \mat{M} \mat{D} \vec{u}$
does not seem to allow such an estimate. Therefore, an estimate leading to
fully discrete stability for an explicit Euler method does not seem to be
possible in this straightforward calculation. Thus, for practical
calculations, a time discretisation with high accuracy should be chosen in
order to avoid stability issues. Note that the non-positive stability result
for the explicit Euler methods extends directly to standard SSP time
discretisations consisting of convex combinations of Euler steps.

\section{CPR methods using SBP operators: Burgers' equation }\label{Chapter_4}

Stability properties for linear and nonlinear problems can be very different.
Thus, although stability for linear advection with constant velocity can be
proven for several SBP CPR schemes (see Theorem \ref{thm:CPR-1D-linear}),
these results do not imply nonlinear stability. For simplicity, the (inviscid)
Burgers' equation
\begin{equation}
  \partial_t u + \partial_x \frac{u^2}{2} = 0
\label{eq:inviscid-burgers}
\end{equation}
in one space dimension with periodic boundary conditions and appropriate
initial condition is considered.

\subsection{Nonlinear stability}

A straightforward application of an SBP CPR method can be written as
\begin{equation}
  \partial_t \vec{u} + \frac{1}{2} \mat{D} \vec{u^2}
  + \mat{C} ( \vec{f}^{num} - \frac{1}{2} \mat{R} \vec{u^2}) = 0
\end{equation}
for the standard element. Estimating the discrete norm similar to the previous
sections results in
\begin{equation}
  \vec{u}^T \mat{M} \partial_t \vec{u}
  =
    - \frac{1}{2} \vec{u}^T \mat{M} \mat{D} \vec{u^2}
    -  \vec{u}^T \mat{M} \mat{C} ( \vec{f}^{num} - \frac{1}{2} \mat{R} \vec{u^2}).
\end{equation}
Applying the SBP property and $\mat{M} \mat{C} = \mat{R}^T \mat{B}$ yields
\begin{equation}
  \frac{1}{2} \od{}{t} \norm{u}_M^2
  = 
    \frac{1}{2} \vec{u}^T \mat{D}^T \mat{M} \vec{u^2}
    - \frac{1}{2} \vec{u}^T \mat{R}^T \mat{B} \mat{R} \vec{u^2}
    - \vec{u}^T \mat{R}^T \mat{B} ( \vec{f}^{num} - \frac{1}{2} \mat{R} \vec{u^2}).
\end{equation}
Unfortunately, the nonlinear flux does not allow a cancellation of boundary
terms as in the linear case. A possibility to overcome this problem in the
setting of DG spectral element methods was investigated by Gassner \cite{gassner2013skew}.
There, he uses Lobatto-Legendre interpolation polynomials as nodal basis
and a skew-symmetric form of the conservation law
\begin{equation}
  \partial_t u + \alpha \partial_x \frac{u^2}{2} + (1-\alpha) u \partial_x u = 0,
  \quad 0 \leq \alpha \leq 1.
\label{eq:inviscid-burgers-skew-symmetric}
\end{equation}
Thus, the divergence of the flux is written as a convex combination of the two
terms $\partial_x \frac{u^2}{2}$ and $u \partial_x u$ which are exactly equal
if the product rule of differentiation is valid for $\partial_x$. The discrete
derivative operator $\mat{D}$ does not fulfil this product rule and therefore,
the split operator form \eqref{eq:inviscid-burgers-skew-symmetric} can be
regarded as the standard conservative form \eqref{eq:inviscid-burgers}
with an additional correction term
\begin{equation}
  \partial_t u + \partial_x \frac{u^2}{2}
  + (1 - \alpha) \left( u \partial_x u - \partial_x \frac{u^2}{2} \right)
  = 0.
\label{eq:inviscid-burgers-corrected-divergence}
\end{equation}
Using the SBP CPR semidiscretisation for this equation results is
\begin{equation}
  \partial_t \vec{u}
  = - \frac{1}{2} \mat{D} \vec{u^2}
    - (1 - \alpha) (\mat{u} \mat{D} \vec{u} - \frac{1}{2} \mat{D} \vec{u^2})
    - \mat{C} ( \vec{f}^{num} - \frac{1}{2} \mat{R} \vec{u^2}) = 0.
\end{equation}
Here, $\mat{u} = \diag{\vec{u}}$ denotes the matrix representing multiplication
with $u$. Now, multiplication with $\vec{u}^T \mat{M}$ results in
\begin{equation}
  \frac{1}{2} \od{}{t} \norm{u}_M^2
  = 
    - \frac{\alpha}{2} \vec{u}^T \mat{M} \mat{D} \vec{u^2}
    - (1 - \alpha) \vec{u}^T \mat{M} \mat{u} \mat{D} \vec{u}
    -  \vec{u}^T \mat{M} \mat{C} ( \vec{f}^{num} - \frac{1}{2} \mat{R} \vec{u^2}).
\end{equation}
Using $\mat{C} = \mat{M}^{-1} \mat{R}^T \mat{B}$ and
$\vec{u^2} = \mat{u} \vec{u}$, this can be written as
\begin{equation}
\begin{aligned}
  \frac{1}{2} \od{}{t} \norm{u}_M^2
  = 
  & - \frac{\alpha}{2} \vec{u}^T \mat{M} \mat{D} \mat{u} \vec{u}
    - (1 - \alpha) \vec{u}^T \mat{M} \mat{u} \mat{D} \vec{u}
\\& -  \vec{u}^T \mat{R}^T \mat{B} \vec{f}^{num}
    + \frac{1}{2} \vec{u}^T \mat{R}^T \mat{B} \mat{R} \mat{u} \vec{u}.
\end{aligned}
\end{equation}
Since a nodal Lobatto-Legendre basis with lumped mass matrix is chosen,
both $\mat{u}$ and $\mat{M}$ are diagonal and therefore commute
\begin{equation}
\begin{aligned}
  \frac{1}{2} \od{}{t} \norm{u}_M^2
  = 
  & - \frac{\alpha}{2} \vec{u}^T \mat{M} \mat{D} \mat{u} \vec{u}
    - (1 - \alpha) \vec{u}^T \mat{u} \mat{M} \mat{D} \vec{u}
\\& -  \vec{u}^T \mat{R}^T \mat{B} \vec{f}^{num}
    + \frac{1}{2} \vec{u}^T \mat{R}^T \mat{B} \mat{R} \mat{u} \vec{u}.
\end{aligned}
\end{equation}
Application of the SBP property results in
\begin{equation}
\begin{aligned}
  \frac{1}{2} \od{}{t} \norm{u}_M^2
  = 
  &   \frac{\alpha}{2} \vec{u}^T \mat{D}^T \mat{M} \mat{u} \vec{u}
    - \frac{\alpha}{2} \vec{u}^T \mat{R}^T \mat{B} \mat{R} \mat{u} \vec{u}
    - (1 - \alpha) \vec{u}^T \mat{u} \mat{M} \mat{D} \vec{u}
\\& -  \vec{u}^T \mat{R}^T \mat{B} \vec{f}^{num}
    + \frac{1}{2} \vec{u}^T \mat{R}^T \mat{B} \mat{R} \mat{u} \vec{u}.
\end{aligned}
\end{equation}
Choosing $\alpha = \frac{2}{3}$, $\frac{\alpha}{2} = 1 - \alpha$ and the 
volume terms cancel out
\begin{equation}
  \frac{1}{2} \od{}{t} \norm{u}_M^2
  = 
      \frac{1}{6} \vec{u}^T \mat{R}^T \mat{B} \mat{R} \mat{u} \vec{u}
    -  \vec{u}^T \mat{R}^T \mat{B} \vec{f}^{num}.
\label{eq:inviscid-burgers-corrected-divergence-result-boundary}
\end{equation}
Thus, Gassner \cite{gassner2013skew} is able to estimate the rate of change in one
element as
\begin{equation}
  \frac{1}{2} \od{}{t} \norm{u}_M^2
  = \frac{1}{6} ( u_R^3 - u_L^3 ) - (u_R f^{num}_R - u_L f^{num}_L),
\label{eq:inviscid-burgers-corrected-divergence-result-boundary-lobatto}
\end{equation}
where the indices $L$ and $R$ denote the values at the left and right boundary
points, respectively. Lobatto-Legendre quadrature is necessary for this
calculation, because the boundary points are nodes for the basis and therefore
the restriction of $u^2$ to the boundary is the square of the restriction of
$u$ to the boundary, i.e. $\mat{R} \vec{u^2} = (\mat{R} \vec{u})^2$. In general,
this is false for other nodes, as for example Gauß-Legendre quadrature.

Continuing the investigation, using periodic boundary conditions and summing
over all elements, the contribution of one boundary can be expressed as
\begin{equation}
  \frac{1}{6} ( u_-^3 - u_+^3 ) - (u_- - u_+) f^{num},
\label{eq:inviscid-burgers-corrected-divergence-result-boundary-terms}
\end{equation}
where the indices $-$ and $+$ indicate the values from the left and right
element, respectively. With the choice of 
\begin{equation}
  f^{num} = \frac{1}{2} \left( \frac{u_+^2}{2} + \frac{u_-^2}{2} \right)
            - \lambda (u_+ - u_-)
\label{eq:burgers-gassner-flux}
\end{equation}
as numerical flux, one can estimate this contribution like \cite{gassner2013skew}
\begin{equation}
\begin{aligned}
  & \frac{1}{6} ( u_-^3 - u_+^3 )
    + \frac{1}{4} (u_+ - u_-) (u_+^2 + u_-^2)
    - \lambda (u_+ - u_-)^2
\\&=\frac{1}{6} ( u_-^3 - u_+^3 )
    + \frac{1}{4} ( u_+^3 - u_+^2 u_- + u_+ u_-^2 - u_-^3 )
    - \lambda (u_+ - u_-)^2
\\&=\frac{1}{12} (u_+^3 - 3 u_+^2 u_- + 3 u_+ u_-^2 - u_-^3)
    - \lambda (u_+ - u_-)^2
\\&=(u_+ - u_-)^2 \left(\frac{u_+ - u_-}{12} - \lambda \right).
\end{aligned}
\label{eq:burgers-gassner-flux-contribution}
\end{equation}
Thus, $\lambda \geq \frac{u_+ - u_-}{12}$ ensures $\od{}{t} \norm{u}_M^2 \leq 0$,
and therefore stability. This proves the following
\begin{lem}[see also \cite{gassner2013skew}]
\label{lem:lobatto-CPR-burgers}
  If the numerical flux $f^{num}$ satisfies
  \begin{equation}
    \frac{1}{6} ( u_-^3 - u_+^3 ) - (u_- - u_+) f^{num}(u_-, u_+) \leq 0,
  \end{equation}
  the CPR method with nodal Lobatto-Legendre basis and
  $\mat{C} = \mat{M}^{-1} \mat{R}^T \mat{B}$ for the skew-symmetric
  inviscid Burgers' equation \eqref{eq:inviscid-burgers-skew-symmetric}
  with correction parameter $\alpha = \frac{2}{3}$, written as
  \begin{equation}
    \partial_t \vec{u}
    + \mat{D} \frac{1}{2} \vec{u^2}
    + \frac{1}{3} \left( \mat{u} \mat{D} \vec{u} - \mat{D} \frac{1}{2} \vec{u^2}\right)
    + \mat{C} \left( \vec{f}^{num} - \mat{R} \frac{1}{2} \vec{u^2} \right)
    = 0,
  \end{equation}
  is stable in the discrete norm $\norm{\cdot}_{M}$ induced by $\mat{M}$.
\end{lem}

As remarked above, this stability result is based on Lobatto-Legendre nodes
including the boundaries. To get stability for a general SBP basis, further
corrections are necessary. To the authors' knowledge, this is a new idea and
not published anywhere else. Recalling equation
\eqref{eq:inviscid-burgers-corrected-divergence-result-boundary}, the contribution
of one boundary is
\begin{equation}
  \frac{1}{2} \od{}{t} \norm{u}_M^2
  = 
    \frac{1}{6} ( u_- (u^2)_- - u_+ (u^2)_+ )
    + (u_+ - u_-) f^{num}.
\end{equation}
In general, multiplication and restriction to the boundary do not commute,
i.e. $(u_R)^2 \neq (u^2)_R$, as mentioned above. In comparison with the
estimate \eqref{eq:inviscid-burgers-corrected-divergence-result-boundary-lobatto}
of \cite{gassner2013skew},
\begin{equation}
  \frac{1}{6} ( u_- (u^2)_- - u_+ (u^2)_+ )
  - \frac{1}{6} ( u_-^3 - u_+^3 )
\end{equation}
appears as additional term on the right hand side, possibly leading to
instability. Therefore, an SBP CPR method with corrected divergence
(skew-symmetric form) and corrected boundary terms is proposed
\begin{equation}
  \partial_t \vec{u}
  + \frac{\alpha}{2} \mat{D} \vec{u^2}
  + (1 - \alpha) \mat{u} \mat{D} \vec{u}
  + \mat{C} \left( \vec{f}^{num}
                  - \frac{\beta}{2} \mat{R} \vec{u^2}
                  - \frac{1 - \beta}{2} (\mat{R} \vec{u})^2 
            \right)
  = 0,
\label{eq:inviscid-burgers-CPR-corrected-divergence-restriction}
\end{equation}
$0 \leq \alpha, \beta \leq 1$. Setting $\alpha = \frac{2}{3}$ and repeating the
calculations as above results in
\begin{equation}
  \frac{1}{2} \od{}{t} \norm{u}_M^2
  = - \vec{u}^T \mat{R} \mat{B} \vec{f}^{num}
    - \left(\frac{1}{3} - \frac{\beta}{2}\right) \vec{u}^T \mat{R}^T \mat{B} \mat{R} \vec{u^2}
    + \frac{1 - \beta}{2} \vec{u}^T \mat{R}^T \mat{B} (\mat{R} \vec{u})^2.
\end{equation}
Therefore, setting $\beta = \frac{2}{3}$ results in 
\begin{equation}
  \frac{1}{2} \od{}{t} \norm{u}_M^2
  = - \vec{u}^T \mat{R} \mat{B} \vec{f}^{num}
    + \frac{1}{6} \vec{u}^T \mat{R}^T \mat{B} (\mat{R} \vec{u})^2,
\end{equation}
and the contribution of one boundary is the same as for Lobatto-Legendre nodes in
\eqref{eq:inviscid-burgers-corrected-divergence-result-boundary-terms}.
This proves the following
\begin{lem}
\label{lem:CPR-burgers-stable}
  If the numerical flux $f^{num}$ satisfies
  \begin{equation}
    \frac{1}{6} ( u_-^3 - u_+^3 ) - (u_- - u_+) f^{num}(u_-, u_+) \leq 0,
  \end{equation}
  the SBP CPR method with $\mat{C} = \mat{M}^{-1} \mat{R}^T \mat{B}$
  for the inviscid Burgers' equation \eqref{eq:inviscid-burgers}
  with correction terms for the divergence and restriction to the boundary,
  written as
  \begin{equation}
  \begin{aligned}
    &\partial_t \vec{u}
    + \mat{D} \frac{1}{2} \vec{u^2}
    + \frac{1}{3} \left( \mat{u} \mat{D} \vec{u} - \mat{D} \frac{1}{2} \vec{u^2}\right)
  \\&+ \mat{C} \left( \vec{f}^{num}
                    - \mat{R} \frac{1}{2} \vec{u^2}
                    - \frac{1}{3} \left( \frac{1}{2} (\mat{R} \vec{u})^2
                                        - \frac{1}{2} \mat{R} \vec{u^2}
                                  \right)
              \right)
    = 0,
  \end{aligned}
  \label{eq:inviscid-burgers-CPR-full-correction}
  \end{equation}
  is stable in the discrete norm $\norm{\cdot}_{M}$ induced by $\mat{M}$.
\end{lem}

The motivation to introduce the skew-symmetric form (the divergence correction)
as described in \cite{gassner2013skew} (see also inter alia
\cite{fisher2013discretely, svard2014review, fernandez2014review})
was the invalid product rule for the discrete derivative operator $\mat{D}$.
In view of the previous lemma, the inexactness of \emph{discrete multiplication}
is stressed, resulting in both an invalid product rule for polynomial bases
and incorrect restriction to the boundary for nodal bases not including
boundary points.

\subsection{Conservation}

In order to be useful, the semidiscretisation
\eqref{eq:inviscid-burgers-CPR-full-correction} also has to be conservative.
As in Lemma \ref{lem:CPR-conservation}, multiplication with the constant function,
represented as $\vec{1}$, yields
\begin{equation}
  \vec{1}^T \mat{M} \partial_t \vec{u}
  =
  - \vec{1}^T \mat{M} \mat{D} \vec{f}
  - \vec{1}^T \mat{M} \vec{c}_{div}
  - \vec{1}^T \mat{M} \mat{C} ( \vec{f}^{num} - \mat{R} \vec{f} - \vec{c}_{res}).
\end{equation}
Here and in the following, $\vec{c}_{div}$ and $\vec{c}_{res}$ denote correction
terms for the divergence and restriction, respectively.
Using $\vec{1}^T \mat{M} \mat{C} = \vec{1}^T \mat{R}^T \mat{B}$ and the
SBP property results in
\begin{equation}
  \od{}{t} \vec{1}^T \mat{M} \vec{u}
  =
    \vec{1}^T \mat{D}^T \mat{M} \vec{f}
  - \vec{1}^T \mat{R}^T \mat{B} \mat{R} \vec{f}
  - \vec{1}^T \mat{M} \vec{c}_{div}
  - \vec{1}^T \mat{R}^T \mat{B} ( \vec{f}^{num} - \mat{R} \vec{f} - \vec{c}_{res}).
\end{equation}
Since the discrete derivative is exact for constant functions,
$\mat{D} \vec{1} = 0$, and the rate of change can be expressed as
\begin{equation}
  \od{}{t} \vec{1}^T \mat{M} \vec{u}
  =
  - \vec{1}^T \mat{M} \vec{c}_{div}
  - \vec{1}^T \mat{R}^T \mat{B} \vec{f}^{num} 
  + \vec{1}^T \mat{R}^T \mat{B} \vec{c}_{res}.
\end{equation}
Inserting the correction terms
\begin{equation}
  \vec{c}_{div}
  = \frac{1}{3} \left( \mat{u} \mat{D} \vec{u}
                      - \frac{1}{2} \mat{D} \mat{u} \vec{u}
                \right)
  ,\qquad
  \vec{c}_{res}
  = \frac{1}{6} \left( (\mat{R} \vec{u})^2 - \mat{R} \mat{u} \vec{u} \right),
\end{equation}
$\od{}{t} \vec{1}^T \mat{M} \vec{u}$ is rewritten as
\begin{equation}
\begin{aligned}
  \od{}{t} \vec{1}^T \mat{M} \vec{u} 
  =
  & - \vec{1}^T \mat{R}^T \mat{B} \vec{f}^{num}
    - \frac{1}{3} \, \vec{1}^T \mat{M} \mat{u} \mat{D} \vec{u}
    + \frac{1}{6} \, \vec{1}^T \mat{M} \mat{D} \mat{u} \vec{u}
\\& + \frac{1}{6} \, \vec{1}^T \mat{R}^T \mat{B} (\mat{R} \vec{u})^2
    - \frac{1}{6} \, \vec{1}^T \mat{R}^T \mat{B} \mat{R} \mat{u} \vec{u}.
\end{aligned}
\end{equation}
For diagonal-norm SBP operators, both $\mat{M}$ and $\mat{u}$ are diagonal
and therefore commute. Using $\mat{u} \vec{1} = \vec{u}$ results in
\begin{equation}
\begin{aligned}
  \od{}{t} \vec{1}^T \mat{M} \vec{u} 
  = 
  & - \vec{1}^T \mat{R}^T \mat{B} \vec{f}^{num}
    - \frac{1}{3} \, \vec{u}^T \mat{M} \mat{D} \vec{u}
    + \frac{1}{6} \, \vec{1}^T \mat{M} \mat{D} \mat{u} \vec{u}
\\& + \frac{1}{6} \, \vec{1}^T \mat{R}^T \mat{B} (\mat{R} \vec{u})^2
    - \frac{1}{6} \, \vec{1}^T \mat{R}^T \mat{B} \mat{R} \mat{u} \vec{u}.
\end{aligned}
\end{equation}
The SBP property yields
\begin{equation}
\begin{aligned}
  \od{}{t} \vec{1}^T \mat{M} \vec{u} = 
  & - \vec{1}^T \mat{R}^T \mat{B} \vec{f}^{num}
    + \frac{1}{3} \, \vec{u}^T \mat{D}^T \mat{M} \vec{u}
    - \frac{1}{3} \, \vec{u}^T \mat{R}^T \mat{B} \mat{R} \vec{u}
\\& + \frac{1}{6} \, \vec{1}^T \mat{M} \mat{D} \mat{u} \vec{u}
    + \frac{1}{6} \, \vec{1}^T \mat{R}^T \mat{B} (\mat{R} \vec{u})^2
    - \frac{1}{6} \, \vec{1}^T \mat{R}^T \mat{B} \mat{R} \mat{u} \vec{u}.
\end{aligned}
\end{equation}
Adding the last two equations and multiplying by one half results in
\begin{equation}
\begin{aligned}
  \od{}{t} \vec{1}^T \mat{M} \vec{u} = 
  & - \vec{1}^T \mat{R}^T \mat{B} \vec{f}^{num}
    - \frac{1}{6} \, \vec{u}^T \mat{R}^T \mat{B} \mat{R} \vec{u}
    + \frac{1}{6} \, \vec{1}^T \mat{M} \mat{D} \mat{u} \vec{u}
\\& + \frac{1}{6} \, \vec{1}^T \mat{R}^T \mat{B} (\mat{R} \vec{u})^2
    - \frac{1}{6} \, \vec{1}^T \mat{R}^T \mat{B} \mat{R} \mat{u} \vec{u}.
\end{aligned}
\end{equation}
Again, by using the SBP property
\begin{equation}
\begin{aligned}
  \od{}{t} \vec{1}^T \mat{M} \vec{u} = 
  & - \vec{1}^T \mat{R}^T \mat{B} \vec{f}^{num}
    - \frac{1}{6} \, \vec{u}^T \mat{R}^T \mat{B} \mat{R} \vec{u}
    - \frac{1}{6} \, \vec{1}^T \mat{D}^T \mat{M} \mat{u} \vec{u}
\\& + \frac{1}{6} \, \vec{1}^T \mat{R}^T \mat{B} \mat{R} \mat{u} \vec{u}
    + \frac{1}{6} \, \vec{1}^T \mat{R}^T \mat{B} (\mat{R} \vec{u})^2
    - \frac{1}{6} \, \vec{1}^T \mat{R}^T \mat{B} \mat{R} \mat{u} \vec{u}.
\end{aligned}
\end{equation}
Gathering terms and using $\mat{D} \vec{1} = 0$, this can be rewritten as
\begin{equation}
\od{}{t} \vec{1}^T \mat{M} \vec{u} = 
  - \vec{1}^T \mat{R}^T \mat{B} \vec{f}^{num}
  - \frac{1}{6} \, \vec{u}^T \mat{R}^T \mat{B} \mat{R} \vec{u}
  + \frac{1}{6} \, \vec{1}^T \mat{R}^T \mat{B} (\mat{R} \vec{u})^2.
\end{equation}
Finally, since
\begin{equation}
  \vec{u}^T \mat{R}^T \mat{B} \mat{R} \vec{u}
  =
  u_R \cdot u_R - u_L \cdot u_L
  = 
  1 \cdot u_R^2 - 1 \cdot u_L^2
  =
  \vec{1}^T \mat{R}^T \mat{B} (\mat{R} \vec{u})^2,
\end{equation}
this reduces to the same equation as in the proof of Lemma
\ref{lem:CPR-conservation}. Thus, the following Lemma is proved
\begin{lem}
\label{lem:CPR-burgers-conservative}
  If $\vec{1}^T \mat{M} \mat{C} = \vec{1}^T \mat{R}^T \mat{B}$,
  the SBP CPR method \eqref{eq:inviscid-burgers-CPR-full-correction}
  for the inviscid Burgers' equation \eqref{eq:inviscid-burgers} is conservative.
\end{lem}

As Lemma \ref{lem:CPR-conservation}, Lemma
\ref{lem:CPR-burgers-conservative} proofs conservation across elements.
On a sub-element level, conservation for diagonal-norm SBP operators
(including boundary nodes) and conservation laws in split form has been proven
in \cite{fisher2013discretely} in the context of the Lax-Wendroff theorem.

\subsection{Numerical fluxes}

In the following, some numerical fluxes for Burgers' equation are investigated.
Gassner \cite{gassner2013skew} considered fluxes of the form
\eqref{eq:burgers-gassner-flux}.
Choosing $\lambda = (u_+ - u_-) / 12$ leads to the \emph{energy conservative}
(ECON) flux of \cite{gassner2013skew}
\begin{equation}
  f^{num}(u_-, u_+)
  =
  \frac{1}{4} ( u_+^2 + u_-^2 )
  - \frac{(u_+ - u_-)^2}{12}
\label{eq:burgers-econ-flux}
\end{equation}
With this choice, the contribution of
the boundary terms vanishes and therefore the energy $\norm{u}^2$ remains constant.
Since the energy is also an entropy, this will result in unphysical solutions
after the formation of shocks.

The choice $\lambda = |u_+ + u_-| / 2$ results in \emph{Roe's} flux
\begin{equation}
  f^{num}(u_-, u_+)
  =
  \frac{1}{4} ( u_+^2 + u_-^2 )
  - \frac{|u_+ - u_-|}{2} (u_+ - u_-)
\label{eq:burgers-roe-flux}
\end{equation}
Unfortunately, the contribution \eqref{eq:burgers-gassner-flux-contribution}
is not guaranteed to be non-negative, since $|u_+ + u_-| \geq (u_+ - u_-) / 6$
is possible, e.g. for $u_+ = -u_- > 0$. Therefore, this choice does not imply
stability.

Finally, Gassner \cite{gassner2013skew} considered the \emph{local Lax-Friedrichs}
(LLF) flux with parameter $\lambda = \frac{\max( |u_+|, |u_-| )}{2}$
\begin{equation}
  f^{num}(u_-, u_+)
  =
  \frac{1}{4} ( u_+^2 + u_-^2 )
  - \frac{\max( |u_+|, |u_-| )}{2} (u_+ - u_-)
\label{eq:burgers-llf-flux}
\end{equation}
leading to entropy stability, since
$\max( |u_+|, |u_-| ) \geq |u_+| + |u_-| \geq (u_+ - u_-) / 6$.

Another possible numerical flux is \emph{Osher's} flux (see
\cite[section 12.1.4]{toro2009riemann})
\begin{equation}
  f^{num}(u_-, u_+)
  =
  \begin{cases}
    \frac{u_-^2}{2}, & u_+, u_- > 0, \\
    \frac{u_+^2}{2}, & u_+, u_- < 0, \\
    \frac{u_+^2}{2} + \frac{u_-^2}{2}, & u_- \geq 0 \geq u_+, \\
    0, & u_- \leq 0 \leq u_+.
  \end{cases}
\label{eq:burgers-osher-flux}
\end{equation}
Inserting this flux in the condition of Lemma \ref{lem:CPR-burgers-stable} for
the case $u_+, u_- > 0$ leads to
\begin{equation}
\begin{aligned}
  &\frac{1}{6} ( u_-^3 - u_+^3 ) - (u_- - u_+) \frac{u_-^2}{2}
  =
    - \frac{1}{3} u_-^3-\frac{1}{6} u_+^3  + \frac{1}{2} u_+ u_-^2
\\&\leq
    - \frac{1}{3} u_-^3-\frac{1}{6} u_+^3 
    + \frac{1}{2} \left( \frac{1}{3} u_+^3 + \frac{2}{3} u_-^3 \right)
  =
  0,
\end{aligned}
\end{equation}
where Young's inequality 
\begin{equation}
  a b \leq \frac{a^p}{p} + \frac{b^q}{q},
  \quad a, b > 0,\quad \frac{1}{p} + \frac{1}{q} = 1
\label{eq:young-inequality}
\end{equation}
was used. The case $u_+, u_- < 0$ is similar. If $u_- \geq 0 \geq u_+$,
the condition of Lemma \ref{lem:CPR-burgers-stable} reads
\begin{equation}
  \frac{1}{6} ( u_-^3 - u_+^3 ) - \frac{1}{2} (u_- - u_+) (u_+^2 + u_-^2)
  =
    - \frac{1}{3} u_-^3 + \frac{1}{3} u_+^3
    - \frac{1}{2} u_+^2 u_- + \frac{1}{2} u_+ u_-^2
  \leq 0,
\end{equation}
since each term is not positive. Finally, for $u_- \leq 0 \leq u_+$, the
contribution $( u_-^3 - u_+^3 ) / 6$ is again not positive. Thus, this
numerical flux results in a stable scheme.

\subsection{Summary and numerical results}

The results are summed up in the following
\begin{thm}
\label{thm:CPR-burgers}
  If the numerical flux $f^{num}$ satisfies
  \begin{equation}
    \frac{1}{6} ( u_-^3 - u_+^3 ) - (u_- - u_+) f^{num}(u_-, u_+) \leq 0,
  \end{equation}
  then an SBP CPR method with $\mat{C} = \mat{M}^{-1} \mat{R}^T \mat{B}$ and
  correction terms for both divergence and restriction to the boundary
  \begin{equation}
  \tag{\eqref{eq:inviscid-burgers-CPR-full-correction}}
  \begin{aligned}
    & \partial_t \vec{u}
      + \mat{D} \frac{1}{2} \vec{u^2}
      + \frac{1}{3} \left( \mat{u} \mat{D} \vec{u} - \mat{D} \frac{1}{2} \vec{u^2}\right)
  \\& + \mat{C} \left( \vec{f}^{num}
                    - \mat{R} \frac{1}{2} \vec{u^2}
                    - \frac{1}{3} \left( \frac{1}{2} (\mat{R} \vec{u})^2
                                        - \frac{1}{2} \mat{R} \vec{u^2}
                                  \right)
              \right)
    = 0,
  \end{aligned}
  \end{equation}
  for the inviscid Burgers' equation \eqref{eq:inviscid-burgers} 
  is both conservative and stable in the discrete norm $\norm{\cdot}_{M}$
  induced by $\mat{M}$.
  Numerical fluxes fulfilling this condition are inter alia
  \begin{itemize}
    \item the energy conservative (ECON) flux \eqref{eq:burgers-econ-flux},
    \item the local Lax-Friedrichs (LLF) flux \eqref{eq:burgers-llf-flux},
    \item and Osher's flux \eqref{eq:burgers-osher-flux}.
  \end{itemize}
\end{thm}
Of course, the ECON flux should not be used in situations involving
discontinuities, as shown in the following numerical examples.
The setting is the same as in the case considered in \cite{gassner2013skew},
i.e. the inviscid Burgers' equation \eqref{eq:inviscid-burgers} in the domain
$[0, 2]$ with periodic boundary conditions is solved. The initial condition is
\begin{equation}
  u(0, x) = u_0(x) = \sin(\pi x) + 0.01.
\end{equation}
Several SBP CPR methods with $N = 20$ equally spaced elements of order $p = 7$
and correction terms for the divergence (and restriction, if mentioned) are
used as semidiscretisation. The classical Runge-Kutta method of fourth order
with $10,000$ equal time steps is used to obtain the discrete solution in the
time interval $[0, 3]$.

Results for the SBP CPR method with Lobatto-Legendre basis points and associated
quadrature as discrete norm are shown in Figure \ref{fig:gassner2013skew-lobatto}.
Since the correction for restriction to the boundary is zero, only a correction
term for the divergence is used.
On the left-hand side, the solution $u(3) = u(3, \cdot)$ obtained with an
energy conservative \eqref{eq:burgers-econ-flux}, local Lax-Friedrichs
\eqref{eq:burgers-llf-flux} and Osher's \eqref{eq:burgers-osher-flux} numerical
flux is plotted. On the right-hand side, the evolution of associated discrete
momentum $\vec{1}^T \mat{M} \vec{u}$ and energy $\vec{u}^T \mat{M} \vec{u}$ in
the time interval $[0, 3]$ is visualized.

The ECON flux yields a conservation
of discrete momentum and energy relative to the initial values of order
$10^{-5}$, as expected. Using a more accurate time integrator would result in
better preservation of these values. Due to the discontinuity around $x = 1$,
the results obtained by the ECON flux are not physically relevant and highly
oscillatory.

Both the local Lax-Friedrichs and Osher's flux yield good results. After the
development of the shock before $t = 0.5$, discrete momentum and energy are
constant. Afterwards, momentum is conserved but energy is dissipated, as it is
an entropy for Burgers' equation. Around the shock, oscillations develop but
remain bounded and the total scheme is stable.

The results in Figure \ref{fig:gassner2013skew-gauss-with-correction} are
qualitatively similar to those mentioned above. There, a Gauß-Legendre basis
is used in an SBP CPR method with correction terms for both divergence and
restriction to the boundary. The plots look very similar to those of Figure
\ref{fig:gassner2013skew-lobatto}, but higher accuracy of Gauß-Legendre 
integration yields slightly less oscillatory solutions for the local
Lax-Friedrichs and Osher's flux and a smoother decay of entropy. As before,
ECON flux does not yield a physically relevant solution.

In contrast, Figure \ref{fig:gassner2013skew-gauss-without-correction} shows
results for a Gauß-Legendre basis without the correction term for restriction.
In accordance with the theoretical investigations, conservation and stability
cannot be guaranteed. A blow-up of energy for the ECON flux occurs around
$t = 0.43$. The other solutions are not physically relevant as well, since
momentum is lost. Therefore, the additional correction term is
necessary.

The results for Roe's flux are shown in Figure \ref{fig:gassner2013skew-roe}.
Stability cannot be guaranteed by using this flux and accordingly the solution
obtained by a Lobatto-Legendre basis blows up around $t = 2.5$. The computations
using Gauß-Legendre basis remain stable and energy is dissipated, but they do
not seem to be as acceptable as those obtained using Osher's or the local
Lax-Friedrichs flux. Without the additional correction term for restriction
to the boundary, momentum conservation is lost and severe oscillations occur.

Finally, results for high order methods using polynomials of degree $p = 25$
and $p = 50$ are shown in Figures \ref{fig:gassner2013skew-high-order-25}
and \ref{fig:gassner2013skew-high-order-50}, respectively. The remaining
parameters are the same as mentioned above, the only difference occurs in
the increased number ($50,000$ and $100,000$, respectively) of time steps.
Of course, very strong oscillations occur, but the method remains stable and
conservative. The plots for momentum and energy look precisely like the ones
obtained for $p = 7$ and are consequently omitted. These numerical results
confirm the proven stability and conservation results even in the case of
very high order methods and discontinuous solutions.

\begin{figure}[!hp]
  \centering
  \begin{subfigure}[b]{0.45\textwidth}
    \includegraphics[width=\textwidth]{%
      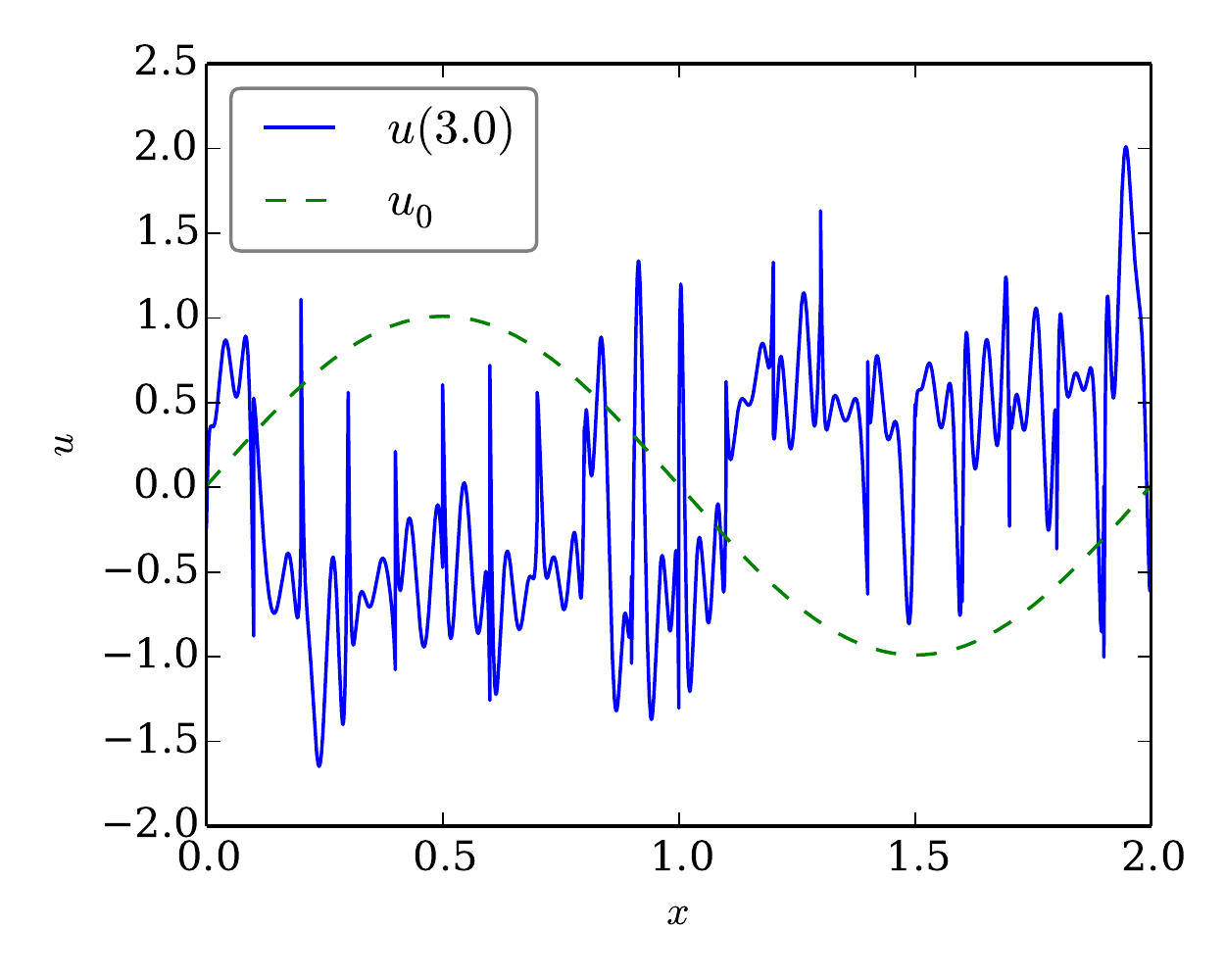}
    \caption{Energy conservative flux.}
  \end{subfigure}%
  ~
  \begin{subfigure}[b]{0.45\textwidth}
    \includegraphics[width=\textwidth]{%
      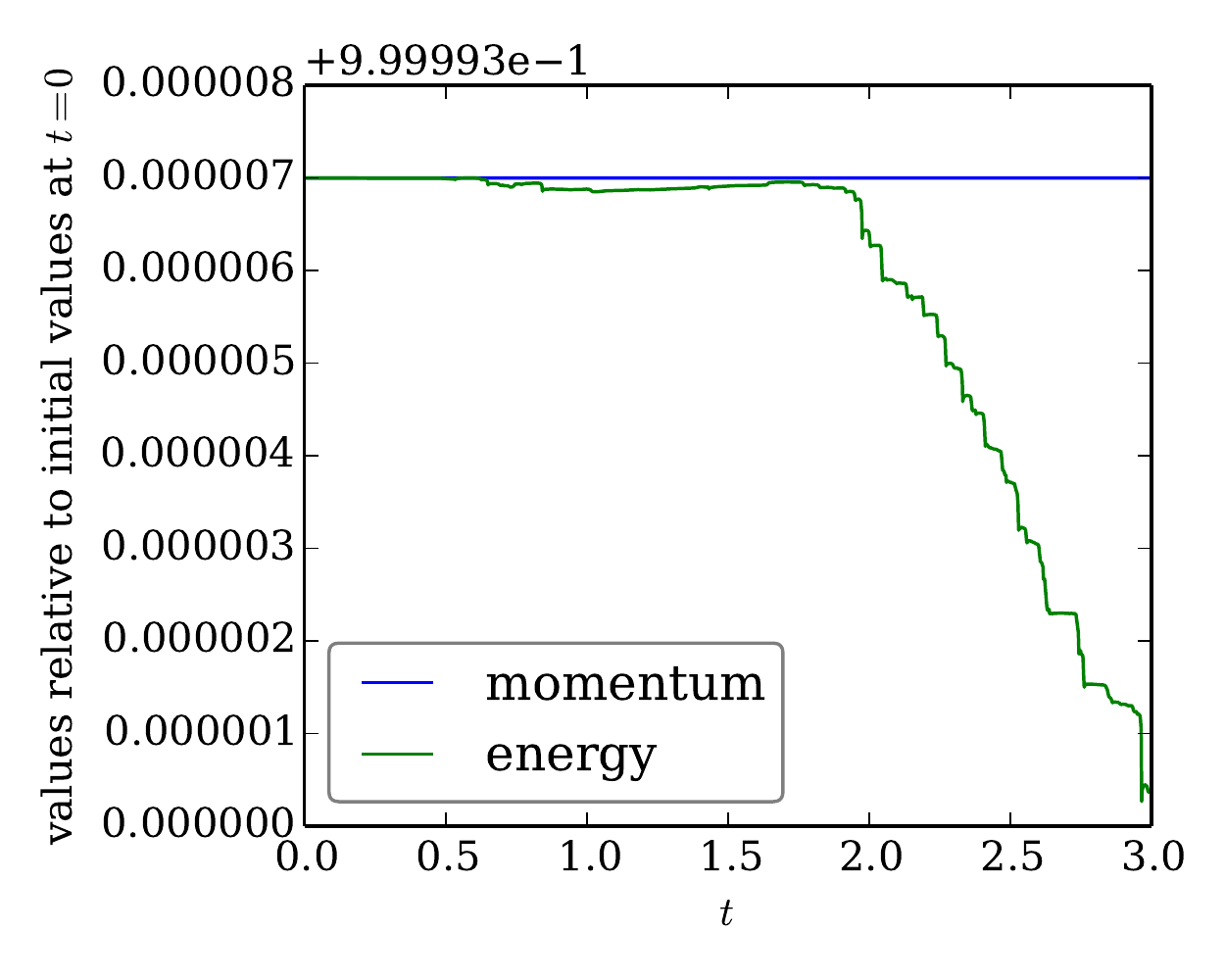}
    \caption{Energy conservative flux.}
  \end{subfigure}%
  ~\\
  \begin{subfigure}[b]{0.45\textwidth}
    \includegraphics[width=\textwidth]{%
      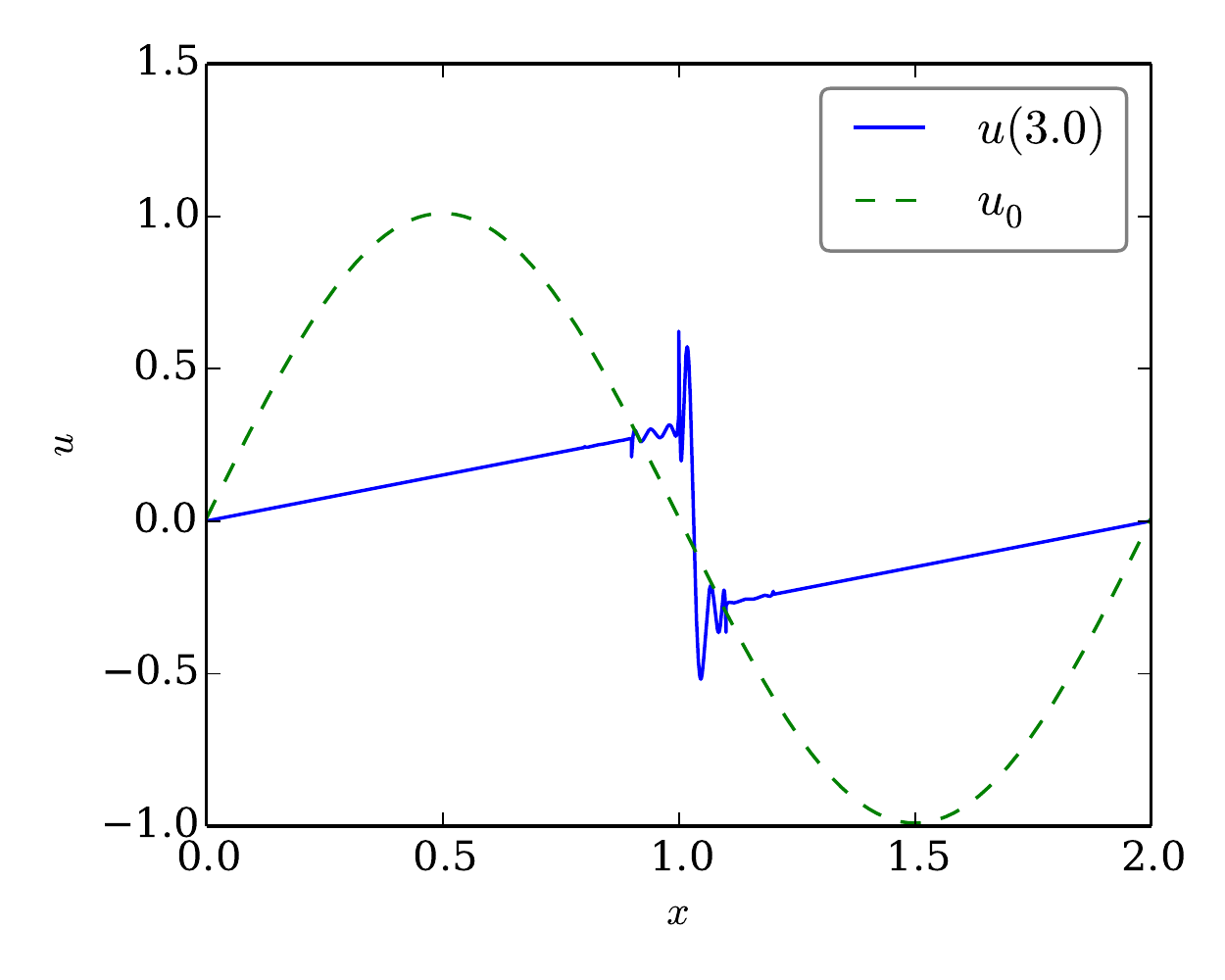}
    \caption{Local Lax-Friedrichs flux.}
  \end{subfigure}%
  ~
  \begin{subfigure}[b]{0.45\textwidth}
    \includegraphics[width=\textwidth]{%
      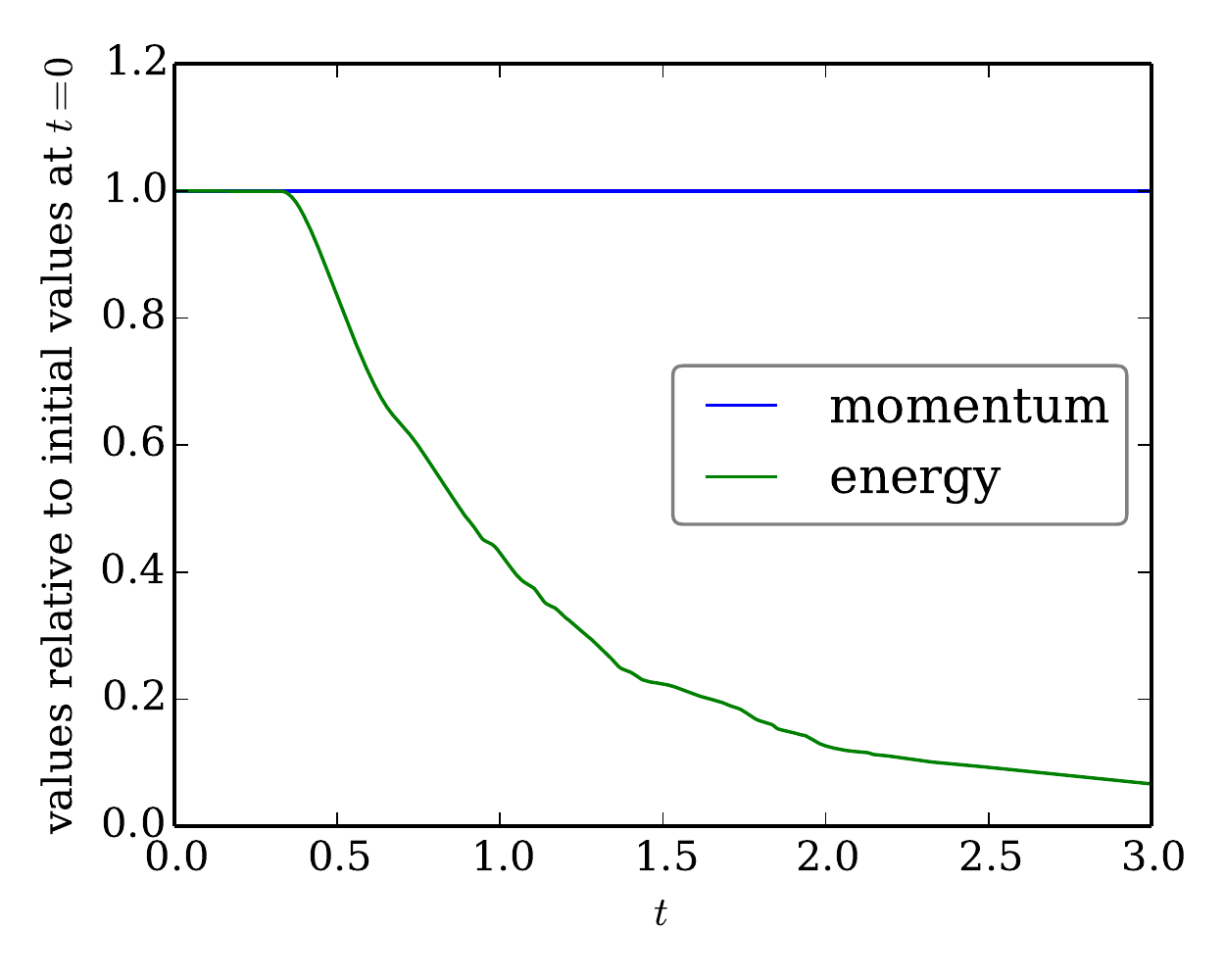}
    \caption{Local Lax-Friedrichs flux.}
  \end{subfigure}%
  ~\\
  \begin{subfigure}[b]{0.45\textwidth}
    \includegraphics[width=\textwidth]{%
      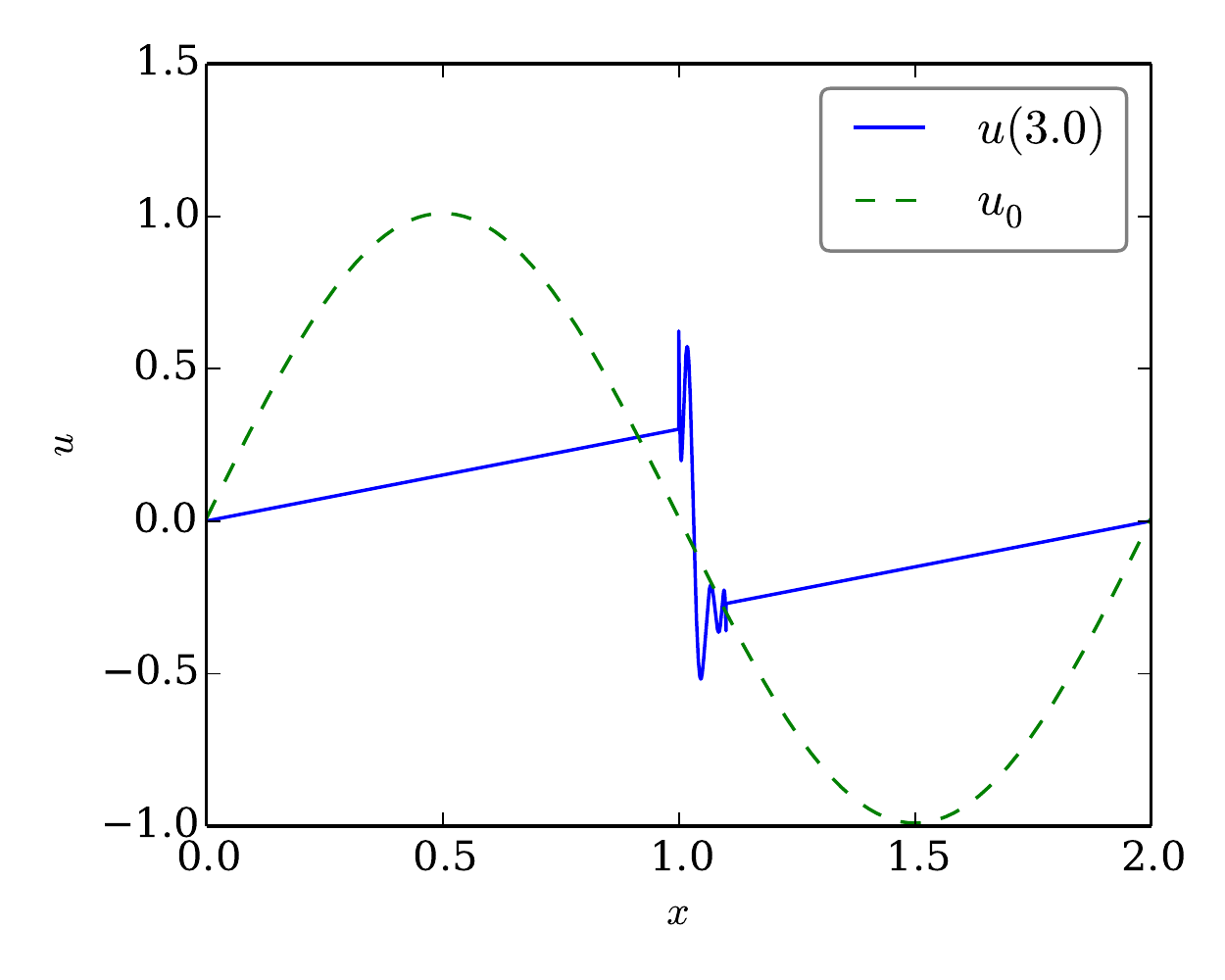}
    \caption{Osher's flux.}
  \end{subfigure}%
  ~
  \begin{subfigure}[b]{0.45\textwidth}
    \includegraphics[width=\textwidth]{%
      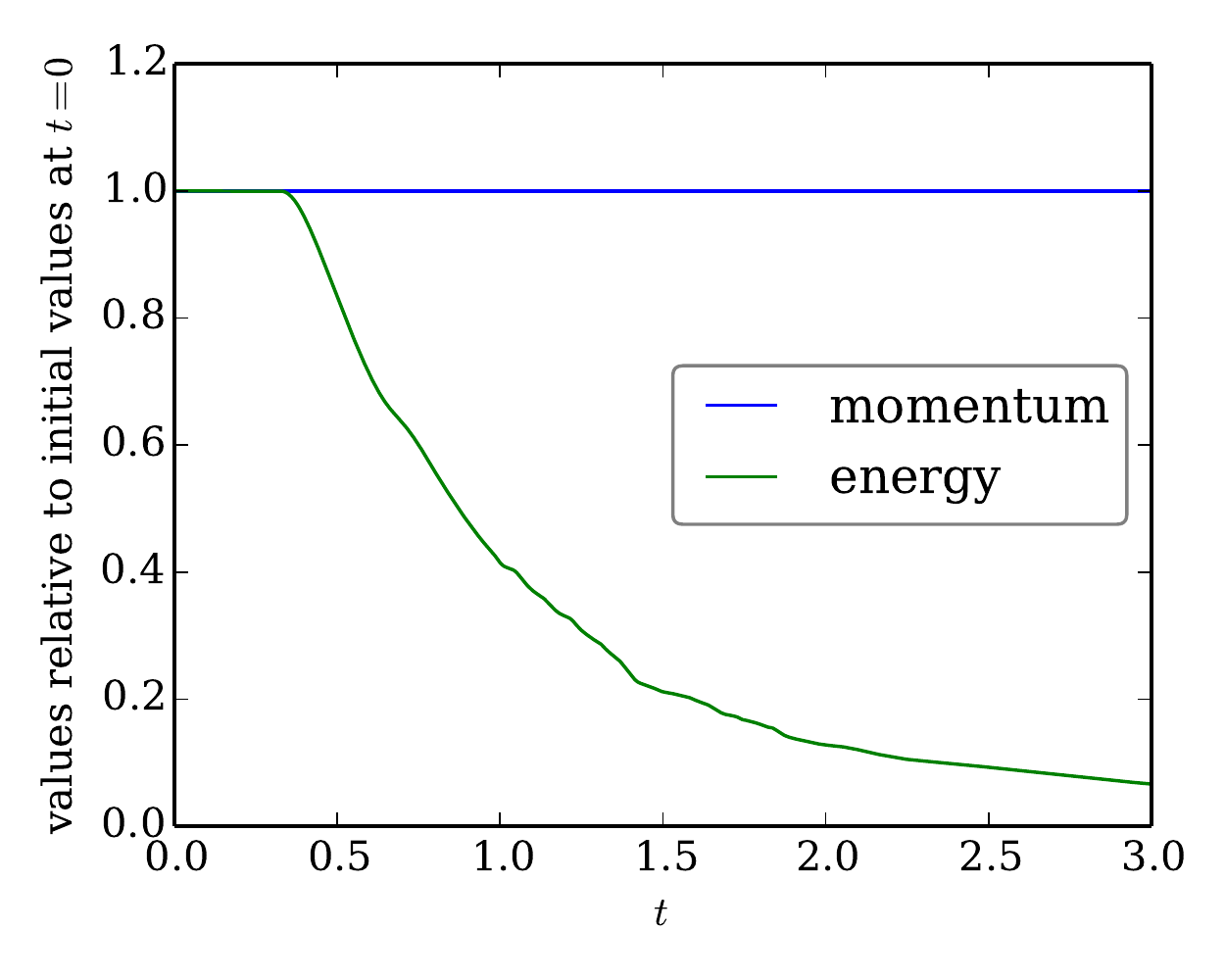}
    \caption{Osher's flux.}
  \end{subfigure}%
  \caption{Results of the simulations for Burgers' equation
            using an SBP CPR method with correction term for the divergence,
            20 elements with a \emph{Lobatto-Legendre}
            basis of order 7 and variant numerical fluxes.
            On the left-hand side, the values of $u(3)$ (blue) and $u(0) = u_0$
            (green) are shown. On the right-hand site, the discrete momentum
            $\vec{1}^T \mat{M} \vec{u}$ (blue) and discrete energy
            $\vec{u}^T \mat{M} \vec{u}$ (green) in the time interval $[0, 3]$
            are plotted.}
  \label{fig:gassner2013skew-lobatto}
\end{figure}

\begin{figure}[!hp]
  \centering
  \begin{subfigure}[b]{0.45\textwidth}
    \includegraphics[width=\textwidth]{%
      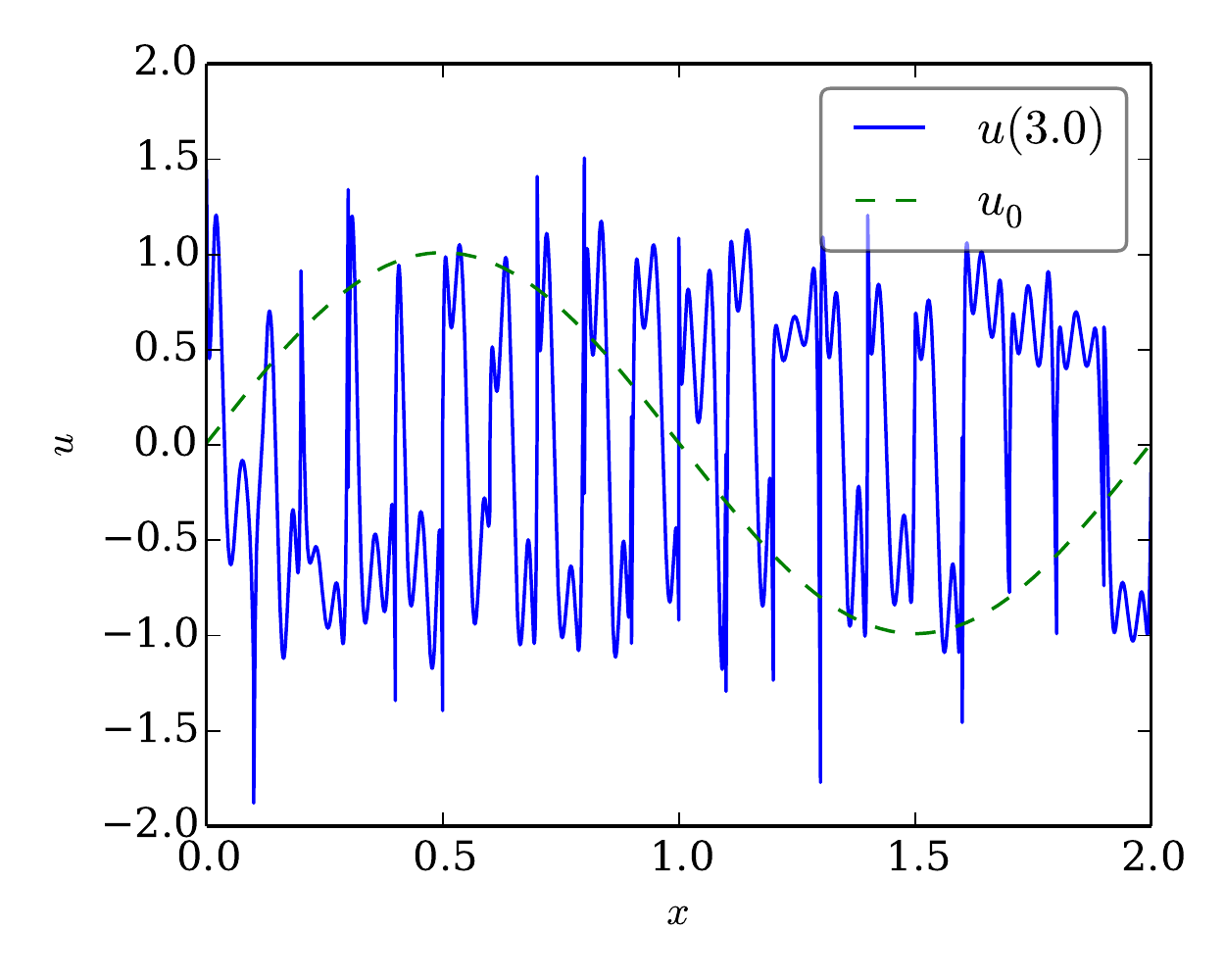}
    \caption{Energy conservative flux.}
  \end{subfigure}%
  ~
  \begin{subfigure}[b]{0.45\textwidth}
    \includegraphics[width=\textwidth]{%
      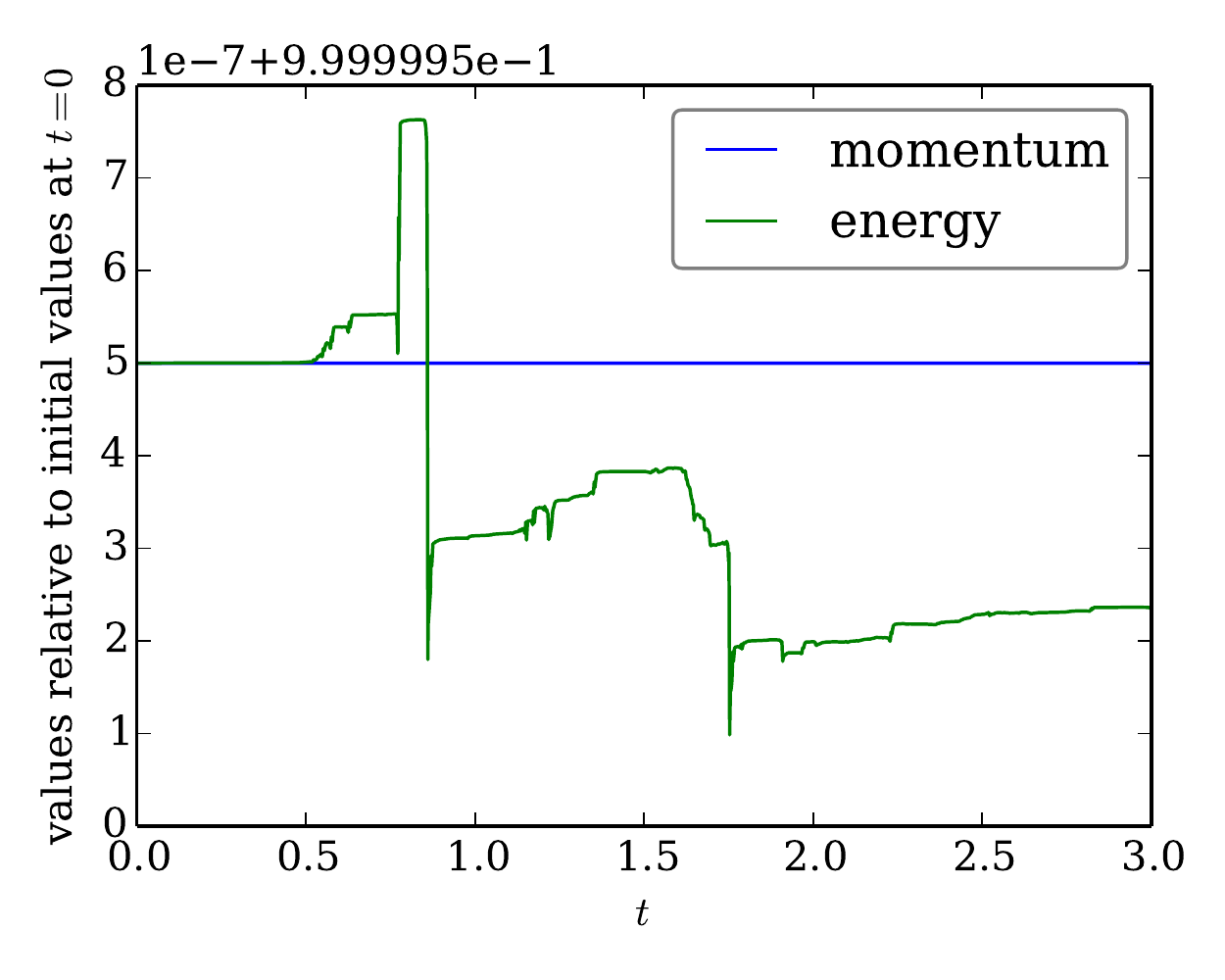}
    \caption{Energy conservative flux.}
  \end{subfigure}%
  ~\\
  \begin{subfigure}[b]{0.45\textwidth}
    \includegraphics[width=\textwidth]{%
      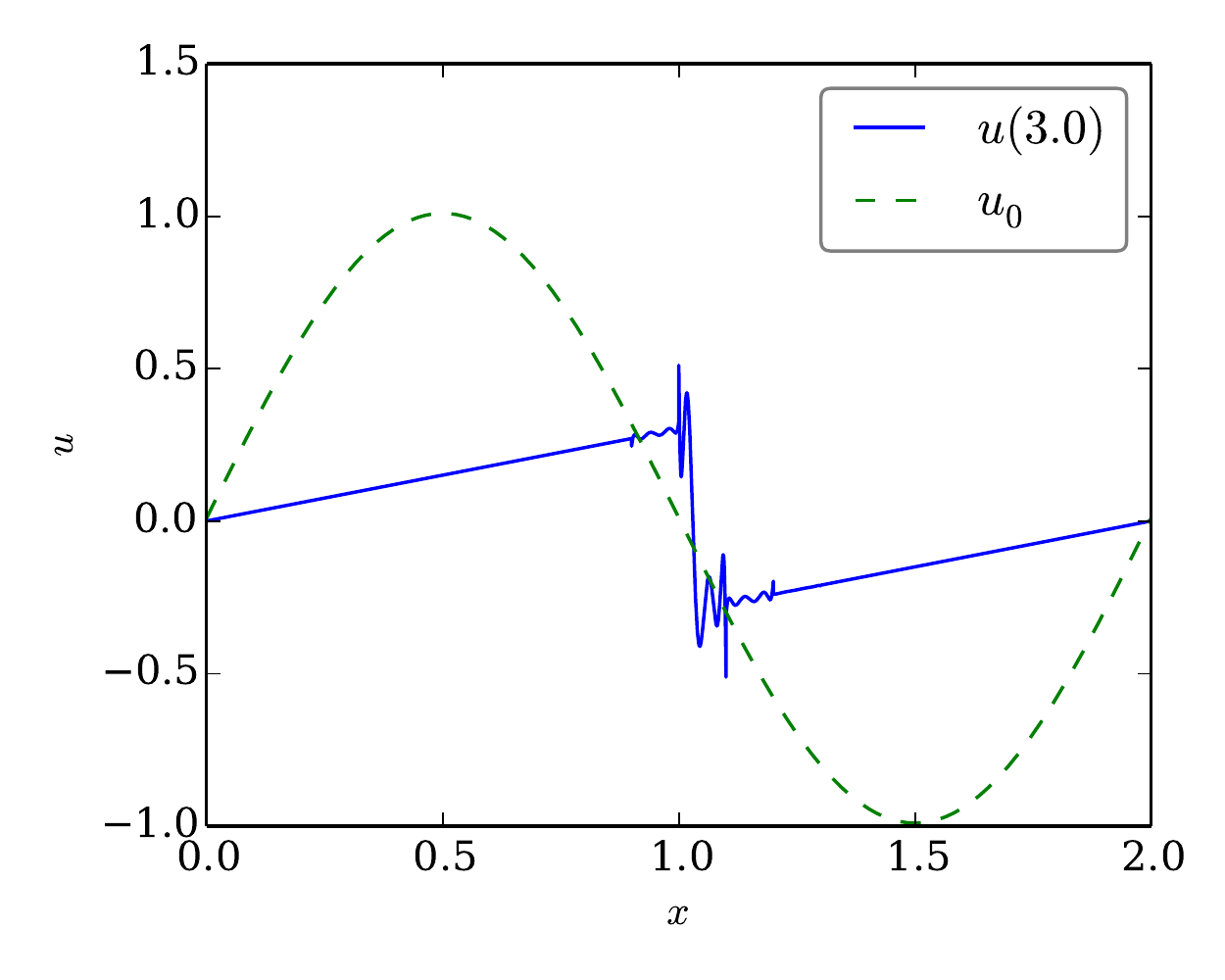}
    \caption{Local Lax-Friedrichs flux.}
  \end{subfigure}%
  ~
  \begin{subfigure}[b]{0.45\textwidth}
    \includegraphics[width=\textwidth]{%
      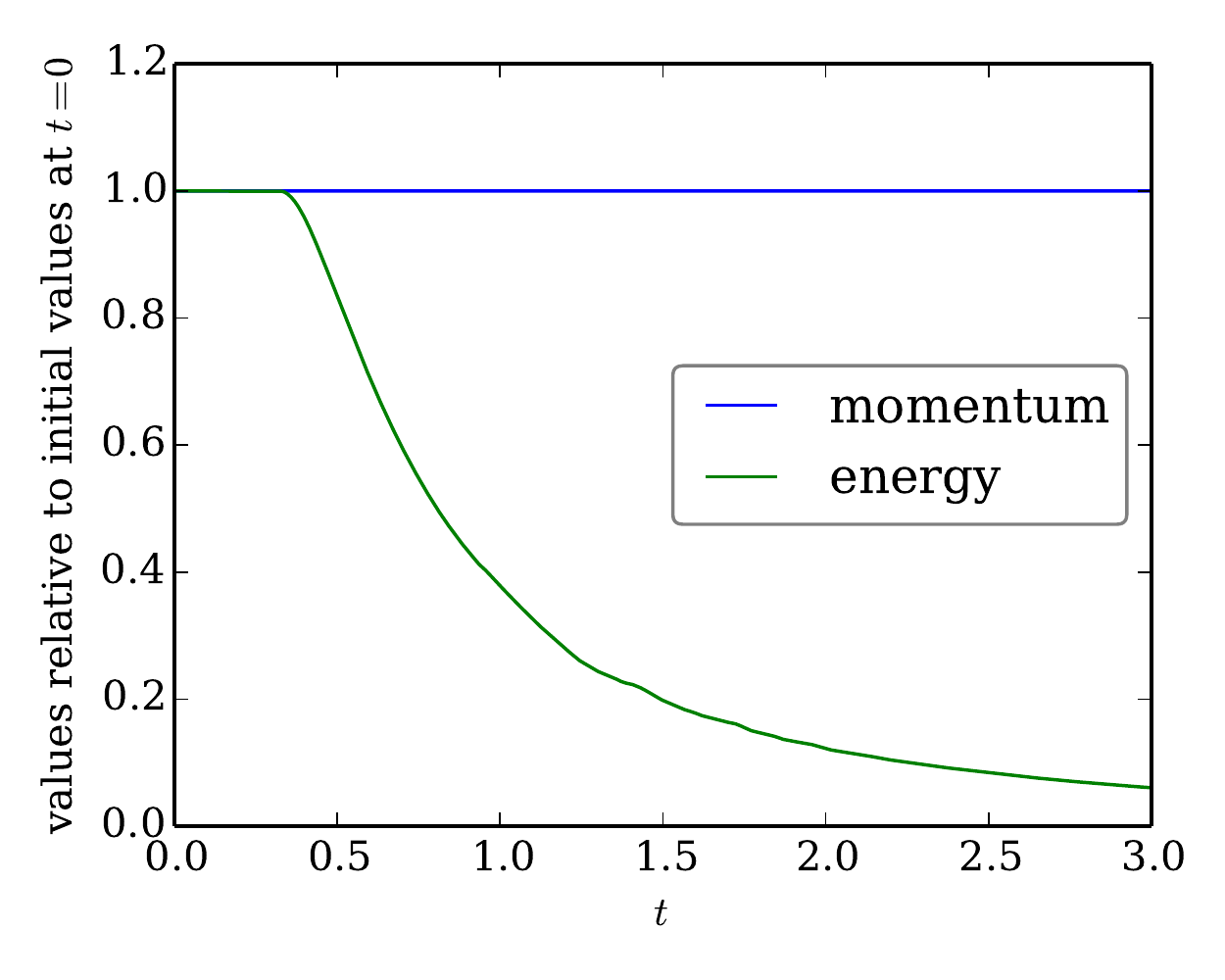}
    \caption{Local Lax-Friedrichs flux.}
  \end{subfigure}%
  ~\\
  \begin{subfigure}[b]{0.45\textwidth}
    \includegraphics[width=\textwidth]{%
      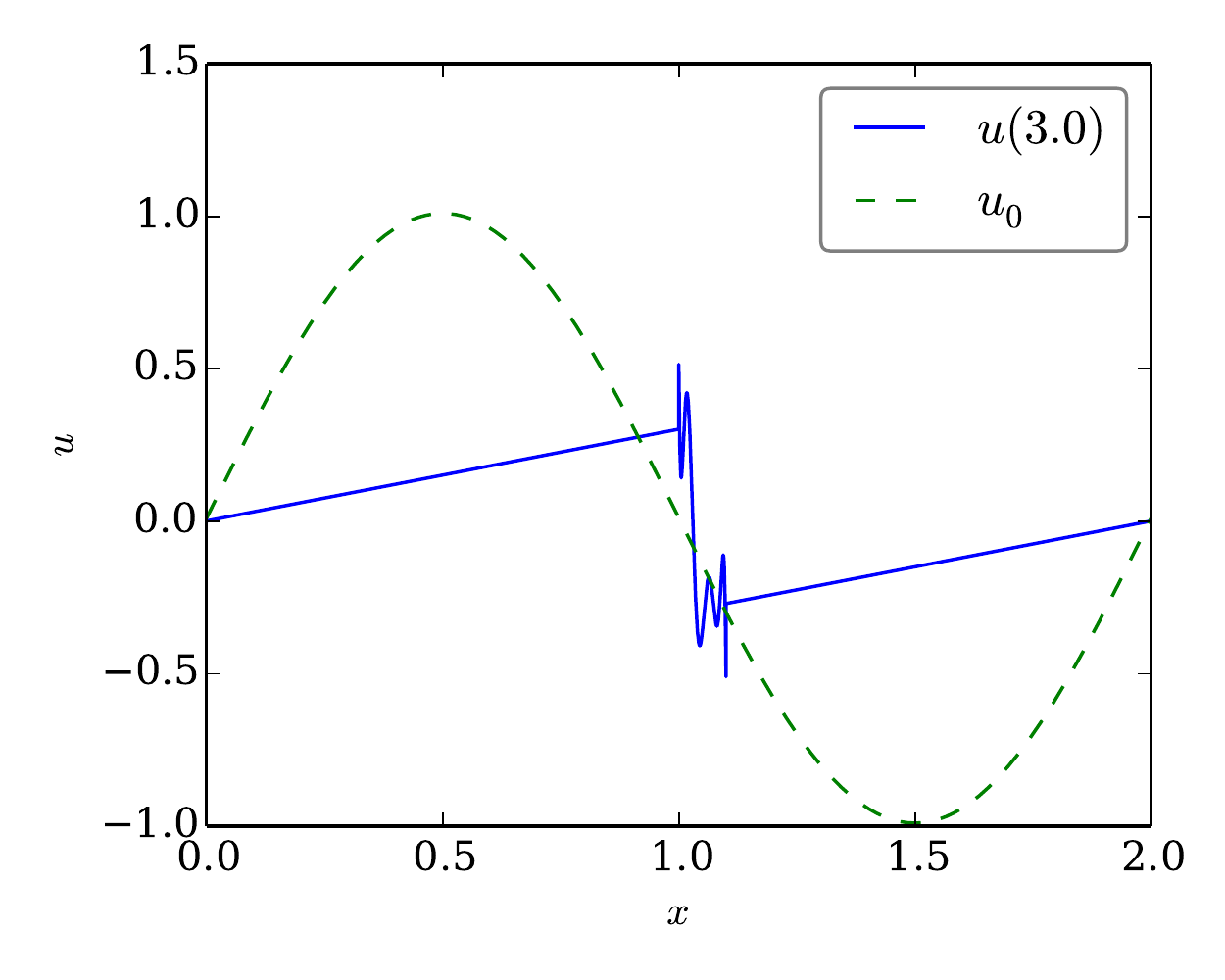}
    \caption{Osher's flux.}
  \end{subfigure}%
  ~
  \begin{subfigure}[b]{0.45\textwidth}
    \includegraphics[width=\textwidth]{%
      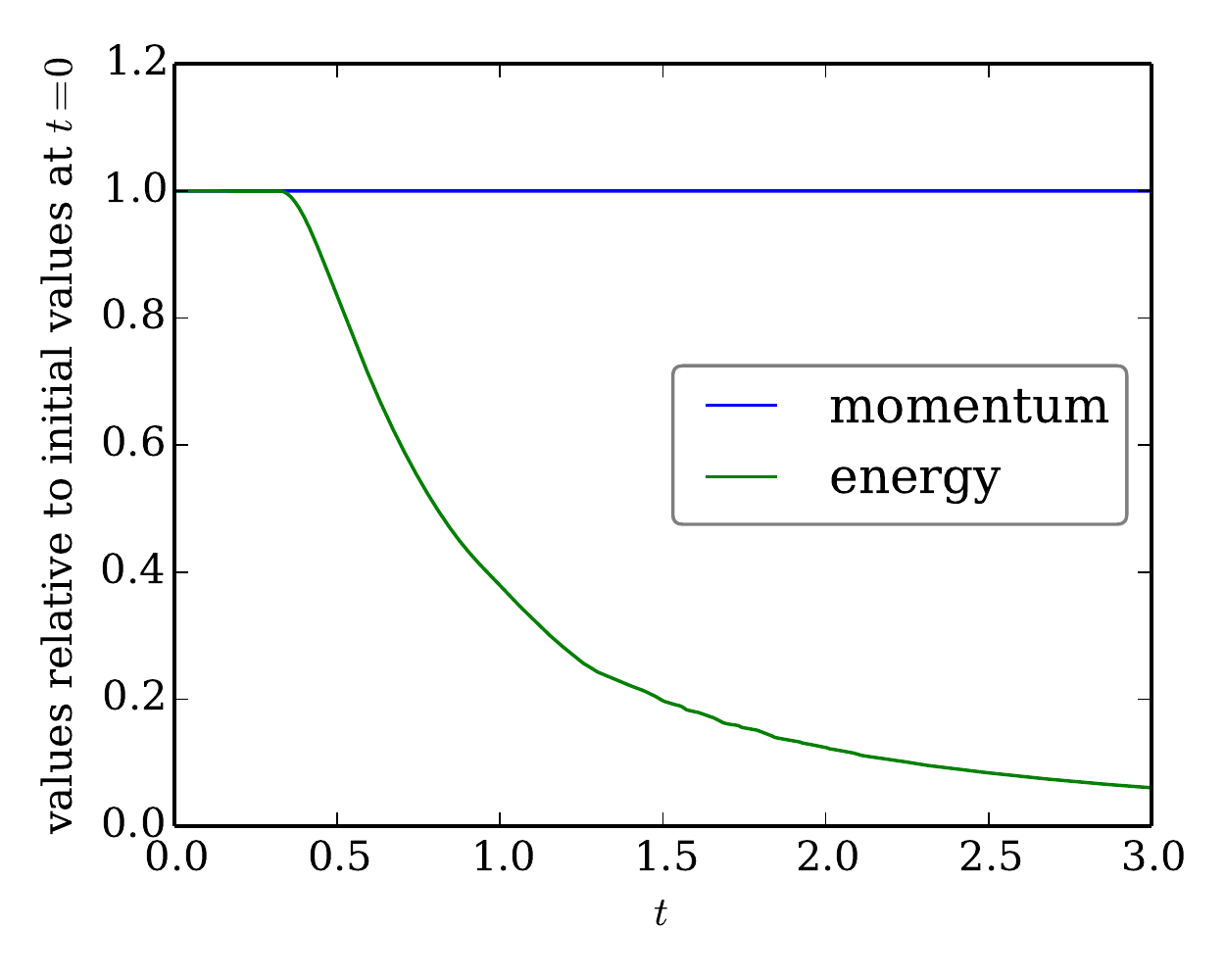}
    \caption{Osher's flux.}
  \end{subfigure}%
  \caption{Results of the simulations for Burgers' equation
            using an SBP CPR method with correction terms for both the divergence
            and restriction to the boundary,
            20 elements with a \emph{Gauß-Legendre}
            basis of order 7 and variant numerical fluxes.
            On the left-hand side, the values of $u(3)$ (blue) and $u(0) = u_0$
            (green) are shown. On the right-hand site, the discrete momentum
            $\vec{1}^T \mat{M} \vec{u}$ (blue) and discrete energy
            $\vec{u}^T \mat{M} \vec{u}$ (green) in the time interval $[0, 3]$
            are plotted.}
  \label{fig:gassner2013skew-gauss-with-correction}
\end{figure}

\begin{figure}[!hp]
  \centering
  \begin{subfigure}[b]{0.45\textwidth}
    \includegraphics[width=\textwidth]{%
      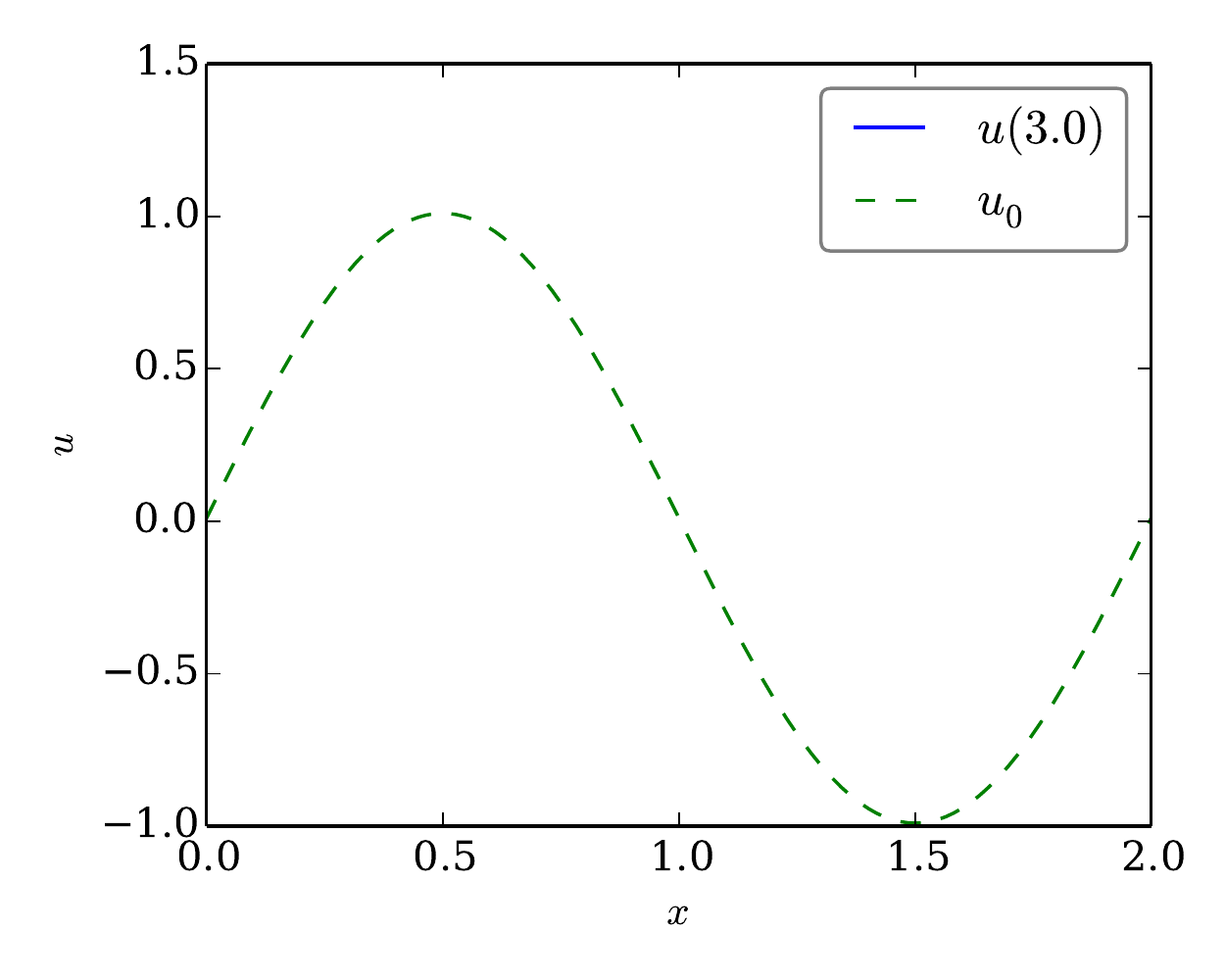}
    \caption{Energy conservative flux.}
  \end{subfigure}%
  ~
  \begin{subfigure}[b]{0.45\textwidth}
    \includegraphics[width=\textwidth]{%
      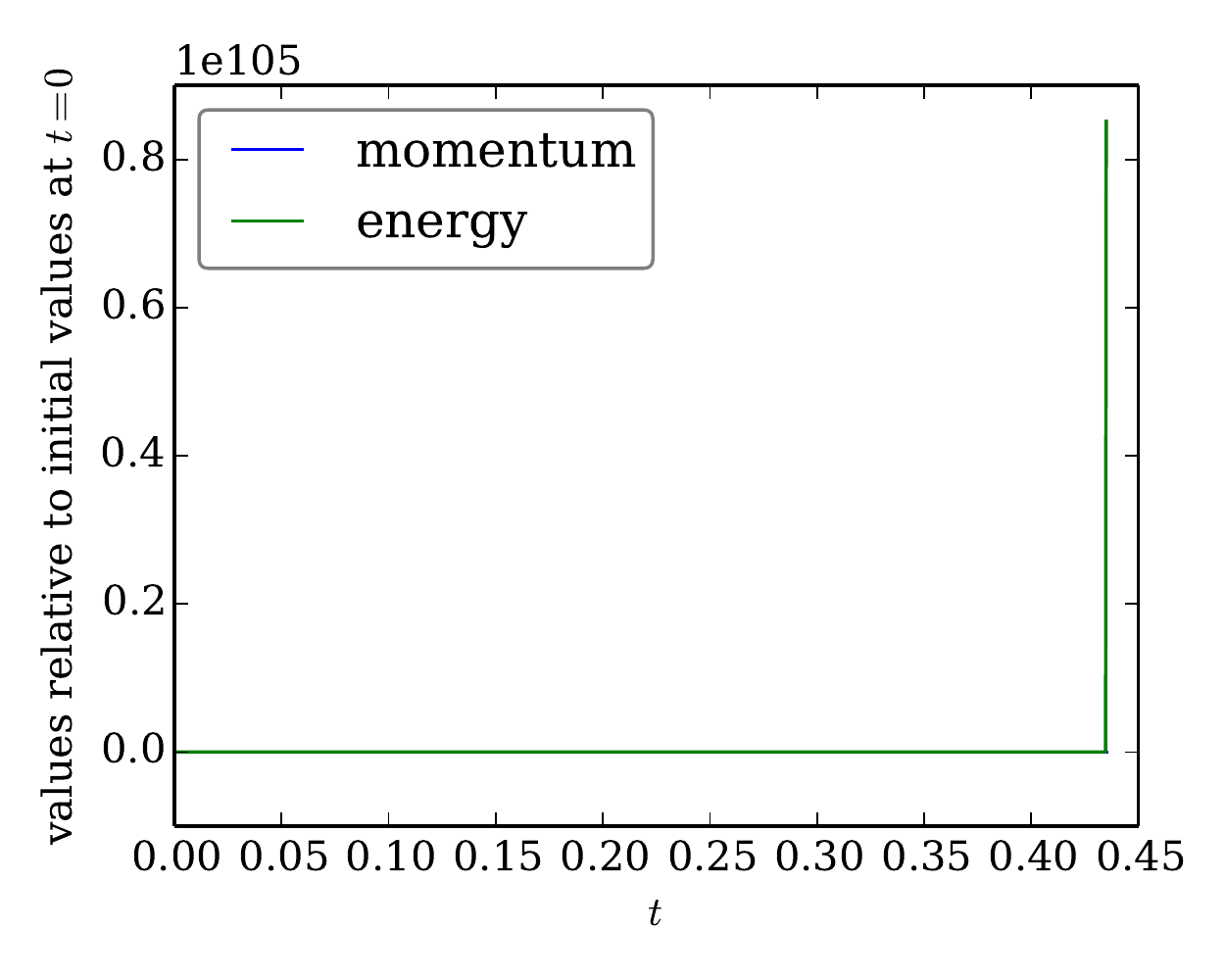}
    \caption{Energy conservative flux.}
  \end{subfigure}%
  ~\\
  \begin{subfigure}[b]{0.45\textwidth}
    \includegraphics[width=\textwidth]{%
      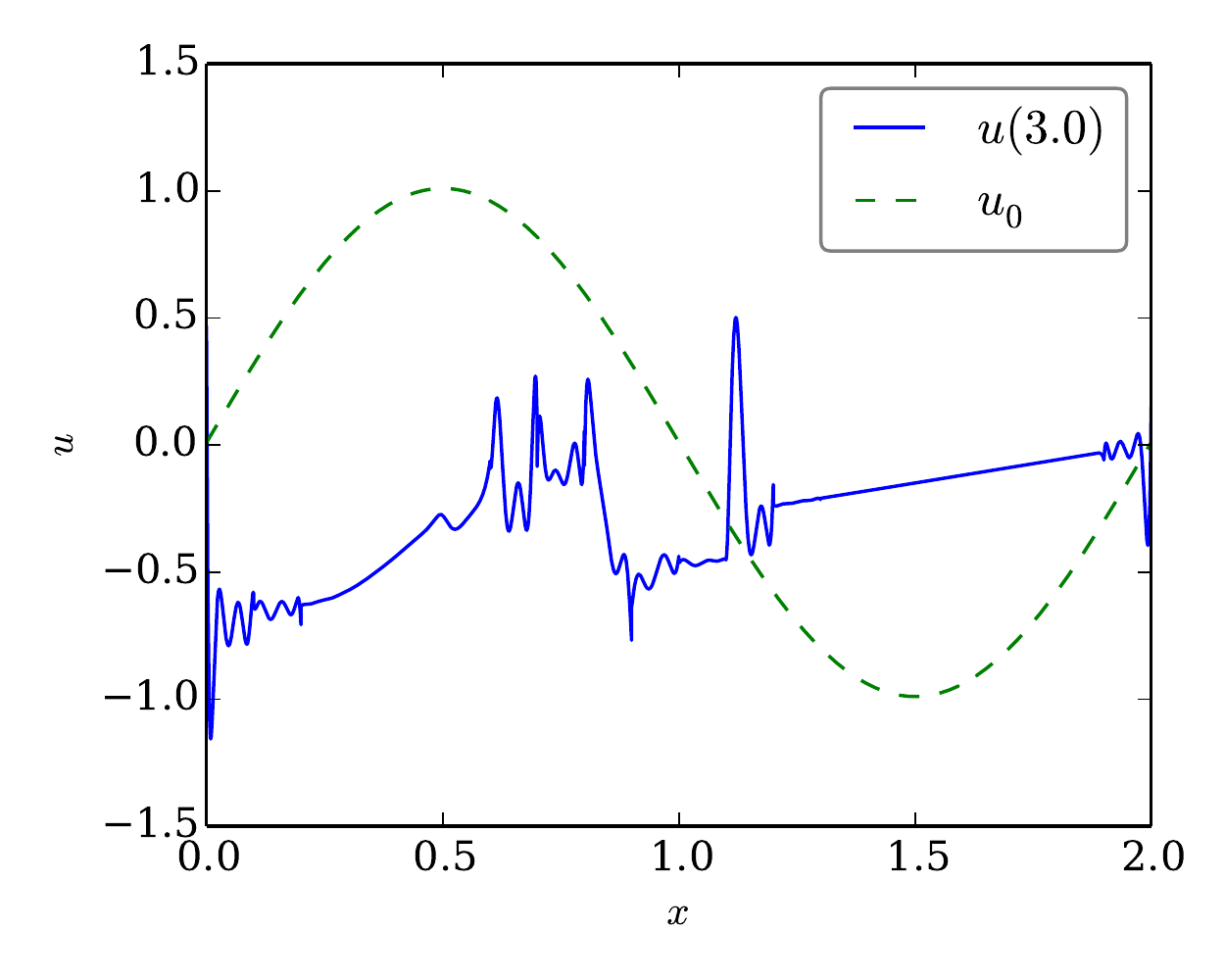}
    \caption{Local Lax-Friedrichs flux.}
  \end{subfigure}%
  ~
  \begin{subfigure}[b]{0.45\textwidth}
    \includegraphics[width=\textwidth]{%
      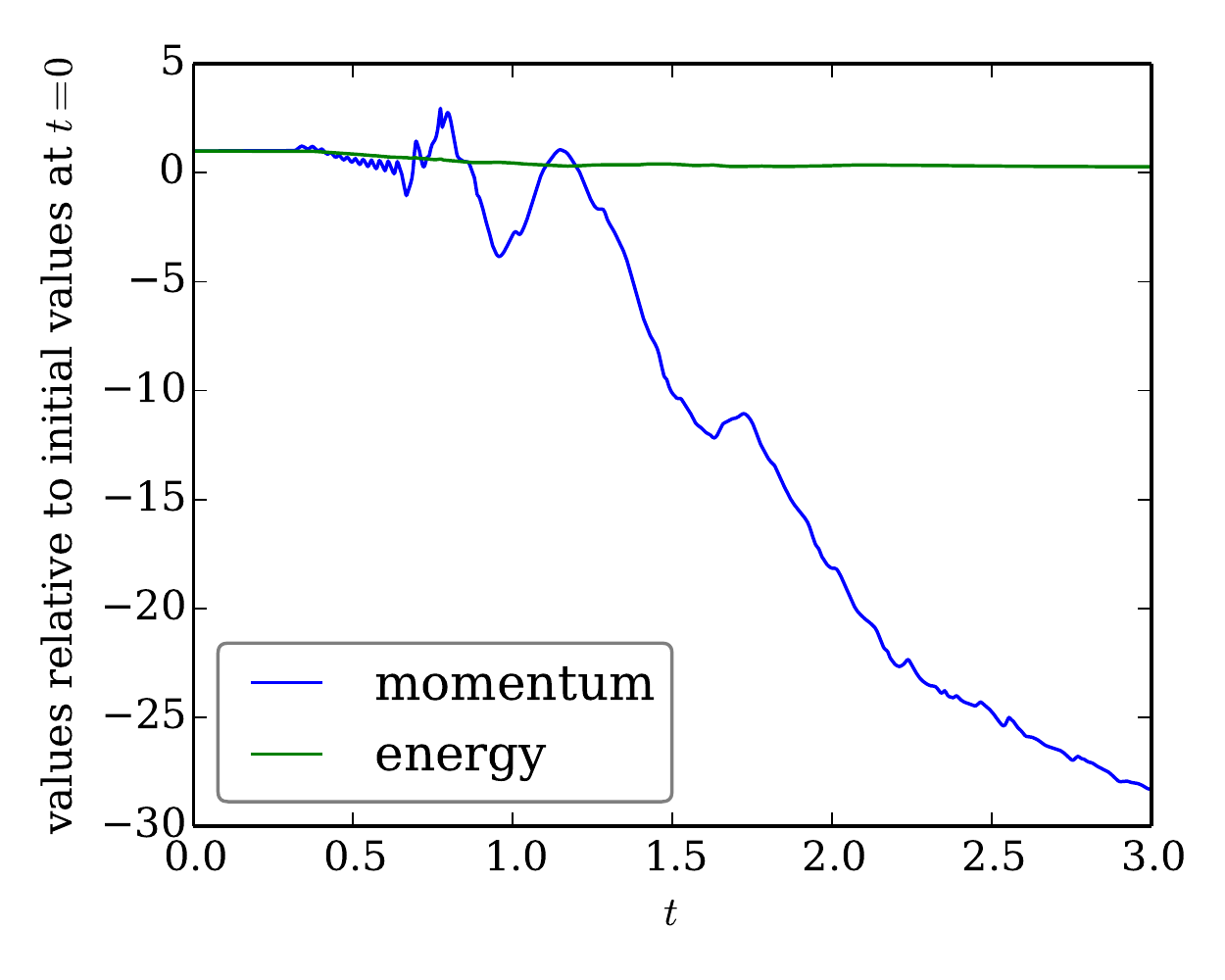}
    \caption{Local Lax-Friedrichs flux.}
  \end{subfigure}%
  ~\\
  \begin{subfigure}[b]{0.45\textwidth}
    \includegraphics[width=\textwidth]{%
      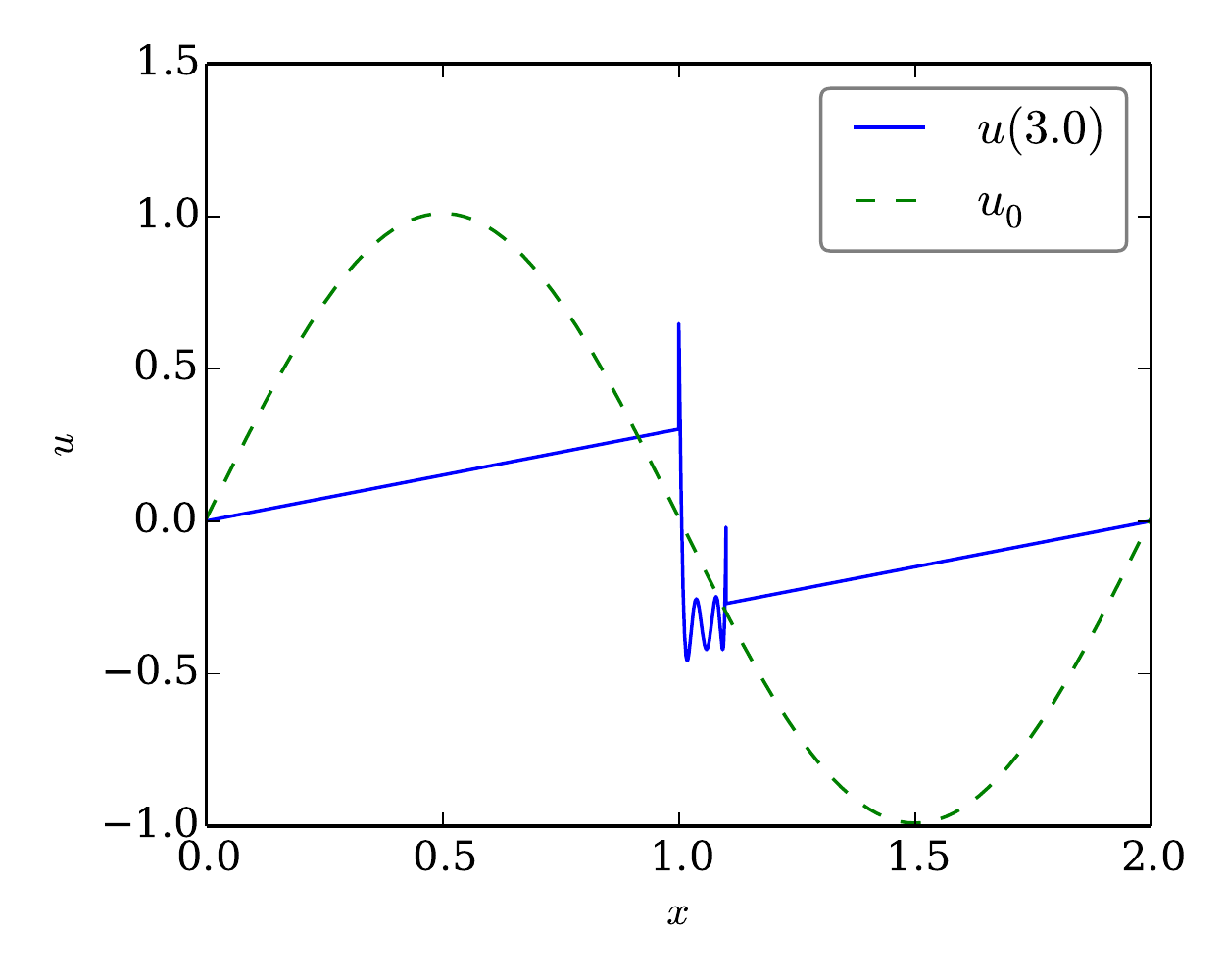}
    \caption{Osher's flux.}
  \end{subfigure}%
  ~
  \begin{subfigure}[b]{0.45\textwidth}
    \includegraphics[width=\textwidth]{%
      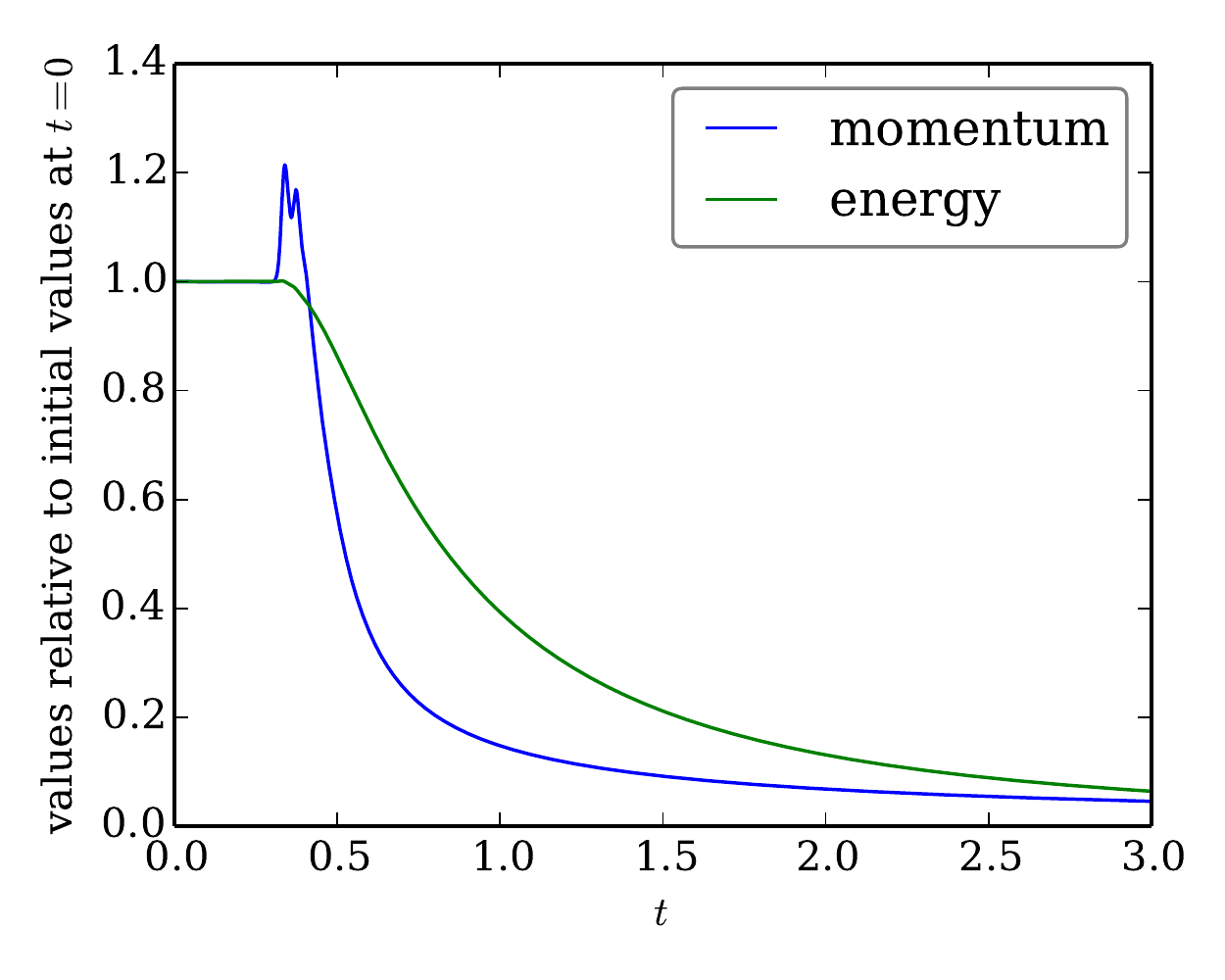}
    \caption{Osher's flux.}
  \end{subfigure}%
  \caption{Results of the simulations for Burgers' equation
            using an SBP CPR method with correction term only for the divergence,
            20 elements with a \emph{Gauß-Legendre}
            basis of order 7 and variant numerical fluxes.
            On the left-hand side, the values of $u(3)$ (blue) and $u(0) = u_0$
            (green) are shown. On the right-hand site, the discrete momentum
            $\vec{1}^T \mat{M} \vec{u}$ (blue) and discrete energy
            $\vec{u}^T \mat{M} \vec{u}$ (green) in the time interval $[0, 3]$
            are plotted.}
  \label{fig:gassner2013skew-gauss-without-correction}
\end{figure}

\begin{figure}[!hp]
  \centering
  \begin{subfigure}[b]{0.45\textwidth}
    \includegraphics[width=\textwidth]{%
      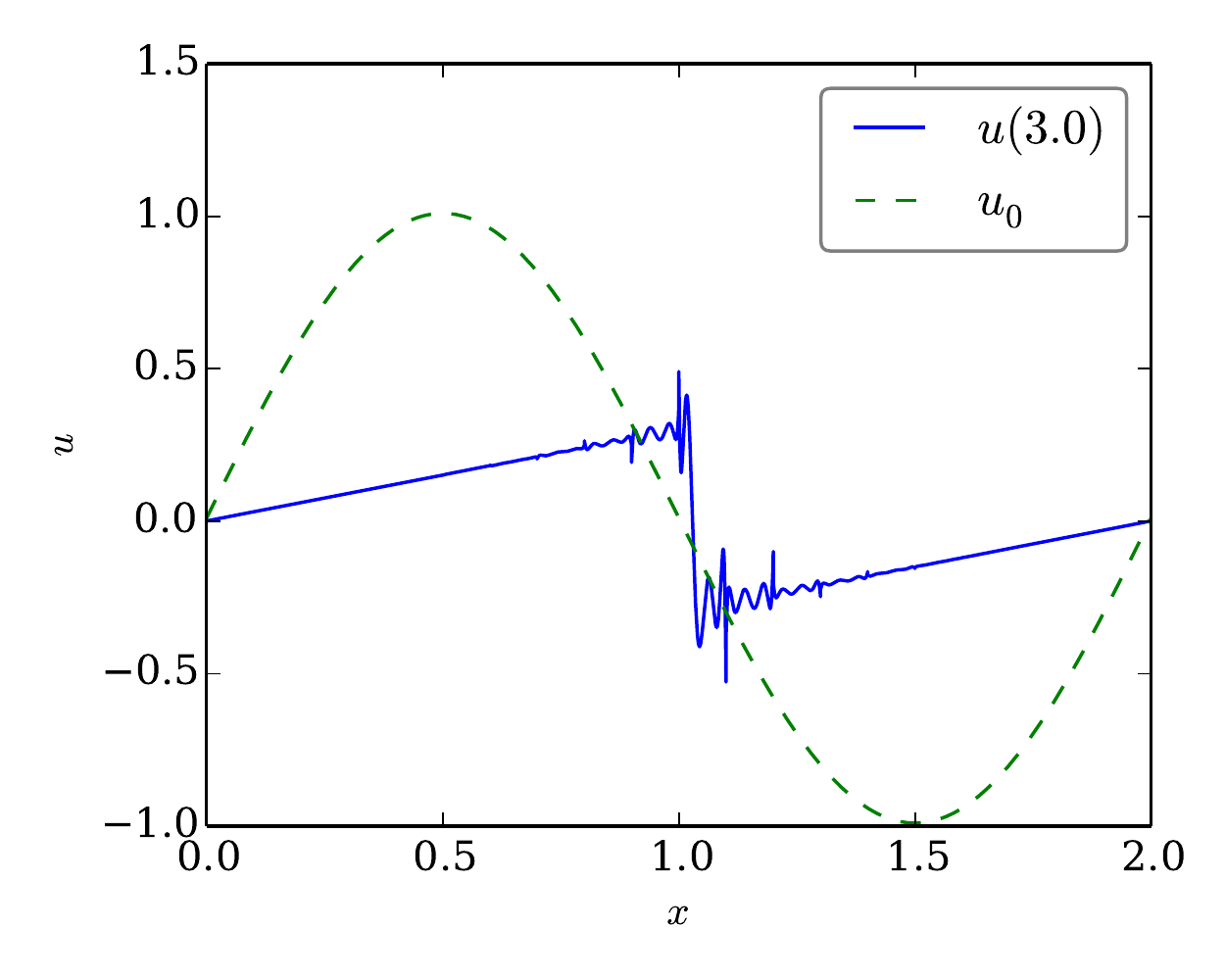}
    \caption{Gauß-Legendre with both correction terms.}
  \end{subfigure}%
  ~
  \begin{subfigure}[b]{0.45\textwidth}
    \includegraphics[width=\textwidth]{%
      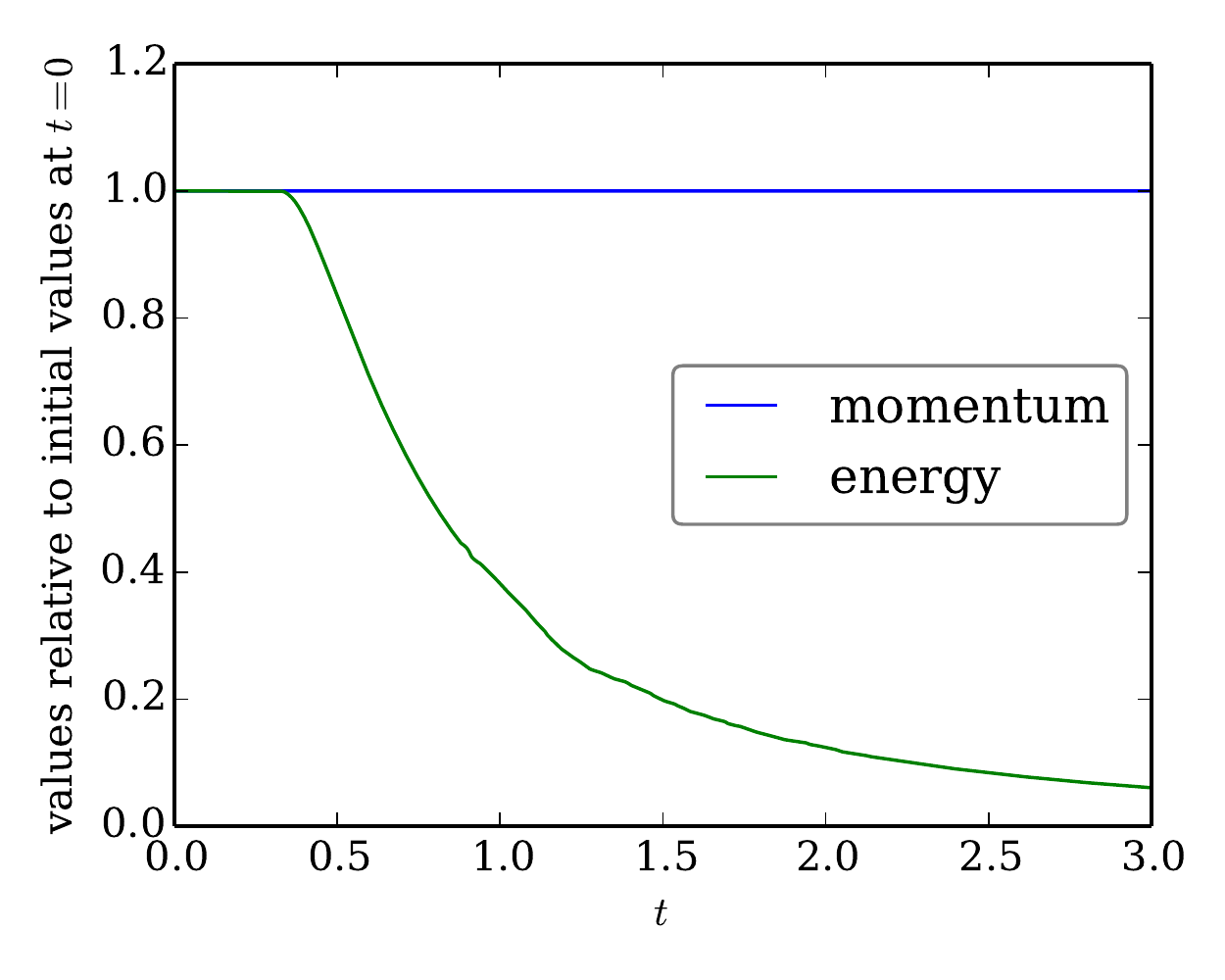}
    \caption{Gauß-Legendre with both correction terms.}
  \end{subfigure}%
  ~\\
  \begin{subfigure}[b]{0.45\textwidth}
    \includegraphics[width=\textwidth]{%
      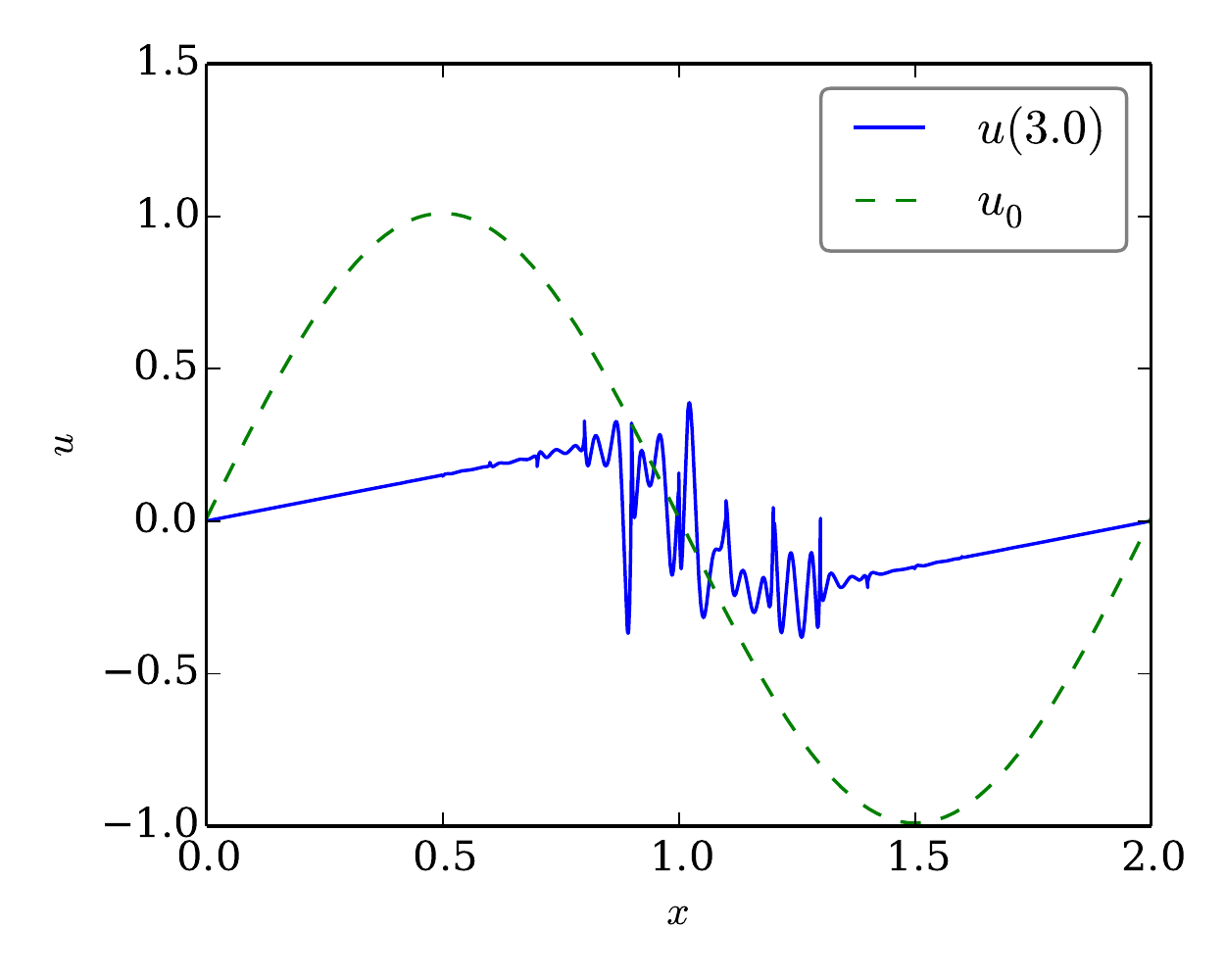}
    \caption{Gauß-Legendre with divergence correction.}
  \end{subfigure}%
  ~
  \begin{subfigure}[b]{0.45\textwidth}
    \includegraphics[width=\textwidth]{%
      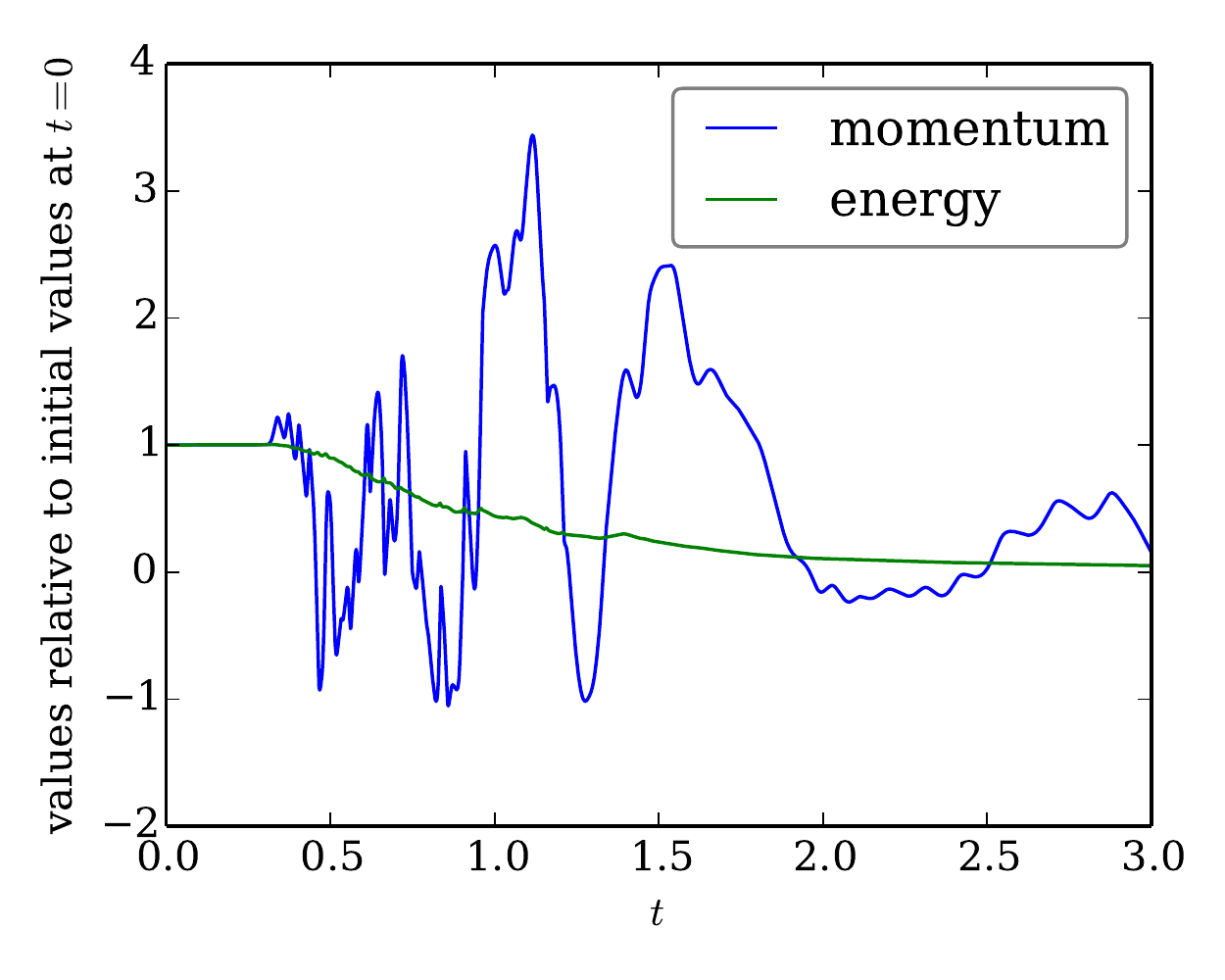}
    \caption{Gauß-Legendre with divergence correction.}
  \end{subfigure}%
  ~\\
  \begin{subfigure}[b]{0.45\textwidth}
    \includegraphics[width=\textwidth]{%
      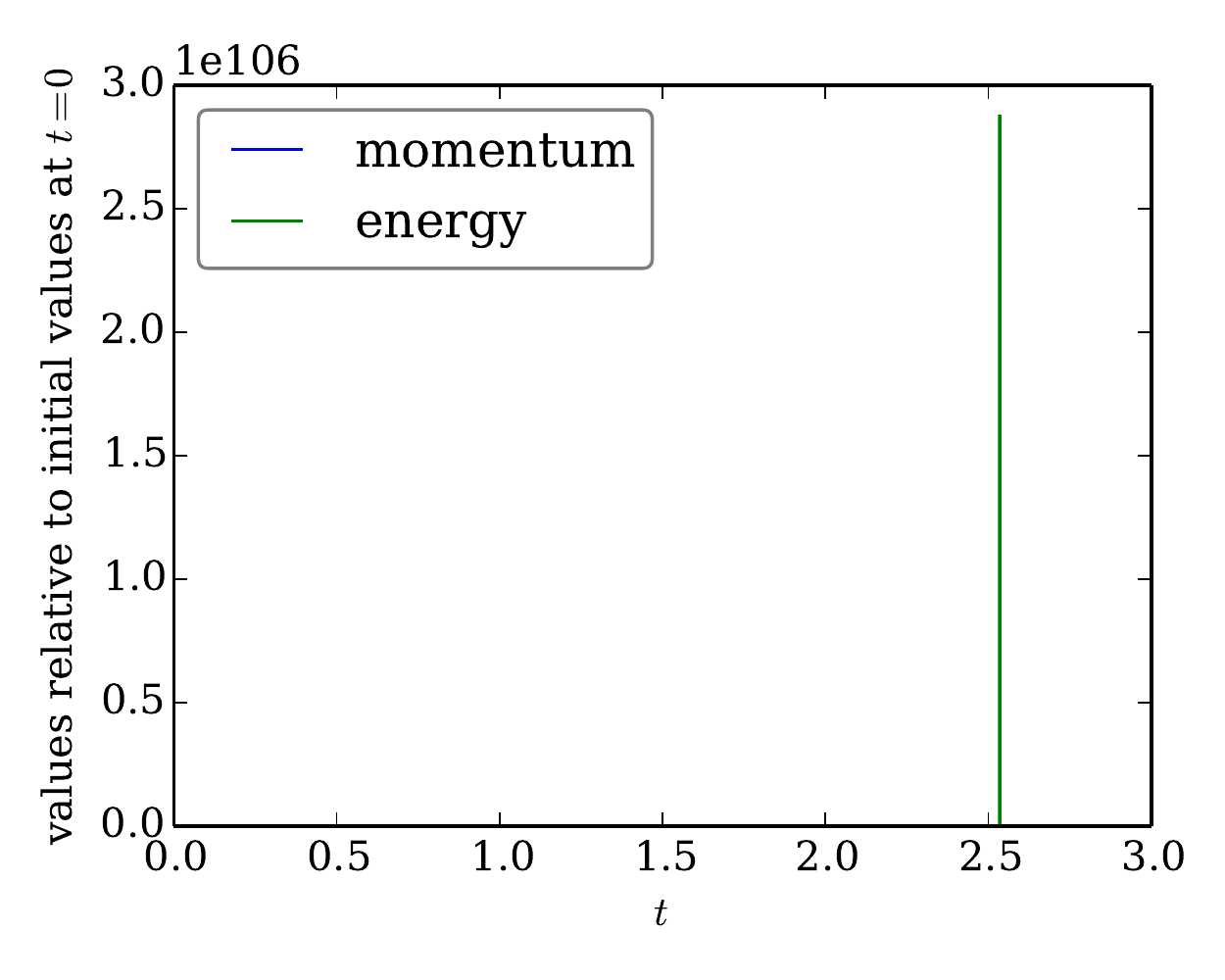}
    \caption{Lobatto-Legendre with divergence correction.}
  \end{subfigure}
  \caption{Results for Burgers' equation
            using SBP CPR methods with 20 elements, different bases of order 7
            and Roe's flux.
            In the first row, results for Gauß-Legendre nodes with
            both correction terms are shown.
            For the second row, only a divergence correction is used.
            Results for a Lobatto-Legendre basis with divergence correction are
            presented in (e).
            In (a) and (c), the values of $u(3)$ (blue) and $u(0) = u_0$
            (green) are shown. The discrete momentum
            $\vec{1}^T \mat{M} \vec{u}$ (blue) and discrete energy
            $\vec{u}^T \mat{M} \vec{u}$ (green) in the time interval $[0, 3]$
            are plotted in the other pictures.}
  \label{fig:gassner2013skew-roe}
\end{figure}

\begin{figure}[!hp]
  \centering
  \begin{subfigure}[b]{0.45\textwidth}
    \includegraphics[width=\textwidth]{%
      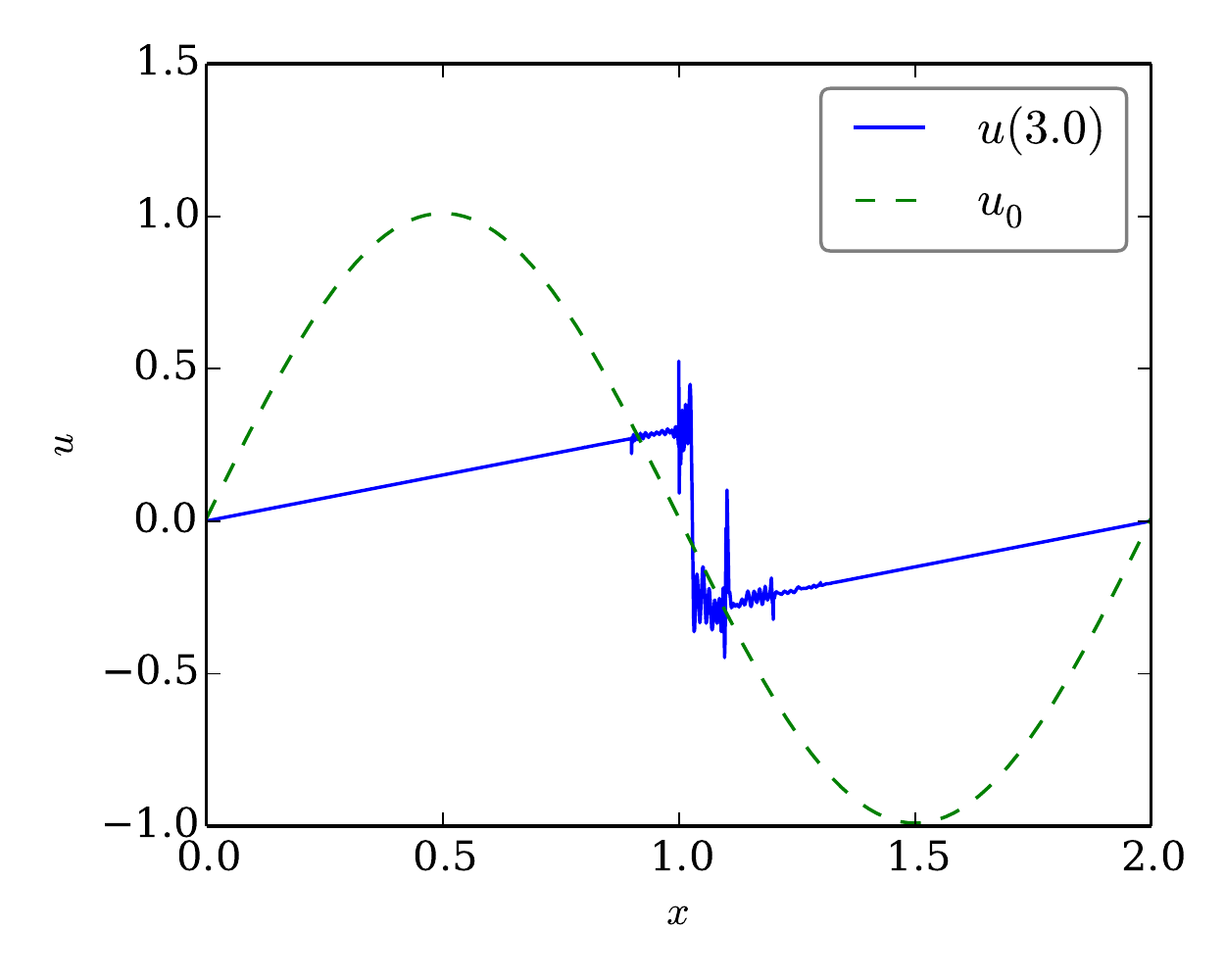}
    \caption{Gauß-Legendre with LLF flux.}
  \end{subfigure}%
  ~
  \begin{subfigure}[b]{0.45\textwidth}
    \includegraphics[width=\textwidth]{%
      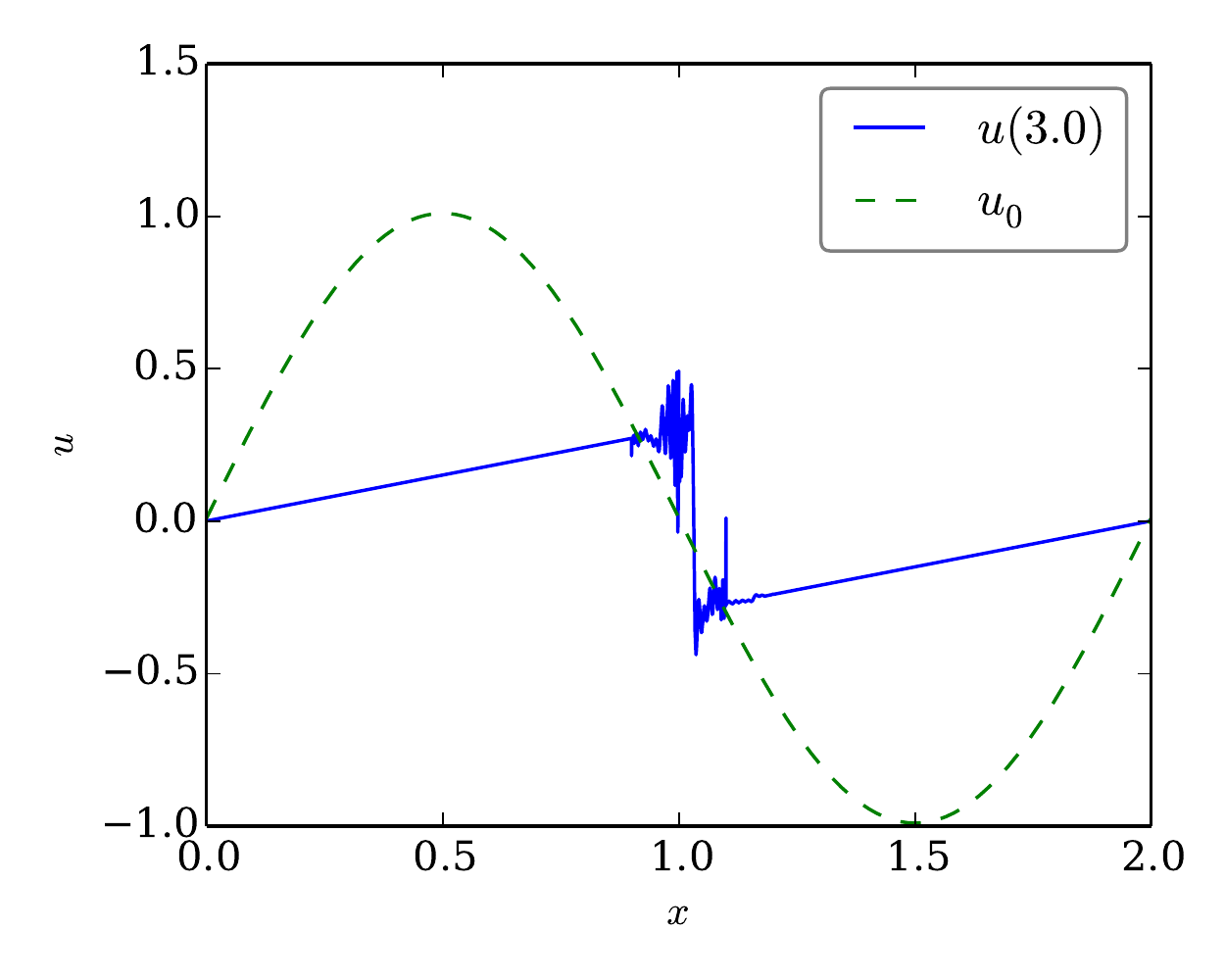}
    \caption{Gauß-Legendre with Osher's flux.}
  \end{subfigure}%
  \\
  \begin{subfigure}[b]{0.45\textwidth}
    \includegraphics[width=\textwidth]{%
      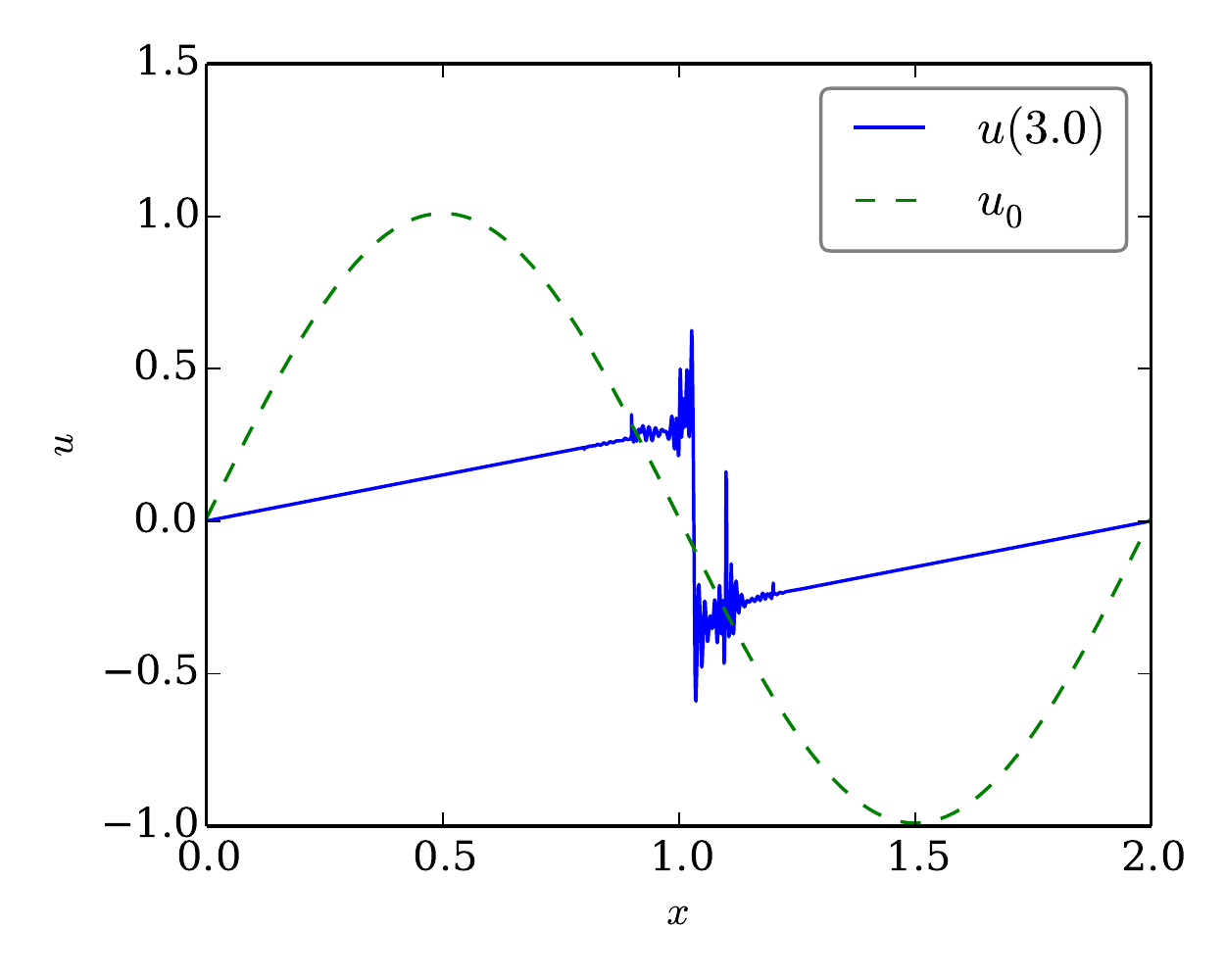}
    \caption{Lobatto-Legendre with LLF flux.}
  \end{subfigure}%
  ~
  \begin{subfigure}[b]{0.45\textwidth}
    \includegraphics[width=\textwidth]{%
      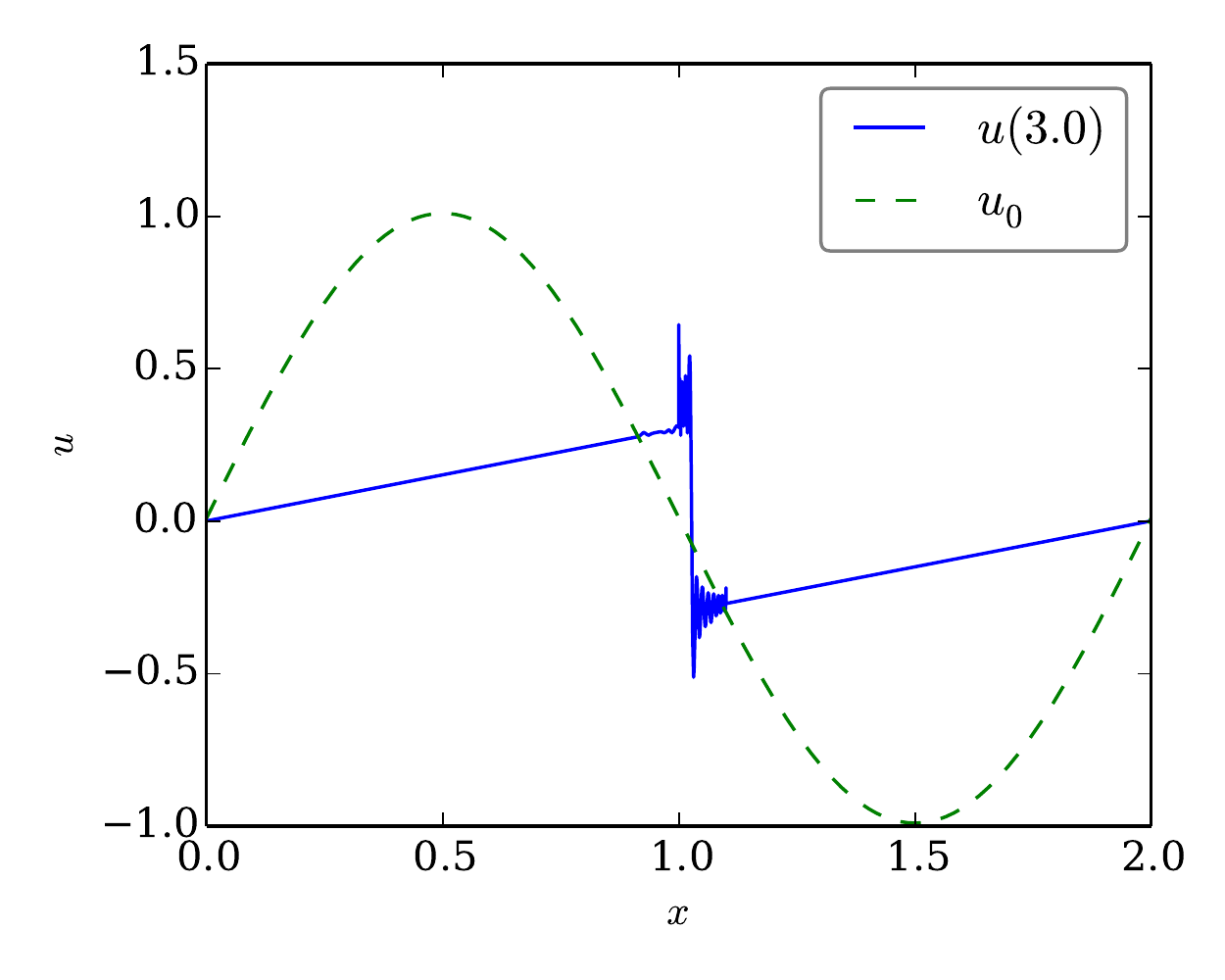}
    \caption{Lobatto-Legendre with Osher's flux.}
  \end{subfigure}%
  \caption{Results of the simulations for Burgers' equation
            using SBP CPR methods with 20 elements, different bases of order 25
            and local Lax-Friedrichs (LLF) or Osher's flux (on the left- and
            right-hand side, respectively).
            In the first row, results for the Gauß-Legendre nodes with
            correction terms for both divergence and restriction are shown.
            For the second row, a Lobatto-Legendre basis with a correction for
            the divergence is used.
            Each figure shows the values of $u(3)$ (blue) and $u(0) = u_0$
            (green).}
  \label{fig:gassner2013skew-high-order-25}
\end{figure}

\begin{figure}[!hp]
  \centering
  \begin{subfigure}[b]{0.45\textwidth}
    \includegraphics[width=\textwidth]{%
      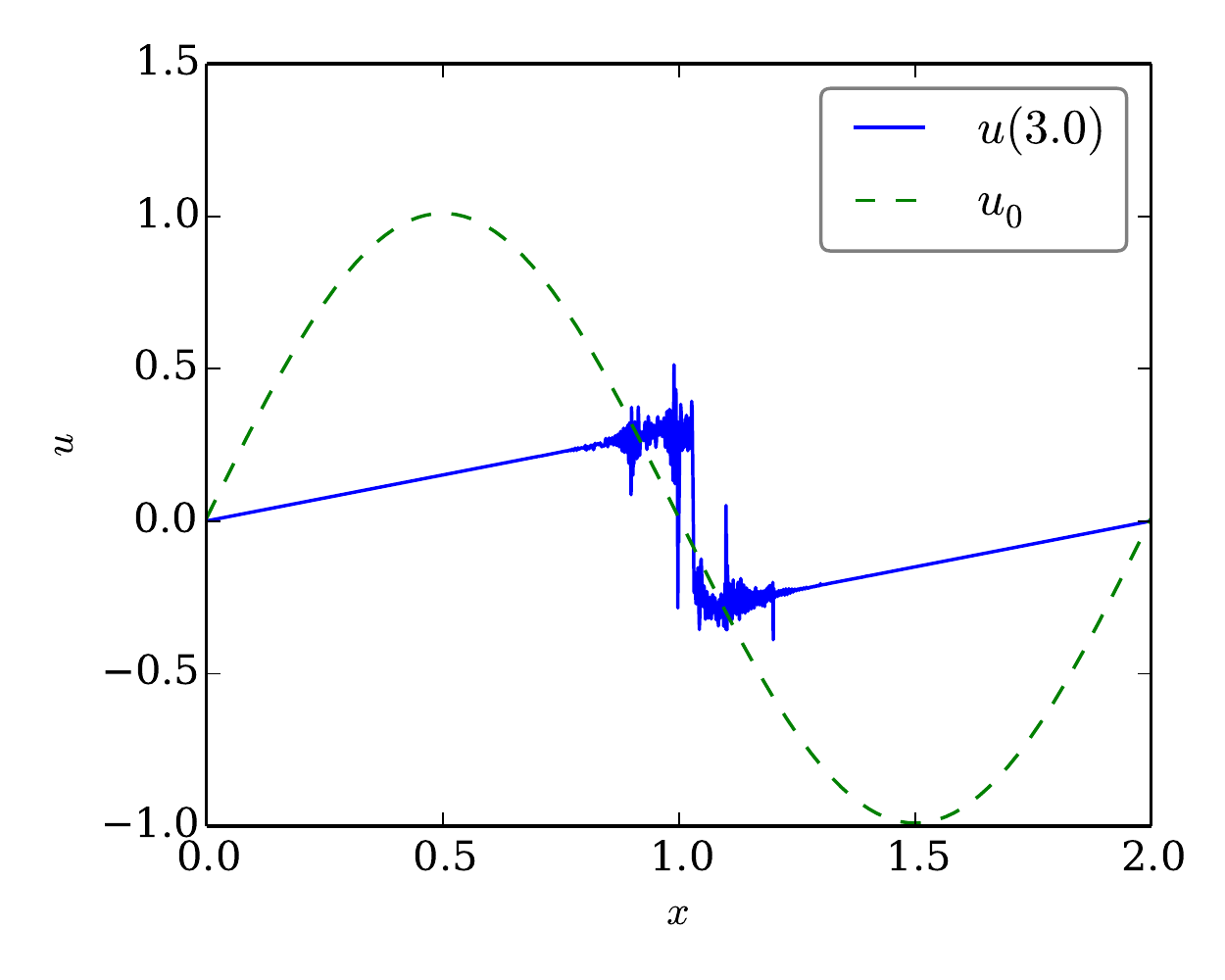}
    \caption{Gauß-Legendre with LLF flux.}
  \end{subfigure}%
  ~
  \begin{subfigure}[b]{0.45\textwidth}
    \includegraphics[width=\textwidth]{%
      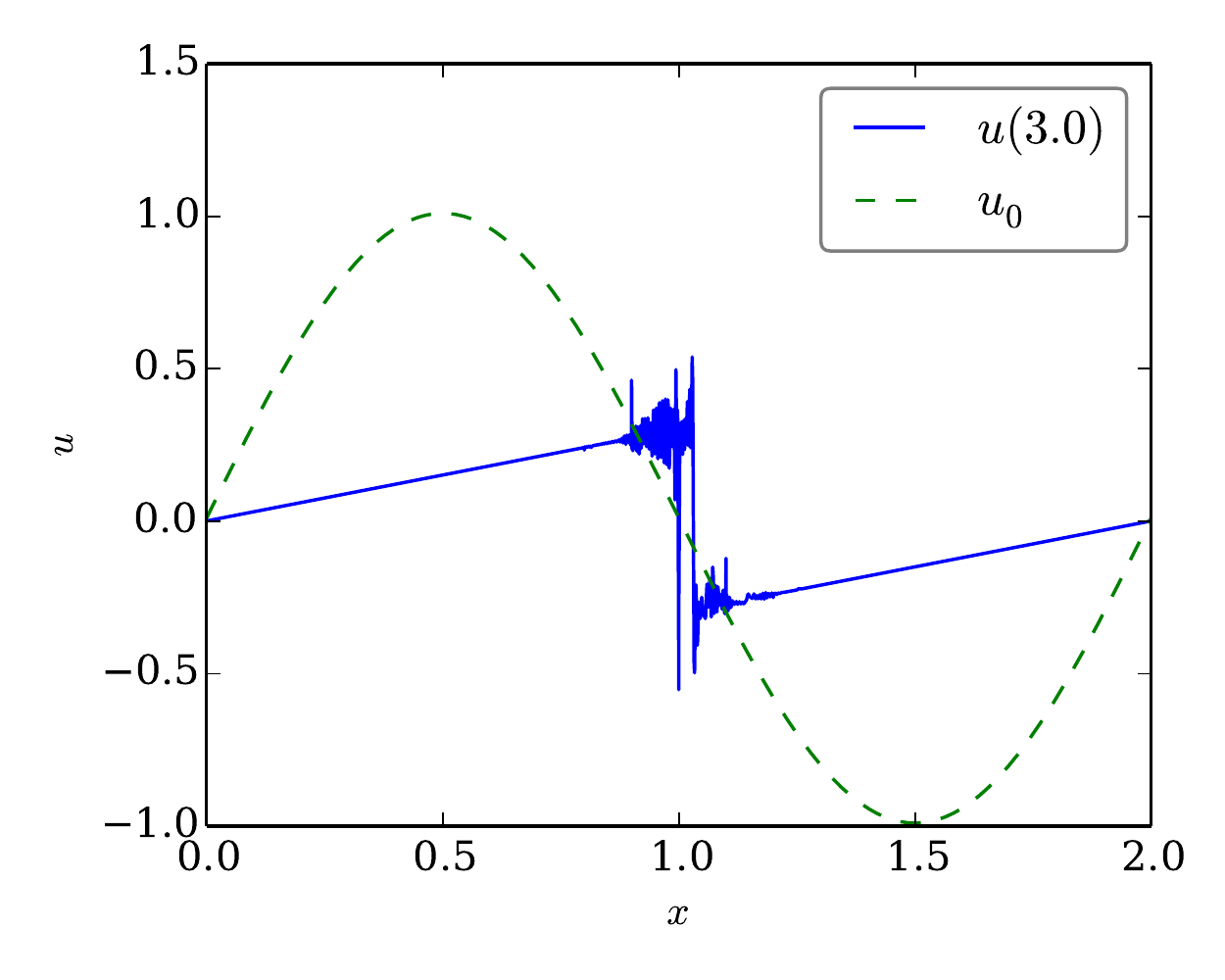}
    \caption{Gauß-Legendre with Osher's flux.}
  \end{subfigure}%
  \\
  \begin{subfigure}[b]{0.45\textwidth}
    \includegraphics[width=\textwidth]{%
      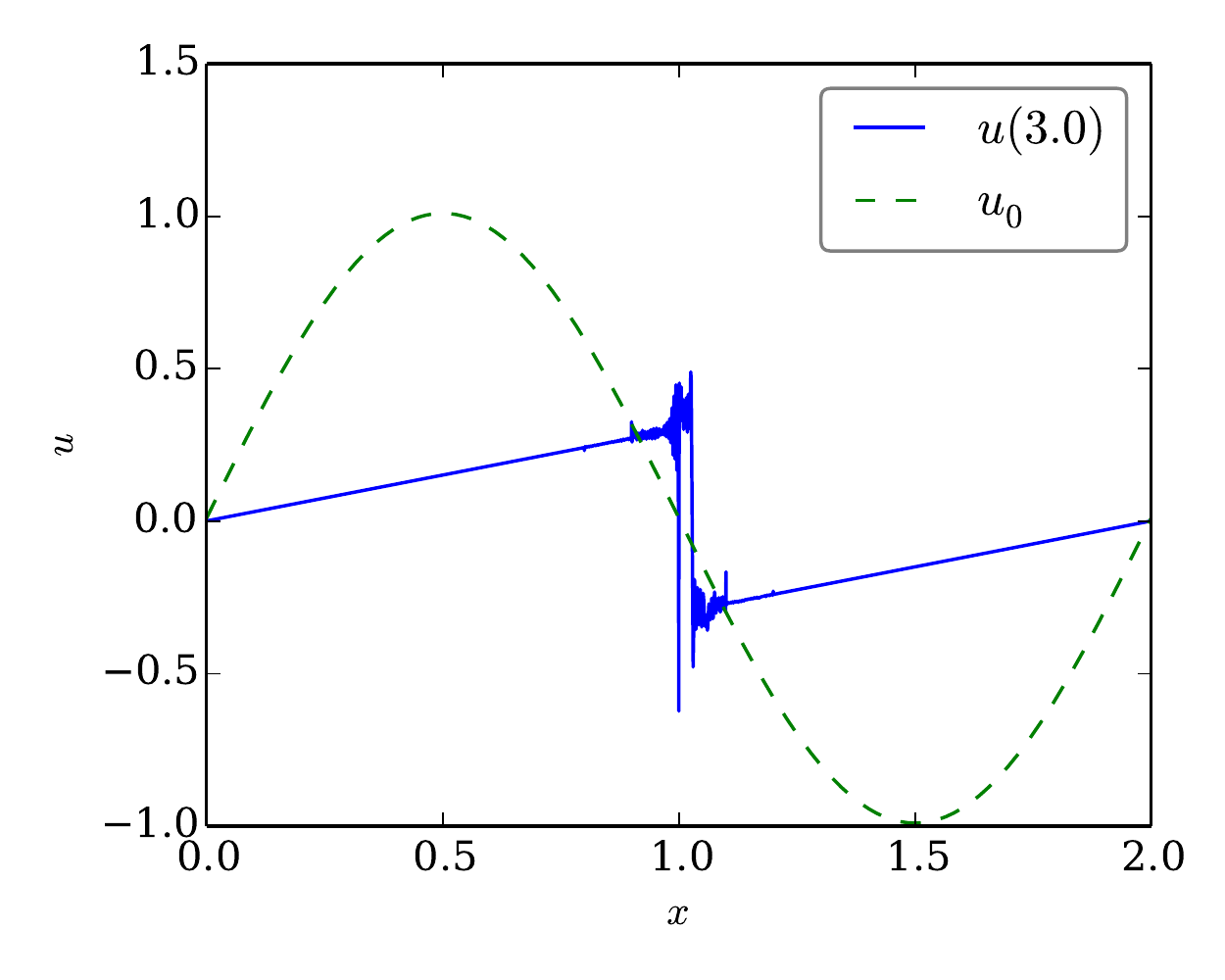}
    \caption{Lobatto-Legendre with LLF flux.}
  \end{subfigure}%
  ~
  \begin{subfigure}[b]{0.45\textwidth}
    \includegraphics[width=\textwidth]{%
      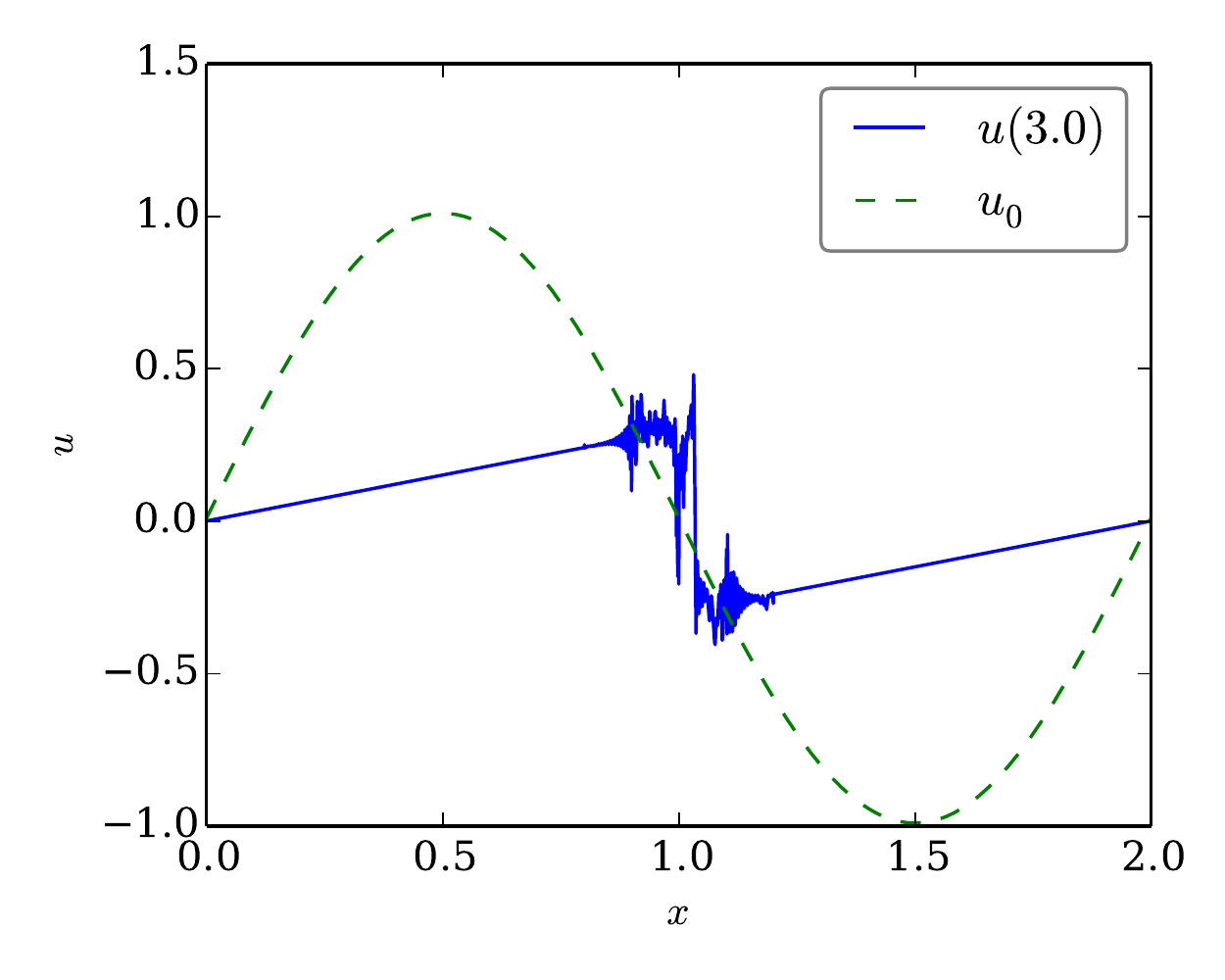}
    \caption{Lobatto-Legendre with Osher's flux.}
  \end{subfigure}%
  \caption{Results of the simulations for Burgers' equation
            using SBP CPR methods with 20 elements, different bases of order 50
            and local Lax-Friedrichs (LLF) or Osher's flux (on the left- and
            right-hand side, respectively).
            In the first row, results for the Gauß-Legendre nodes with
            correction terms for both divergence and restriction are shown.
            For the second row, a Lobatto-Legendre basis with a correction for
            the divergence is used.
            Each figure shows the values of $u(3)$ (blue) and $u(0) = u_0$
            (green).}
  \label{fig:gassner2013skew-high-order-50}
\end{figure}

\subsection{Extension of the CPR idea}
\label{sec:extension-of-the-CPR-idea}

Extending the idea to use a different norm for proving stability does not seem
to extend to the corrected formulation of Burgers' equation, at least in a
straightforward way. Indeed, multiplying by $\vec{u}^T (\mat{M} + \mat{K})$
instead of $\vec{u}^T \mat{M}$, equation
\eqref{eq:inviscid-burgers-CPR-full-correction} becomes
\begin{equation}
\begin{aligned}
  & \frac{1}{2} \od{}{t} \norm{u}_{M+K}^2
    +\frac{1}{3} \vec{u}^T (\mat{M} + \mat{K}) \mat{D} \vec{u^2}
    + \frac{1}{3} \vec{u}^T (\mat{M} + \mat{K}) \mat{u} \mat{D} \vec{u}
\\& + \vec{u}^T (\mat{M} + \mat{K}) \mat{C} (\dots)
    = 0.
\end{aligned}
\end{equation}
The standard choice $\mat{C} = (\mat{M} + \mat{K})^{-1} \mat{R}^T \mat{B}$
leads to additional terms
\begin{equation}
  \frac{1}{3} \vec{u}^T \mat{K} \mat{D} \mat{u} \vec{u}
  + \frac{1}{3} \vec{u}^T \mat{K} \mat{u} \mat{D} \vec{u}
  =
  \frac{1}{3} \vec{u}^T ( \mat{K} \mat{D} \mat{u} + \mat{K} \mat{u} \mat{D} )
    \vec{u}
\end{equation}
in comparison with the results for $\mat{K} = 0$. Enforcing stability by
requiring $\mat{K} \mat{D} \mat{u} + \mat{K} \mat{u} \mat{D}$ to be
skew-symmetric (leading to no further contribution) or symmetric (leading to
negative contributions for the rate of change
$\frac{1}{2} \od{}{t} \norm{u}_{M+K}^2$ in the positive definite case) implies
$\mat{K} = 0$, at least for Gauß-Legendre and Lobatto-Legendre bases of small
degree. For brevity, these calculations are not repeated here.

\section{Discussion and summary }\label{Chapter_5}

In this work, the general frameworks of CPR methods and SBP SAT operators
are united, leading to a general formulation of semidiscretisations for
conservation laws. The linearly stable schemes of Vincent et al.
\cite{vincent2011newclass, vincent2015extended} are embedded in this framework,
leading to proofs for both conservation and stability in a discrete norm
adapted to the method.

Moreover, the DGSEM introduced by Gassner \cite{gassner2013skew} using a
skew-symmetric formulation of the conservation law is embedded in this framework.
Thus, nonlinear stability results for a special choice of nodal basis
(Lobatto-Legendre) can be proven. These results are extended in this work by
adapting a new formulation of the conservation law, introducing an additional
correction term. In this way, conservation and entropy stability are ensured
for more general SBP CPR methods. Especially, a Gauß-Legendre basis not
including boundary points is an admissible part of these schemes.

First considerations on the extension of the CPR idea to use different
norms in stability proofs in section \ref{sec:extension-of-the-CPR-idea} were not
successful. Therefore, SBP CPR methods should be used with the canonical
correction matrix $\mat{C} = \mat{M}^{-1} \mat{R}^T \mat{B}$ ($\kappa = 0$),
facilitating stability proofs as in sections \ref{Chapter_3} and \ref{Chapter_4}.
Thus, Gauß-Legendre bases seem to preferable, due to their higher order for
quadrature and better approximation qualities, see also Figures
\ref{fig:vincent2011newclass-p-10}(e) and \ref{fig:vincent2011newclass-N-4}(e).
However, extensions to more complex systems seem to be difficult without
relying on boundary nodes. There, Lobatto-Legendre nodes can be advantageous.

Future work will include investigations of different correction matrices
(correction functions in the framework of FR methods) ensuring entropy
stability for nonlinear conservation laws, since the straightforward extension of
the linear case seems not to be possible. Additionally, fully discrete
schemes will be investigated, incorporating the influence of different time
discretisations into the hitherto obtained results. Therefore, special attention
must be paid to the described volume term.

Of course, a straightforward application of SBP CPR schemes in multiple
dimensions using tensor products is possible. Future work will focus on
inherently multi-dimensional SBP CPR bases for both tensor product type
elements and simplices, using the new formulation of these schemes.

\bibliographystyle{abbrv}
\bibliography{literature}

\begin{thebibliography}{10}

\bibitem{allaneau2011connections}
Y.~Allaneau and A.~Jameson.
\newblock Connections between the filtered discontinuous {G}alerkin method and
  the flux reconstruction approach to high order discretizations.
\newblock {\em Computer Methods in Applied Mechanics and Engineering},
  200(49):3628--3636, 2011.

\bibitem{castonguay2012newclass}
P.~Castonguay, P.~E. Vincent, and A.~Jameson.
\newblock A new class of high-order energy stable flux reconstruction schemes
  for triangular elements.
\newblock {\em Journal of Scientific Computing}, 51(1):224--256, 2012.

\bibitem{castonguay2013energy}
P.~Castonguay, D.~Williams, P.~E. Vincent, and A.~Jameson.
\newblock Energy stable flux reconstruction schemes for advection--diffusion
  problems.
\newblock {\em Computer Methods in Applied Mechanics and Engineering},
  267:400--417, 2013.

\bibitem{degrazia2014connections}
D.~De~Grazia, G.~Mengaldo, D.~Moxey, P.~E. Vincent, and S.~Sherwin.
\newblock Connections between the discontinuous galerkin method and high-order
  flux reconstruction schemes.
\newblock {\em International journal for numerical methods in fluids},
  75(12):860--877, 2014.

\bibitem{fernandez2014generalized}
D.~C. D.~R. Fern{\'a}ndez, P.~D. Boom, and D.~W. Zingg.
\newblock A generalized framework for nodal first derivative summation-by-parts
  operators.
\newblock {\em Journal of Computational Physics}, 266:214--239, 2014.

\bibitem{fernandez2014review}
D.~C. D.~R. Fern{\'a}ndez, J.~E. Hicken, and D.~W. Zingg.
\newblock Review of summation-by-parts operators with simultaneous
  approximation terms for the numerical solution of partial differential
  equations.
\newblock {\em Computers {\&} Fluids}, 95:171--196, 2014.

\bibitem{fisher2013discretely}
T.~C. Fisher, M.~H. Carpenter, J.~Nordstr{\"o}m, N.~K. Yamaleev, and
  C.~Swanson.
\newblock Discretely conservative finite-difference formulations for nonlinear
  conservation laws in split form: {T}heory and boundary conditions.
\newblock {\em Journal of Computational Physics}, 234:353--375, 2013.

\bibitem{gassner2013skew}
G.~J. Gassner.
\newblock A skew-symmetric discontinuous {G}alerkin spectral element
  discretization and its relation to {SBP-SAT} finite difference methods.
\newblock {\em SIAM Journal on Scientific Computing}, 35(3):A1233--A1253, 2013.

\bibitem{gassner2014kinetic}
G.~J. Gassner.
\newblock A kinetic energy preserving nodal discontinuous {G}alerkin spectral
  element method.
\newblock {\em International Journal for Numerical Methods in Fluids},
  76(1):28--50, 2014.

\bibitem{gassner2011comparison}
G.~J. Gassner and D.~A. Kopriva.
\newblock A comparison of the dispersion and dissipation errors of {G}auss and
  {G}auss-{L}obatto discontinuous {G}alerkin spectral element methods.
\newblock {\em SIAM Journal on Scientific Computing}, 33(5):2560--2579, 2011.

\bibitem{gassner2016well}
G.~J. Gassner, A.~R. Winters, and D.~A. Kopriva.
\newblock A well balanced and entropy conservative discontinuous {G}alerkin
  spectral element method for the shallow water equations.
\newblock {\em Applied Mathematics and Computation}, 272:291--308, 2016.

\bibitem{hicken2013summation}
J.~E. Hicken and D.~W. Zingg.
\newblock Summation-by-parts operators and high-order quadrature.
\newblock {\em Journal of Computational and Applied Mathematics},
  237(1):111--125, 2013.

\bibitem{huynh2007flux}
H.~Huynh.
\newblock A flux reconstruction approach to high-order schemes including
  discontinuous {G}alerkin methods.
\newblock {\em AIAA paper}, 4079:2007, 2007.

\bibitem{huynh2014high}
H.~Huynh, Z.~J. Wang, and P.~E. Vincent.
\newblock High-order methods for computational fluid dynamics: A brief review
  of compact differential formulations on unstructured grids.
\newblock {\em Computers {\&} Fluids}, 98:209--220, 2014.

\bibitem{jameson2010proof}
A.~Jameson.
\newblock A proof of the stability of the spectral difference method for all
  orders of accuracy.
\newblock {\em Journal of Scientific Computing}, 45(1-3):348--358, 2010.

\bibitem{jameson2012nonlinear}
A.~Jameson, P.~E. Vincent, and P.~Castonguay.
\newblock On the non-linear stability of flux reconstruction schemes.
\newblock {\em Journal of Scientific Computing}, 50(2):434--445, 2012.

\bibitem{kopriva2010quadrature}
D.~A. Kopriva and G.~J. Gassner.
\newblock On the quadrature and weak form choices in collocation type
  discontinuous {G}alerkin spectral element methods.
\newblock {\em Journal of Scientific Computing}, 44(2):136--155, 2010.

\bibitem{kopriva2014energy}
D.~A. Kopriva and G.~J. Gassner.
\newblock An energy stable discontinuous {G}alerkin spectral element
  discretization for variable coefficient advection problems.
\newblock {\em SIAM Journal on Scientific Computing}, 36(4):A2076--A2099, 2014.

\bibitem{svard2014review}
M.~Sv{\"a}rd and J.~Nordstr{\"o}m.
\newblock Review of summation-by-parts schemes for initial--boundary-value
  problems.
\newblock {\em Journal of Computational Physics}, 268:17--38, 2014.

\bibitem{toro2009riemann}
E.~F. Toro.
\newblock {\em Riemann solvers and numerical methods for fluid dynamics: a
  practical introduction}.
\newblock Springer Science \& Business Media, 2009.

\bibitem{vincent2011insights}
P.~E. Vincent, P.~Castonguay, and A.~Jameson.
\newblock Insights from von {N}eumann analysis of high-order flux
  reconstruction schemes.
\newblock {\em Journal of Computational Physics}, 230(22):8134--8154, 2011.

\bibitem{vincent2011newclass}
P.~E. Vincent, P.~Castonguay, and A.~Jameson.
\newblock A new class of high-order energy stable flux reconstruction schemes.
\newblock {\em Journal of Scientific Computing}, 47(1):50--72, 2011.

\bibitem{vincent2015extended}
P.~E. Vincent, A.~M. Farrington, F.~D. Witherden, and A.~Jameson.
\newblock An extended range of stable-symmetric-conservative flux
  reconstruction correction functions.
\newblock {\em Computer Methods in Applied Mechanics and Engineering},
  296:248--272, 2015.

\bibitem{wang2009unifying}
Z.~Wang and H.~Gao.
\newblock A unifying lifting collocation penalty formulation including the
  discontinuous {G}alerkin, spectral volume/difference methods for conservation
  laws on mixed grids.
\newblock {\em Journal of Computational Physics}, 228(21):8161--8186, 2009.

\bibitem{williams2013energy}
D.~Williams, P.~Castonguay, P.~E. Vincent, and A.~Jameson.
\newblock Energy stable flux reconstruction schemes for advection--diffusion
  problems on triangles.
\newblock {\em Journal of Computational Physics}, 250:53--76, 2013.

\bibitem{witherden2014analysis}
F.~D. Witherden and P.~E. Vincent.
\newblock An analysis of solution point coordinates for flux reconstruction
  schemes on triangular elements.
\newblock {\em Journal of Scientific Computing}, 61(2):398--423, 2014.

\bibitem{witherden2015identification}
F.~D. Witherden and P.~E. Vincent.
\newblock On the identification of symmetric quadrature rules for finite
  element methods.
\newblock {\em Computers {\&} Mathematics with Applications},
  69(10):1232--1241, 2015.

\bibitem{yu2013connection}
M.~Yu and Z.~Wang.
\newblock On the connection between the correction and weighting functions in
  the correction procedure via reconstruction method.
\newblock {\em Journal of Scientific Computing}, 54(1):227--244, 2013.

\end{thebibliography}

\end{document}